\documentclass[10pt]{amsart}
\usepackage{tikz-cd}
\usepackage{amssymb}
\usepackage{mathrsfs}
\usepackage[shortlabels]{enumitem}
\usepackage[colorlinks,linkcolor=black,citecolor=black,urlcolor=black]{hyperref}
\usepackage[color = blue!20, bordercolor = black, textsize = tiny]{todonotes}
\usepackage[capitalise]{cleveref}

\numberwithin{equation}{section}
\newtheorem{thm}[subsection]{Theorem}
\newtheorem{cor}[subsection]{Corollary}
\newtheorem{lem}[subsection]{Lemma}
\newtheorem{prop}[subsection]{Proposition}
\newtheorem{thmalpha}{Theorem}

\theoremstyle{definition}
\newtheorem{df}[subsection]{Definition}
\newtheorem{rmk}[subsection]{Remark}
\newtheorem{exm}[subsection]{Example}
\newtheorem{const}[subsection]{Construction}
\newtheorem{quest}[subsection]{Question}
\newtheorem{conj}[subsection]{Conjecture}

\newcommand{\A}{\mathbb{A}}

\newcommand{\C}{\mathbb{C}}

\newcommand{\E}{\mathbb{E}}
\newcommand{\F}{\mathbb{F}}
\newcommand{\G}{\mathbb{G}}

\renewcommand{\L}{\mathbb{L}}

\newcommand{\N}{\mathbb{N}}
\renewcommand{\P}{\mathbb{P}}
\newcommand{\Q}{\mathbb{Q}}

\newcommand{\Sph}{\mathbb{S}}
\newcommand{\T}{\mathbb{T}}

\newcommand{\Z}{\mathbb{Z}}

\newcommand{\cA}{\mathcal{A}}
\newcommand{\cB}{\mathcal{B}}
\newcommand{\cC}{\mathcal{C}}
\newcommand{\cD}{\mathcal{D}}
\newcommand{\cE}{\mathcal{E}}
\newcommand{\cF}{\mathcal{F}}
\newcommand{\cG}{\mathcal{G}}

\newcommand{\cM}{\mathcal{M}}

\newcommand{\cO}{\mathcal{O}}
\newcommand{\cP}{\mathcal{P}}

\newcommand{\cV}{\mathcal{V}}

\newcommand{\cX}{\mathcal{X}}

\newcommand{\rB}{\mathrm{B}}

\newcommand{\rD}{\mathrm{D}}

\newcommand{\rH}{\mathrm{H}}

\newcommand{\rL}{\mathrm{L}}

\newcommand{\rN}{\mathrm{N}}

\DeclareMathOperator{\Hom}{Hom}
\DeclareMathOperator{\Spec}{Spec}

\DeclareMathOperator{\coker}{coker}

\DeclareMathOperator{\rank}{rank}
\newcommand{\colim}{\mathop{\mathrm{colim}}}

\newcommand{\id}{\mathrm{id}}
\newcommand{\ul}{\underline}
\newcommand{\ol}{\overline}

\newcommand{\Cone}{\mathrm{Cone}}

\newcommand{\pt}{\mathrm{pt}}
\newcommand{\lSm}{\mathrm{lSm}}

\newcommand{\lSch}{\mathrm{lSch}}

\newcommand{\Sm}{\mathrm{Sm}}
\newcommand{\SmlSm}{\mathrm{SmlSm}}
\newcommand{\RglRg}{\mathrm{RglRg}}
\newcommand{\lRg}{\mathrm{lRg}}

\newcommand{\sat}{\mathrm{sat}}
\newcommand{\gp}{\mathrm{gp}}

\newcommand{\Adm}{\mathrm{Adm}}
\newcommand{\SAdm}{\mathrm{SAdm}}
\newcommand{\divi}{\mathrm{div}}
\newcommand{\Sing}{\mathrm{Sing}}
\newcommand{\Bl}{\mathrm{Bl}}
\newcommand{\SBl}{\mathrm{SBl}}
\newcommand{\Rg}{\mathrm{Rg}}

\newcommand{\CSing}{\mathrm{CSing}}
\newcommand{\Gys}{\mathrm{Gys}}

\newcommand{\nor}{\mathrm{nor}}
\newcommand{\Gmlog}{\mathbb{\G}_m^{\log}}
\newcommand{\cSm}{\mathrm{cSm}}
\newcommand{\Sdiv}{\mathrm{Sdiv}}

\newcommand{\Fil}{\mathrm{Fil}}
\newcommand{\fil}{\mathrm{f}}
\newcommand{\gr}{\mathrm{gr}}
\DeclareMathOperator{\Fun}{Fun}
\DeclareMathOperator{\fib}{fib}

\newcommand{\Sh}{\mathrm{Sh}}
\newcommand{\Sp}{\mathrm{Sp}}
\newcommand{\Spc}{\mathrm{Spc}}

\newcommand{\Sch}{\mathrm{Sch}}

\newcommand{\ver}{\mathrm{ver}}
\newcommand{\CAlg}{\mathrm{CAlg}}
\newcommand{\bDelta}{\mathbf{\Delta}}

\newcommand{\op}{\mathrm{op}}

\newcommand{\Ring}{\mathrm{Ring}}
\newcommand{\Mon}{\mathrm{Mon}}
\newcommand{\PSh}{\mathrm{PSh}}
\newcommand{\Ab}{\mathrm{Ab}}
\newcommand{\Cube}{\mathrm{Cube}}
\newcommand{\ECube}{\mathrm{ECube}}

\newcommand{\CycSp}{\mathrm{CycSp}}
\newcommand{\SmAdm}{\mathrm{SmAdm}}
\newcommand{\Log}{\mathrm{\ell og}}
\newcommand{\CLog}{\mathrm{\ell og}^\mathrm{cube}}
\newcommand{\logCLog}{\widetilde{\mathrm{\ell og}}^\mathrm{cube}}
\newcommand{\cube}{\mathrm{cube}}

\newcommand{\logSH}{\mathrm{logSH}}

\newcommand{\SH}{\mathrm{SH}}

\newcommand{\DM}{\mathrm{DM}}

\newcommand{\LogSh}{\mathrm{Sh}_{\ell \mathrm{og}}}
\newcommand{\MS}{\mathrm{MS}}

\newcommand{\CH}{\mathrm{CH}}

\newcommand{\Kth}{\mathrm{K}}
\newcommand{\HH}{\mathrm{HH}}

\newcommand{\THH}{\mathrm{THH}}
\newcommand{\TC}{\mathrm{TC}}

\newcommand{\BMS}{\mathrm{BMS}}
\newcommand{\HKR}{\mathrm{HKR}}
\newcommand{\Tr}{\mathrm{Tr}}

\newcommand{\logTr}{\mathrm{logTr}}

\newcommand{\Zar}{\mathrm{Zar}}
\newcommand{\seta}{\mathrm{s\acute{e}t}}

\newcommand{\sNis}{\mathrm{sNis}}
\newcommand{\syn}{\mathrm{syn}}
\newcommand{\Nis}{\mathrm{Nis}}

\newcommand{\mot}{\mathrm{mot}}
\newcommand{\lmot}{\mathrm{lmot}}

\newcommand{\logTHH}{\mathrm{logTHH}}

\newcommand{\logTC}{\mathrm{logTC}}

\newcommand{\blogTC}{\mathrm{log}\mathbf{TC}}

\newcommand{\Sna}{\mathrm{Sna}}
\newcommand{\logHH}{\mathrm{logHH}}

\newcommand{\Th}{\mathrm{Th}}
\newcommand{\unit}{\mathbf{1}}

\newcommand{\slice}{\mathrm{s}}
\newcommand{\MZ}{\mathbf{M}\mathbb{Z}}
\newcommand{\KGL}{\mathbf{KGL}}
\newcommand{\logKGL}{\mathrm{log}\mathbf{KGL}}

\newcommand{\MGL}{\mathbf{MGL}}

\usepackage[bbgreekl]{mathbbol}
\usepackage{relsize}

\DeclareSymbolFontAlphabet{\mathbb}{AMSb} 
\DeclareSymbolFontAlphabet{\mathbbl}{bbold}

\renewcommand{\cong}{\simeq}

\begin{document}
\title{Construction of logarithmic cohomology theories I}
\author{Doosung Park}
\address{Department of Mathematics and Informatics, University of Wuppertal, Germany}
\email{dpark@uni-wuppertal.de}
\date{\today}
\begin{abstract}
We propose a method for constructing cohomology theories of logarithmic
schemes with strict normal crossing boundaries by employing techniques from logarithmic motivic homotopy theory over $\mathbb{F}_1$. This method
recovers the K-theory of the open complement of a strict normal
crossing divisor from the K-theory of schemes as well as logarithmic
topological Hochschild homology from the topological Hochschild
homology of schemes. In our applications, we establish that the
K-theory of non-regular schemes is representable in the logarithmic
motivic homotopy category, and we introduce the logarithmic cyclotomic
trace for the regular log regular case.
\end{abstract}
\maketitle

\section{Introduction}

A log scheme may be conceptualized as a ``scheme with boundary.'' Extending a cohomology theory from schemes to log
schemes presents a natural challenge. A cohomology theory formulated
for log schemes can provide valuable insights into both the cohomology
theories of schemes and their extensions. 

\subsection{Logarithmic extensions of non-\texorpdfstring{$\A^1$}{A1}-invariant cohomology theories}

An early example includes the extension of Hodge cohomology
to smooth varieties $X$ over a field $k$ with a normal crossing
boundary $D$, utilizing the log differential forms
$\Omega_{X/k}^q(\log D)$. For a smooth variety $U$ over $\mathbb{C}$,
one can find a proper smooth variety $X$ over $\mathbb{C}$ containing
$U$ as an open subscheme with a normal crossing divisor $D$ as its
complement by Hironaka's resolution of singularities
\cite{Hir}. Grothendieck \cite{zbMATH03234150} proved that there is a natural
quasi-isomorphism
\[
R\Gamma_\mathrm{dR}(U) \simeq
R\Gamma_\mathrm{dR}(U^\mathrm{an}),
\]
which connects algebraic and
analytic de Rham complexes, and the log differential forms $\Omega_{X/k}^q(\log D)$
plays an important intermediate role. Subsequently, Deligne, Faltings, Fontaine,
Illusie, and Kato developed the theory of logarithmic geometry,
where
the logarithmic forms $\Omega_{X/\mathbb{C}}^q(\log D)$ are realized
as differential forms of the associated log scheme.

Another notable example is the extension of topological Hochschild
homology and its related theories to fs log schemes. For a log ring
$(A,M)$, which consists of a ring $A$ and a monoid $M$ equipped with a
map of monoids $M \to A$,
Rognes \cite[Definition
8.11]{zbMATH05610455} defined topological Hochschild homology
$\logTHH(A,M)$. This admits a cyclotomic structure in the sense of \cite[Definition II.1.1]{zbMATH07009201} by \cite[Definition 3]{Ob18},
which leads to the definitions
of negative topological cyclic homology $\logTC^-(A,M)$ and
topological cyclic homology $\logTC(A,M)$.

Bhatt, Morrow,
and Scholze \cite[Theorem 1.12]{BMS19} established motivic filtrations
on the $p$-completed ones
\[
\THH(A;\mathbb{Z}_p),
\text{ }
\TC^-(A;\mathbb{Z}_p),
\text{ }
\TC(A;\mathbb{Z}_p)
\]
for quasi-syntomic
rings $A$.
The $0$th graded pieces of $\TC^-(A;\mathbb{Z}_p)$ and
$\TC(A;\mathbb{Z}_p)$ can be identified with Nygaard-completed
prismatic cohomology and syntomic cohomology respectively. Binda,
Lundemo, Park, and {\O}stv{\ae}r \cite[Theorem 1.3]{BLPO2} extended
these filtrations to
\[
\logTHH((A,M);\mathbb{Z}_p),
\text{ }
\logTC^-((A,M);\mathbb{Z}_p),
\text{ }
\logTC((A,M);\mathbb{Z}_p)
\]
for log quasi-syntomic rings $(A,M)$.
Ultimately, many cohomology theories
of schemes admit logarithmic extensions.

One application of these
logarithmic extensions is the construction of Gysin sequences.
For a separated finite dimensional noetherian scheme $S$
(which we denote as $S \in \Sch$),
let $\Sm/S$ denote the category of smooth $S$-schemes in $\Sch$.
Define $\SmlSm/S$ as the category of ``smooth log smooth fine saturated (fs for short) log schemes,'' that is, log smooth fs log schemes over $S$ whose underlying schemes are also smooth over $S$.
For $X \in \SmlSm/S$, let
$\partial X$ denote the boundary of $X$, which corresponds to the
strict normal crossing divisor on the underlying scheme $\ul{X}$.
The complement $X - \partial X$ is the largest open subscheme of $X$ with the trivial log structure.
If $Y \in \Sm/S$ and $D$ is its strict normal crossing divisor, we can associate $(Y,D) \in \SmlSm/S$ whose underlying scheme is $Y$ and boundary is $D$ using the
Deligne-Faltings log structure \cite[Definition III.1.7.1]{Ogu}.
There exists a strict Nisnevich topology on $\SmlSm/S$ \cite[Definition
3.1.4]{BPO}, which is abbreviated as $\sNis$.
If a strict Nisnevich sheaf of spectra $\E$ on $\SmlSm/S$ satisfies
$(\P^n,\P^{n-1})$-invariance for all integers $n$,
then $\E$ is an object of the logarithmic motivic homotopy category $\logSH_{S^1}(S)$ as noted in \cite[Remark 2.4.14]{BPO2}.
According to \cite[Theorem 7.5.4]{BPO},
we have the Gysin fiber sequence
\[
\E(\Th(\rN_Z X)) \to
\E(X) \to \E(\Bl_Z X)
\]
for every closed immersion
$Z \to X$ in $\Sm/S$, where $E$ is the exceptional divisor on the blow-up $\Bl_Z X$, and $\Th(\rN_Z X)$ is the logarithmic Thom space of the normal bundle $\rN_Z X$ in the sense of \cite[Definition 7.4.3]{BPO}.
Binda, Lundemo, Merici, and Park \cite{BLMP}
utilized this method to construct the Gysin sequences for Nygaard-completed prismatic cohomology and syntomic cohomology.

So far, we have discussed cohomology theories of fs log schemes that require additional data for their definitions.
All of
these examples are also non-$\A^1$-invariant.

\subsection{Logarithmic extensions of \texorpdfstring{$\A^1$}{A1}-invariant cohomology theories}
On the other hand,
for an $\A^1$-invariant Nisnevich sheaf of spectra $\E$ such as K-theory on $\Sm/S$ with regular noetherian $S$,
we can simply define
\[
\E(X):=\E(X-\partial X)
\]
for $X\in \SmlSm/S$.
This extension is automatically obtained and requires no additional data.

Thus, there is a
discrepancy between the extensions of
$\A^1$-invariant and non-$\A^1$-invariant cohomology theories to fs log
schemes.
This discrepancy presents challenges when attempting to
extend a morphism from an $\A^1$-invariant cohomology theory to a non-$\A^1$-invariant cohomology theory as the methods for extending these theories are not comparable.
The primary example under consideration is the cyclotomic trace
\[
\Tr\colon \Kth(S)\to \TC(S)
\]
for all $S\in \Sch$.
When $X\in \SmlSm/S$,
the cyclotomic trace only gives $\Tr\colon \Kth(X-\partial X)\to \TC(X-\partial X)$.
Rognes \cite{RogTHH} conjectured that the \emph{log cyclotomic trace}
\[
\Kth(X - \partial X)
\to
\logTC(X)
\]
exists, making the following diagram
\[
\begin{tikzcd}
& \logTC(X) \ar[d]
\\
\Kth(X - \partial X) \ar[r,
"\Tr"] \ar[ru, dashed]
&
\TC(X - \partial X)
\end{tikzcd}
\]
commutative.
In this situation, $\logTC(X)$ serves as a \emph{better approximation
of $\Kth(X-\partial X)$} than $\TC(X - \partial X)$.

An obvious example of $X$ where $\logTC(X)$ provides a better approximation
than $\TC(X-\partial X)$ is $X := \square_k$ with a field $k$.
In this case, we have $X-\partial X \cong \A^1_k$.
Due to the $\A^1$-invariance of K-theory for regular noetherian schemes, we have $\Kth(\A^1_k) \cong \Kth(k)$.
However, $\TC(\A_k^1)$ is much more complicated than $\TC(k)$.
On the other hand, $\logTC$ is $\square$-invariant, which gives us $\TC(\square_k)\cong \TC(k)$.

\subsection{Suslin functors in non-\texorpdfstring{$\A^1$}{A1}-invariant motivic homotopy theories}

The purpose of this paper is to offer an alternative method for extending cohomology theories from $\Sm/S$ to $\SmlSm/S$ with $S\in \Sch$ ensuring that we obtain $\Kth(X-\partial X)$ and $\logTC(X)$ for $X\in \SmlSm/S$ when we apply this method to $\Kth$ and $\TC$ for $\Sm/S$.

The initial motivation for this method comes from the Suslin functor for the modulus interval 
$\ol{\Box}$ as proposed by Kahn-Miyazaki-Saito-Yamazaki \cite[Definition 5.2.5]{MR4442406}.
According to \cite[Corollary 6.3.8]{MR4442406}, the Suslin functor makes $\Z(n)[2n]$ for every integer $n$ into the presheaf
\[
\cX=(\ol{\cX},\cX^\infty)\mapsto R\Gamma_\mot(\ol{\cX}-\cX^\infty,\Z(n)[2n])
\]
for modulus pairs $(\ol{\cX},\cX^\infty)$ over a perfect field $k$,
where $R\Gamma_\mot(-,\Z(n))$ denotes the weight $n$ motivic cohomology complex.
We allow non-smooth $\ol{\cX}$ for the modulus pair $(\ol{\cX},\cX^\infty)$ to ensure that the proof does not depend on resolution of singularities.
However, this introduces difficulties in representing cohomology theories.
For instance, we currently know that Hodge cohomology is representable in the modulus setting only in characteristic $0$ primarily due to the requirement for resolution of singularities as shown in \cite[Theorem 1.3]{2306.06864}.

An analogous result holds in the logarithmic case as well. Binda, Park, and {\O}stv{\ae}r \cite[Theorem 8.2.11]{BPO} showed that the Suslin functor for the logarithmic interval 
 $\square:=(\P^1,\infty)$ transforms $\Z(n)[2n]$ for every integer $n$ into the presheaf
\[
X
\mapsto
R\Gamma_\mot(X-\partial X,\Z(n)[2n])
\]
for $X\in \SmlSm/k$ over a perfect field $k$ admitting resolution of singularities.
In contrast to the modulus case, Hodge cohomology is representable in the $\infty$-category of effective logarithmic motives $\mathrm{logDM}^\mathrm{eff}(k)$ without relying on resolution of singularities as shown in \cite[Theorem 9.7.1]{BPO}.

\subsection{Logarithmic motivic homotopy theory over \texorpdfstring{$\F_1$}{F1}}
The fundamental problem of the Suslin functors for the intervals $\ol{\Box}$ and $\square$ is that there are too many blow-ups to consider,
which makes the treatment of non-smooth schemes
or the resolution of singularities almost unavoidable.

Our proposal is to work with the field with one element $\F_1$ as a base,
which would enable us to manage much simpler blow-ups.
There is no general consensus on the definition of $\F_1$ and the category of $\F_1$-schemes, see e.g.\ \cite[\S 2]{zbMATH06894815} for various approaches.
However, the definition we are using for an $\F_1$-scheme is a monoid scheme developed by Kato \cite{MR1296725} and Deitmar \cite{MR2176588} because we want to work with a minimized category of $\F_1$-schemes.
In this context, our notation $\Sm/\F_1$ refers to the category of smooth toric monoid schemes,
whose objects are associated with smooth fans, but whose morphisms are more general than morphisms of fans.
Arndt's thesis \cite[\S 4.4]{Arndt} suggested the construction of $\SH(\F_1)$ using monoid schemes.

To develop logarithmic motivic homotopy theory over $\F_1$,
we introduce log monoid schemes in Appendix \ref{logmonoid} by adding log structures to monoid schemes.
We can also define $\SmlSm/\F_1$ in analogy to $\SmlSm/S$ with $S\in \Sch$.
A morphism $f\colon Y\to X$ in $\SmlSm/\F_1$ is called an admissible blow-up if $\ul{f}$ is proper birational and $f-\partial f\colon Y-\partial Y\to X-\partial X$ is an isomorphism.
Let $\Adm$ denote the class of admissible blow-ups in $\SmlSm/\F_1$.
In this context,
$\square:=(\P^1,\infty)\in \SmlSm/\F_1$ serves an interval object of $\SmlSm/\F_1[\Adm^{-1}]$, see Proposition \ref{logSH.18}.
In Definitions \ref{logmonoid.19} and \ref{logmonoid.18},
we also have the notions of the Zariski topology and dividing covers in $\SmlSm/\F_1$.
This leads us to define the logarithmic motivic homotopy category $\logSH(\F_1)$ in Definition \ref{logSH.6}.
We provide an alternative model
\[
\logSH(\F_1)\simeq \Sp_{\P^1}(\Sh_\Zar(\SmlSm/\F_1,\Sp))[(\P^\bullet,\P^{\bullet-1})^{-1}],
\]
in Theorem \ref{logSH.8},
where $(\P^\bullet,\P^{\bullet-1})$ denotes the set of projections $X\times (\P^n,\P^{n-1})\to X$ for $X\in \SmlSm/\F_1$ and integer $n\geq 1$.

We also consider the Suslin functor
\[
\Sing^\square \colon \PSh(\SmlSm/\F_1[\Adm^{-1}],\Sp)
\to
\PSh(\SmlSm/\F_1[\Adm^{-1}],\Sp).
\]
For $\cF\in \PSh(\SmlSm/\F_1,\Sp)$ and $X\in \SmlSm/\F_1$,
we obtain
\[
\Sing^\square L_\Adm \cF(X)
:=
\colim_{n\in \bDelta^\op}\colim_{Y\in \Adm_{X\times \square^n}^\op}\cF(Y)
\]
in \eqref{logSH.15.1}.
Here, $L_\Adm\colon \PSh(\SmlSm/\F_1,\Sp)\to \PSh(\SmlSm/\F_1[\Adm^{-1}],\Sp)$ denotes the localization functor,
and $\Adm_{X\times \square^n}$ denotes the category of admissible blow-ups of $X\times \square^n$ in $\SmlSm/\F_1$.

\subsection{Logarithm functor}

To utilize $\Sing^\square$ defined for log monoid schemes in the original world of log schemes,
consider $S\in \Sch$ and $X\in \SmlSm/S$ having a chart $\A_P$ for a sharp fs monoid $P$.
Such a sharp chart $P$ locally exists by \cite[Proposition II.2.3.7]{Ogu}.
We can then consider the category of admissible blow-ups of $\A_P\times \square^n\in \SmlSm/\F_1$,
and we can pull back these admissible blow-ups to $X$.
Inspired by this process,
we construct the \emph{logarithm functor}
\[
\Log\colon \Sh_\Nis(\Sm/S,\Sp)
\to
\Sh_\sNis(\SmlSm/S,\Sp)
\]
in Definition \ref{boundary.53}.
It satisfies
\[
\Log \cF(X)
\cong
\colim_{n\in \bDelta^\op}
\colim_{Y_0\in \Adm_{\A_P\times \square^n}^\op}
\cF(\ul{X}\times_{\ul{\A_P}}\ul{Y_0})
\]
for $X\in \SmlSm/S$ with a sharp chart $P$, see \eqref{boundary.52.2}.

A Nisnevich sheaf of spectra $\cF$ on $\Sm/S$ is called \emph{logarithmic} if
\[
\cF(X)\cong \Log \cF(X\times \square^n)
\]
for all $X\in \Sm/S$ and integers $n\geq 0$.
In this situation,
we show $\Log \cF\in \logSH_{S^1}(S)$ in Proposition \ref{boundary.47},
indicating that $\Log \cF$ is a natural extension of $\cF$ to $\SmlSm/S$ that retains motivic properties.

Our first main theorem asserts that various cohomology theories are logarithmic:

\begin{thmalpha}[See Theorems \ref{boundarify.31}, \ref{boundarify.6}, \ref{boundarify.21}, and \ref{boundarify.23}]
\label{thm.A}
Let $S\in \Sch$ and $q\in \N$. Then the Nisnevich sheaves of spectra
\[
\Kth,
\text{ }
\THH,
\text{ }
\TC,
\text{ }
R\Gamma_\Zar(-,\Omega_{-/S}^q)
\]
on $\Sm/S$ are logarithmic.
Moreover, if $S=\Spec(k)$ with a perfect field $k$,
then the Nisnevich sheaf of weight $q$ motivic cohomology complexes $R\Gamma_\mot(-,\Z(q))$ on $\Sm/k$ is also logarithmic.
\end{thmalpha}

The proof of Theorem \ref{thm.A} relies on Theorem \ref{boundarify.9},
whose proof makes extensive use of toric geometry and is postponed to \cite{logSHF2}.

As a consequence,
we recover Rognes' topological Hochschild homology using $\Log$:

\begin{thmalpha}[See Theorem \ref{boundarify.20}]
For every $S\in \Sch$ and $X\in \SmlSm/S$,
we have a natural isomorphism of $\E_\infty$-rings in cyclotomic spectra
\[
\Log \THH(X) \cong \logTHH(X)
\]
and a natural isomorphism of $\E_\infty$-rings
\[
\Log \TC(X) \cong \logTC(X).
\]
\end{thmalpha}

\subsection{Motivic representability of K-theory}

For $S\in \Rg$,
we have $\logKGL\in\logSH(S)$ in \cite[Definition 6.5.6]{BPO2},
which satisfies
\[
\Kth(X-\partial X)
\cong
\hom_{\logSH(S)}(\Sigma^{2n,n}\Sigma_{\P^1}^\infty X_+,\logKGL)
\]
for $X\in \SmlSm/S$ and integer $n$ as stated in \cite[Theorem 6.5.7]{BPO2}.
For general non-regular $S\in \Sch$,
whatever $\logKGL$ is,
it is not the case that such an isomorphism always holds.
Indeed, if $X=\square_S$,
then we have the induced commutative square
\[
\begin{tikzcd}
\Kth(S)\ar[r]\ar[d,"\not\cong"']&
\hom_{\logSH(S)}
(\Sigma_{\P^1}^\infty S_+,\logKGL)\ar[d,"\cong"]
\\
\Kth(\A_S^1)\ar[r]&
\hom_{\logSH(S)}(\Sigma_{\P^1}^\infty (\square_S)_+,\logKGL).
\end{tikzcd}
\]
The left vertical morphism is not an isomorphism in general,
but the right vertical morphism is an isomorphism since $\square$ is contractible in $\logSH(S)$.
Hence we cannot have that the two horizontal morphisms are isomorphisms simultaneously.

This is the reason why representing the (non-$\A^1$-invariant) K-theory of schemes in $\logSH(S)$ has been a nontrivial question.
However,
with Theorem \ref{thm.A} in hand,
we can simply apply $\Log$ to the $\P^1$-spectrum $\KGL:=(\Kth,\Kth,\cdots)$ to obtain the $\P^1$-spectrum
\[
\Log \KGL
\cong
(\Log \Kth,\Log \Kth,\ldots).
\]
This represents K-theory in the following sense:

\begin{thmalpha}[See Theorem \ref{boundarify.19}]
\label{thm.B}
For every $S\in \Sch$, $X\in \SmlSm/S$, and integer $n$,
there exists a natural isomorphism
\[
\Log
\Kth(X)
\cong
\hom_{\logSH(S)}(\Sigma^{2n,n}\Sigma_{\P^1}^\infty X_+,\Log \KGL).
\]
\end{thmalpha}
Note that we have $\Kth(X)\cong \Log \Kth(X)$ for $X\in \Sm/S$ by Theorem \ref{thm.A}.

Together with a motivic version of the Snaith theorem due to Annala-Iwasa \cite[Theorem 5.3.3]{AI},
we have a log motivic version of the Snaith theorem as follows:

\begin{thmalpha}[See Theorem \ref{boundarify.34}]
For $S\in \Sch$,
there exists a natural isomorphism
\[
\Sigma_{\P^1}^\infty (\P^\infty)_+ [\beta^{-1}]
\cong
\Log \KGL
\]
in $\logSH(S)$,
see \textup{Construction \ref{boundarify.36}} for $\beta$.
\end{thmalpha}

\subsection{Log cyclotomic trace}
Let $\Rg$ denote the category of regular schemes in $\Sch$.
Let $\RglRg$ be the category of ``regular log regular fs log schemes,''
i.e., log regular fs log schemes whose underlying schemes are also regular.
Given  $Y\in \Rg$ with a strict normal crossing divisor $D$,
we can associate $(Y,D)\in \RglRg$ whose underlying scheme is $Y$ and whose boundary is $D$ using the Deligne-Faltings log structure.
We construct the \emph{logarithm functor}
\[
\Log \colon \Sh_\Nis(\Rg,\Sp)
\to
\Sh_\sNis(\RglRg,\Sp)
\]
in Definition \ref{boundary.53}.
Combining Theorem \ref{thm.A} and a cdh descent argument in Theorem \ref{boundary.6},
we obtain the following result:

\begin{thmalpha}[See Theorem \ref{boundarify.29}]
For every $X\in \RglRg$,
we have a natural isomorphism of $\E_\infty$-rings
\[
\Log \Kth(X)\cong \Kth(X-\partial X).
\]
\end{thmalpha}

Another application of Theorem \ref{thm.A} is the construction of the log cyclotomic trace whose existence was conjectured by Rognes \cite{RogTHH}:

\begin{thmalpha}[See Theorem \ref{Trace.2}]
\label{thm.C}
Let $X\in \RglRg$ or $\SmlSm/S$ with $S\in \Sch$.
Then there exists a natural morphism of $\E_\infty$-rings
\[
\logTr\colon
\Log \Kth(X)
\to
\logTC(X)
\]
that coincides with the usual cyclotomic trace when $X$ has the trivial log structure.
\end{thmalpha}

By combining the log cyclotomic trace with the comparison between the \'etale K-theory and TC due to Clausen-Mathew-Morrow in \cite[Theorem A]{MR4280864},
we obtain the following result.
Additional results obtained by similar methods can
be found in Theorems \ref{trace.16}, \ref{trace.9}, and \ref{trace.10}.

\begin{thmalpha}[See Theorem \ref{trace.8}]
\label{thm.D}
Let $R$ be a strictly henselian regular local ring with residue characteristic $p>0$,
and let $x_1,\ldots,x_n$ be a regular sequence in $R$ such that the divisor $(x_1)+\cdots +(x_n)$ is strict normal crossing.
Then there is a natural isomorphism of $\E_\infty$-rings
\[
\Kth(R[1/x_1,\ldots,1/x_n];\Z_p)
\cong
\logTC((R,(x_1)+\cdots+(x_n));\Z_p).
\]
\end{thmalpha}

The motivic filtrations on $\TC(A;\Z_p)$ for quasi-syntomic rings $A$ are useful for the computation of K-theory.
This includes the recent computation of the homotopy groups of $\Kth(\Z/p^n)$ due to Antieau-Krause-Nikolaus \cite{2405.04329}.
We hope that the log cyclotomic trace contributes to the computations of the K-theory of schemes when combined with the motivic filtrations on $\logTC((A,M);\Z_p)$ for log quasi-syntomic rings $(A,M)$.
A future project is computing the log syntomic cohomology of $(\cO_K,\langle \pi \rangle)$, which would enhance our understanding of $\Kth(K;\Z_p)$.
Its homotopy groups were previously computed by Rognes-Weibel \cite{MR1697095} for $p=2$ and Hesselholt-Madsen \cite{MR1998478} for $p>2$.

Note that Lundemo's forthcoming paper \cite{Lun} also shows that for $X\in \RglRg$,
there is a log cyclotomic trace $\Kth(X-\partial X) \to \logTC(X)$ without claiming that it is an $\E_\infty$-ring morphism at least for the case that $\partial X$ has $1$ irreducible component.
This should be enough for the above project computing the homotopy groups $\pi_*\Kth(K;\Z_p)$ using the log syntomic cohomology of $(\cO_K,\langle \pi \rangle)$.

Nevertheless, our result that the log cyclotomic trace is an $\E_\infty$-ring morphism is good to know (e.g., to understand the ring structure on $\pi_*\Kth$ via the log cyclotomic trace) and has an application to logarithmic motivic homotopy theory.
We plan to adapt the motivic Atiyah duality due to Annala-Iwasa-Hoyois \cite{2403.01561} to the logarithmic setting in future work.
They used the cyclotomic trace to show that the motivic topological cyclic homology spectrum is a $\unit_{\A^1}$-module for $S\in \Rg$ \cite[Theorem 1.5]{2403.01561}.
Thanks to the fact that the log cyclotomic trace is an $\E_\infty$-ring morphism,
we can regard $\blogTC\in \logSH(S)$ in \cite[Definition 8.5.3]{BPO2} as an $\omega^*\KGL$-algebra and hence an $\omega^*\unit$-algebra,
see \cite[Construction 4.0.8]{BPO2} for $\omega^*\colon \SH(S)\to \logSH(S)$.
In particular,
$\blogTC$ is an $\omega^*\unit$-module.

This would allows us to recover \cite[Theorem 1.7]{2403.01561}, which shows that the log crystalline cohomology is independent of the choice of compactifications,
using logarithmic Gysin sequences \cite[Theorem 7.5.4]{BPO} instead of Tang's Gysin sequences \cite{Tang}.

\subsection*{Notation and conventions}

We refer to \cite{Ogu} for log geometry and \cite{MR495499}, \cite{Fulton:1436535}, and \cite{CLStoric} for toric geometry.

Every log scheme in the main body is equipped with a Zariski log structure unless otherwise stated.
We employ the following notation throughout the paper.

\begin{tabular}{l|l}
$\Sch$ & category of finite dimensional noetherian separated schemes
\\
$\lSch$ & category of finite dimensional noetherian separated fs log schemes
\\
$\Rg$ & full subcategory of $\Sch$ spanned by regular schemes
\\
$\lRg$ & full subcategory of $\lSch$ spanned by log regular schemes
\\
$\Sm$ & class of smooth morphisms in $\Sch$
\\
$\lSm$ & class of log smooth morphisms in $\lSch$
\\
$\bDelta$ & Simplex category
\\
$\Spc$ & $\infty$-category of spaces
\\
$\Spc_*$ & $\infty$-category of pointed spaces
\\
$\Sp$ & $\infty$-category of spectra
\\
$\Hom_{\cC}$ & hom space in $\cC$
\\
$\hom_{\cC}$ & hom spectra in a stable $\infty$-category $\cC$
\\
$\rD(A)$ & derived $\infty$-category of $A$-modules
\\
$\PSh(\cC,\cV)$ & $\infty$-category of presheaves with values in $\cV$
\\
$\Sh_t(\cC,\cV)$ & $\infty$-category of $t$-sheaves with values in $\cV$
\end{tabular}

\

We regard $\A^1$, $\G_m$, and $\P^n$ for integers $n\geq 1$ as monoid schemes, see Definition \ref{msch.17}.
When $X$ is a scheme and $Y$ is a monoid scheme, 
we can take the product $X\times Y$ in the unified category of schemes and monoid schemes, see Construction \ref{msch.6}.
We have a similar convention when $X$ is an fs log scheme and $Y$ is an fs log monoid scheme.

For a lattice $L$ and its elements $x_1,\ldots,x_r$,
let
\[
\Cone(x_1,\ldots,x_r)
:=
\{x\in L:x=a_1x_1+\cdots+a_rx_r\text{ for some $a_1,\ldots,a_r\in \Q_{\geq 0}$}\}
\]
be the cone in $L$ generated by $x_1,\ldots,x_r$.

\subsection*{Acknowledgement}

This research was conducted in the framework of the DFG-funded research training group GRK 2240: \emph{Algebro-Geometric Methods in Algebra, Arithmetic and Topology}.
We thank Federico Binda, Jens Hornbostel, Tommy Lundemo, Alberto Merici, and Paul Arne {\O}stv{\ae}r for helpful conversations and comments on the subject of this paper.
We thank Jens Hornbostel and Paul Arne {\O}stv{\ae}r for careful reading and constructive suggestions on the draft of this manuscript.
We also thank Tommy Lundemo for sharing the draft of \cite{Lun}.

\section{Logarithmic motivic homotopy theory over \texorpdfstring{$\F_1$}{F1}}

Motivic homotopy theory over a scheme $S$ is built out of $\Sm/S$. We cannot classify smooth schemes over $S$ Zariski locally unlike algebraic topology, where a manifold is locally an Euclidean space.
This makes motivic homotopy theory over $S$ profound but more complicated than algebraic topology.

To avoid complicated algebraic geometry but keep Tate twists for the purpose of deviating from algebraic topology, we work with smooth fans.
See Definition \ref{msch.22} for the formal definition of $\Sm/\F_1$.
Smooth fans are much easier to deal with than smooth schemes because smooth fans are completely determined by combinatorial data.
For example, we have toric resolution of singularities  \cite[\S 2.6]{Fulton:1436535}, but resolution of singularities over a field of characteristic $p>0$ or mixed characteristic is unknown.

The stable motivic homotopy category $\SH(\F_1)$ over $\F_1$ is suggested by Arndt \cite[\S 4.4]{Arndt}.
In this section,
we define $\logSH(\F_1)$ for the purpose of exploiting the simplicity of $\Sm/\F_1$ in logarithmic motivic homotopy theory.
We refer to Appendices \ref{msch} and \ref{logmonoid} for monoid schemes and log monoid schemes,
which we use throughout this paper.

\

For a presentable symmetric monoidal $\infty$-category $\cC$ and its object $T$,
let $\Sp_T(\cC)$ be the formal inversion of $X$ in $\cC$ in the sense of \cite[Definition 2.6]{zbMATH06374152}.
By \cite[Proposition 2.9]{zbMATH06374152},
if $\cD$ is another presentable symmetric monoidal $\infty$-category,
we have the induced functor
\begin{equation}
\label{logSH.1.1}
\Fun^{\rL,\otimes}(\Sp_T(\cC),\cD)
\to
\Fun^{\rL,\otimes}(\cC,\cD),
\end{equation}
where $\Fun^{\rL,\otimes}$ denotes the $\infty$-category of colimit preserving symmetric monoidal functors.
Furthermore,
\eqref{logSH.1.1} is fully faithful and its essential image is spanned by the functors $F\colon \cC\to \cD$ such that $F(T)\in \cD$ is invertible.

\begin{df}[cf.\ {\cite[\S 4.4]{Arndt}}]
\label{logSH.1}
Imitating \cite{MV},
we define the presentably symmetric monoidal $\infty$-categories
\begin{gather*}
\rH(\F_1)
:=
\Sh_\Zar(\Sm/\F_1,\Spc)[(\A^1)^{-1}],
\\
\rH_*(\F_1)
:=
\Sh_\Zar(\Sm/\F_1,\Spc_*)[(\A^1)^{-1}],
\\
\SH_{S^1}(\F_1)
:=
\Sh_\Zar(\Sm/\F_1,\Sp)[(\A^1)^{-1}],
\\
\SH(\F_1)
:=
\Sp_{\P^1}(\SH_{S^1}(\F_1)).
\end{gather*}
where $\A^1$ in the formulation denotes the set of projections $X\times \A^1\to X$ for $X\in \Sm/\F_1$,
and $\Zar$ denotes the Zariski topology.
\end{df}

\begin{rmk}
\label{logSH.19}
It is unclear whether $\P^1$ is a symmetric object of $\SH_{S^1}(\F_1)$ in the sense of \cite[Definition 2.16]{zbMATH06374152} or not. Because this is a condition in \cite[Corollary 2.22]{zbMATH06374152}, it is unknown and probably false that there is an equivalence between $\SH(\F_1)$ and the $\infty$-category of the sequential $\P^1$-spectra.
This makes computations like stable motivic homotopy groups of $\F_1$ complicated.
\end{rmk}

\begin{rmk}
As a consequence of \cite[Propositions 3.8, 5.9]{Vcdtop} and Remark \ref{logmonoid.17},
the following conditions are equivalent for a presheaf $\cF$ of spaces on $\Sm/\F_1$:
\begin{enumerate}
\item[(1)] For every Zariski distinguished square $Q$,
$\cF(Q)$ is cocartesian.
\item[(2)]
$\cF$ is a Zariski sheaf, i.e., $\cF$ satisfies \v{C}ech descent for the Zariski topology.
\item[(3)]
$\cF$ is a Zariski hypersheaf, i.e., $\cF$ satisfies hyperdescent for the Zariski topology.
\end{enumerate}
We have the same results for presheaves of pointed spaces and presheaves of spectra.
\end{rmk}

\begin{df}
\label{logSH.2}
A morphism $f\colon X\to S$ in $\lSm/\F_1$ is an \emph{admissible blow-up} if $f$ is proper and the induced morphism $X-\partial X\to S-\partial S$ is an isomorphism.
Observe that such a morphism $f$ is birational.
Let $\Adm$ be the class of admissible blow-ups in $\lSm/\F_1$ or $\SmlSm/\F_1$.
\end{df}

Recall from \cite[Dual of I.2.2]{GZ} that a class of morphisms $\cA$ in a category $\cC$ admits a \emph{calculus of right fractions} if the following three conditions are satisfied:
\begin{itemize}
\item[(i)] $\cA$ contains all isomorphisms and closed under compositions.
\item[(ii)] (Right Ore condition) For morphisms $f\colon X\to S$ in $\cC$ and $g\colon S'\to S$ in $\cA$,
there exists a commutative square
\[
\begin{tikzcd}
X'\ar[d]\ar[r,"g'"]&
X\ar[d,"f"]
\\
S'\ar[r,"g"]&
S
\end{tikzcd}
\]
in $\cC$ such that $g'\in \cA$.
\item[(iii)] (Right cancellability condition) For morphisms $f,g\colon X\rightrightarrows S$ in $\cC$ and $u\colon S\to S'$ in $\cA$ such that $uf=gf$,
there exists $v\colon X'\to X$ in $\cA$ such that $fv=gv$.
\end{itemize}

\begin{prop}
\label{logSH.3}
The class $\Adm$ in $\lSm/\F_1$ (resp.\ $\SmlSm/\F_1$) admits a calculus of right fractions.
\end{prop}
\begin{proof}
The right cancellability condition is a consequence of Propositions \ref{image.6} and \ref{logmonoid.5}.

To check the right Ore condition,
let $f\colon X\to S$ and $g\colon S'\to S$ be morphisms in $\lSm/\F_1$ (resp.\ $\SmlSm/\F_1$) such that $g$ is an admissible blow-up.
We can replace $X$ by any admissible blow-up of $X$.
Hence by Proposition \ref{resolution.2},
we may also assume $X\in \SmlSm/\F_1$.
Consider the normalization $\ul{X''}$ of the closure of $X-\partial X\cong (S'-\partial S')\times_{S-\partial S}(X-\partial X)$ in $\ul{S'}\times_{\ul{S}} \ul{X}$,
see Definition \ref{image.1} for the notion of the closure.
Then Lemma \ref{image.3} shows that $\ul{X''}$ is a toric monoid scheme in the sense of Definition \ref{msch.21}.
By Proposition \ref{logmonoid.1},
there exists a proper birational morphism of monoid schemes $\ul{X'}\to \ul{X''}$ such that $\ul{X'}\in \Sm/\F_1$.
Consider $X'\in \cSm/\F_1$ with the underlying scheme $\ul{X'}$ such that $X'-\partial X':=X'-\partial X'$.
Then we have $X'\in \SmlSm/\F_1$ by applying Proposition \ref{logmonoid.25} to the induced morphism $X'\to X$.
To conclude,
observe that the induced morphism $X'\to X$ is an admissible blow-up.
\end{proof}

We refer to Definition \ref{logmonoid.18} for dividing covers in $\lSm/\F_1$.

\begin{prop}
\label{logSH.21}
The class $\divi$ of dividing covers in $\lSm/\F_1$ (resp.\ $\SmlSm/\F_1$) admits calculus of right fractions.
\end{prop}
\begin{proof}
The right cancellability condition is a consequence of Propositions \ref{image.6} and \ref{logmonoid.5}.

To check the right Ore condition,
let $f\colon X\to S$ and $g\colon S'\to S$ be morphisms of log monoid schemes in $\lSm/\F_1$ (resp.\ $\SmlSm/\F_1$) such that $g$ is a dividing cover.
Then the pullback $g'\colon X\times_S S'\to S'$ of $g$ is a dividing cover in $\lSm/\F_1$ by Proposition \ref{logmonoid.12}.
Furthermore, there exists a dividing cover $X'\to X\times_S S'$ such that $X'\in \SmlSm/\F_1$ by Proposition \ref{resolution.2},
and the composite $X'\to S'$ is a dividing cover by Proposition \ref{logmonoid.11}.
\end{proof}

\begin{prop}
\label{logSH.5}
There is an equivalence of categories
\begin{gather*}
\SmlSm/\F_1[\divi^{-1}]
\simeq
\lSm/\F_1[\divi^{-1}],
\\
\SmlSm/\F_1[\Adm^{-1}]
\simeq
\lSm/\F_1[\Adm^{-1}],
\end{gather*}
\end{prop}
\begin{proof}
This is an immediate consequence of Propositions \ref{resolution.2}, \ref{logSH.3}, and \ref{logSH.21}.
\end{proof}

We refer to \cite[Definition 3.1.4]{BPO} for the strict Nisnevich topology on $\lSch$,
which we abbreviate as $\sNis$.
Recall that for $S\in \Sch$,
$\SmlSm/S$ denotes the full subcategory of $\lSm/S$ spanned by those $X$ such that $\ul{X}$ is smooth over $S$.
For $X\in \SmlSm/S$,
there exists a strict normal crossing divisor $D$ on $\ul{X}$ over $S$ in the sense of \cite[Definition 7.2.1]{BPO} such that $X\cong (\ul{X},D)$ by \cite[Lemma A.5.10]{BPO},
where $(\ul{X},D)$ is equipped with the Deligne-Faltings log structure \cite[Definition III.1.7.1]{Ogu}.

Recall that a dividing cover in $\lSch$ is a surjective proper log \'etale monomorphism.
\begin{prop}
Let $f\colon Y\to X$ be a dividing cover in $\lSm/\F_1$.
Then its pullback $\Spec(\Z)\times Y\to \Spec(\Z)\times X$ is a dividing cover.
\end{prop}
\begin{proof}
We can work locally on $X$.
Hence we may assume that $f$ is the product of $\id\colon S\to S$ and $\T_\Delta\to \T_\Sigma$ for some $S\in \Sm/\F_1$ and subdivision $\Delta\to \Sigma$.
Use \cite[Example 2.3.10]{BPO2} to conclude.
\end{proof}

Recall from \cite{BPO2} that for $S\in \Sch$,
we have the presently symmetric monoidal $\infty$-categories
\begin{gather*}
\logSH_{S^1}(S)
:=
\Sh_\sNis(\SmlSm/S,\Sp)[\square^{-1},\divi^{-1}],
\\
\logSH(S)
:=
\Sp_{\P^1}(\logSH_{S^1}(S)),
\end{gather*}
where $\square$ in the formulation denotes the set of projections $X\times \square\to X$ for $X\in \SmlSm/\F_1$,
and $\divi$ in the formulation denotes the set of dividing covers in $\SmlSm/S$.

\begin{df}
\label{logSH.6}
We define the presentably symmetric monoidal $\infty$-categories
\begin{gather*}
\logSH_{S^1}(\F_1)
:=
\Sh_\Zar(\SmlSm/\F_1,\Sp)[\square^{-1},\divi^{-1}],
\\
\logSH(\F_1)
:=
\Sp_{\P^1}(\logSH_{S^1}(\F_1)),
\end{gather*}
where $\square$ in the formulation denotes the set of projections $X\times \square\to X$ for $X\in \lSm/\F_1$,
and $\divi$ in the formulation denotes the set of dividing covers in $\SmlSm/\F_1$.
\end{df}

\begin{rmk}
\label{logSH.20}
Like Remark \ref{logSH.19},
it is unclear whether $\P^1$ is a symmetric object of $\logSH_{S^1}(\F_1)$ in the sense of \cite[Definition 2.16]{zbMATH06374152} or not.
\end{rmk}

Recall from \cite[Definition 7.2.3]{BPO} (see also \cite[Definition 2.6.1]{BPO2}) that a proper birational morphism $f\colon X'\to X$ in $\SmlSm/S$ with $S\in \Sch$ is an \emph{admissible blow-up along a smooth center} if $X'$ is identified with the blow-up $\Bl_Z X$ in the sense of \cite[Definition 7.4.1]{BPO} for some smooth subscheme $Z$ of $\ul{X}$ such that $Z$ has strict normal crossing with $\partial X$ in the sense of \cite[Definition 7.2.1]{BPO} and $Z\subset \partial X$.

The categories $\SmlSm/\F_1$ and $\SmlSm/S$ with $S\in \Sch$ are equipped with the cartesian monoidal structures $\times$.

\begin{thm}
\label{logSH.12}
Let $\cC$ be a symmetric monoidal stable $\infty$-category,
and let
\[
M\colon \SmlSm/S\to \cC
\]
be a $\square$-invariant symmetric monoidal functor sending strict Nisnevich distinguished squares to cartesian squares,
where $S\in \Sch$.
Then the following conditions are equivalent.
\begin{enumerate}
\item[\textup{(1)}]
$M$ is invariant under dividing covers.
\item[\textup{(2)}]
$M$ is invariant under admissible blow-ups along smooth centers.
\item[\textup{(3)}]
$M$ is $(\P^n,\P^{n-1})$-invariant for every integer $n\geq 1$.
\end{enumerate}
\end{thm}
\begin{proof}
(1)$\Rightarrow$(2)
See \cite[Theorem 7.2.10]{BPO},
which requires \cite[Proposition 7.2.5]{BPO}.

(2)$\Rightarrow$(1) Obvious.

(1)$\Rightarrow$(3) See \cite[Proposition 7.3.1]{BPO},
which requires \cite[Theorem 7.2.10]{BPO}.

(3)$\Rightarrow$(1)
See \cite[Theorem 7.7.4]{BPO},
which requires \cite[Proposition 7.2.5, Lemma 7.7.1]{BPO}.
\end{proof}

\begin{df}
\label{logSH.26}
A morphism $f\colon X\to S$ in $\SmlSm/\F_1$ is an \emph{admissible blow-up along a smooth center} if $f$ is an admissible blow-up and $\ul{f}$ is associated with a star subdivision of fans.
Let $\SmAdm$ denote the smallest class of morphisms in $\SmlSm/\F_1$ containing all admissible blow-ups along smooth centers and closed under compositions.
\end{df}

\begin{prop}
\label{logSH.24}
Let $\cC$ be a symmetric monoidal stable $\infty$-category,
and let
\[
M\colon \SmlSm/\F_1\to \cC
\]
be a $\square$-invariant symmetric monoidal functor sending Zariski distinguished squares to cartesian squares.
\begin{enumerate}
\item[\textup{(1)}] Let $X\in \SmlSm/\F_1$.
Consider $\A^2\in \Sm/\F_1$ and its blow-up $\Bl_O\A^2$ at the origin, which means the star subdivision of the fan $\A^2$ relative to itself.
Then the induced morphism
\[
M(X\times [\Bl_O \A^2,\A^1\times \G_m])
\to
M(X\times [\A^2,\A^1\times \G_m])
\]
is an isomorphism if and only if the induced morphism
\[
M(X\times [\Bl_O \A^2,\G_m^2])
\to
M(X\times [\A^2,\G_m^2])
\]
is an isomorphism,
see \textup{Construction \ref{logmonoid.7}} for the notation $[-,-]$.
\item[\textup{(2)}] Assume that $M$ is invariant under dividing covers.
Then $M$ is invariant under admissible blow-ups along smooth centers.
\item[\textup{(3)}] Assume that $M$ is invariant under dividing covers.
Then $M$ is $(\P^n,\P^{n-1})$-invariant for every integer $n\geq 1$.
\item[\textup{(4)}] Assume that $M$ is $(\P^n,\P^{n-1})$-invariant for every integer $n\geq 1$.
Let $X\in \SmlSm/\F_1$.
Then the induced morphism
\[
M(X\times [\Bl_O \A^2,\G_m^2])
\to
M(X\times [\A^2,\G_m^2])
\]
is an isomorphism.
\item[\textup{(5)}] Assume that $M$ is $(\P^n,\P^{n-1})$-invariant for every integer $n\geq 1$.
Then $M$ is invariant under dividing covers.
\end{enumerate}
\end{prop}
\begin{proof}
(1) Argue as in \cite[Proposition 7.2.5]{BPO}.

(2) Let $f\colon X\to S$ be an admissible blow-up along a smooth center in $\SmlSm/\F_1$.
We need to show that $M(f)\colon M(X)\to M(S)$ is an isomorphism.
This question is Zariski local on $S$.
Hence we may assume $S=\A_\N^r\times \A^s\times \G_m^t$ for some integers $r,s,t\geq 0$.
To conclude, argue as in \cite[Steps 2 and 3 in Theorem 7.2.10]{BPO}, but use (1) also.

(3) Argue as in \cite[Proposition 7.3.1]{BPO}, but use (2) also.

(4) Argue as in \cite[Lemma 7.7.1]{BPO}, but use (1) also.

(5) Let $f\colon X\to S$ be a dividing cover.
We need to show that $M(f)\colon M(X)\to M(S)$ is an isomorphism.
This question is Zariski local on $S$.
Hence we may assume $S=\A_\N^r \times \A^s\times \G_m^t$ for some integers $r,s,t\geq 0$.
Using \cite[Lemma 7.7.3]{BPO}, we reduce to the case where $f$ is induced by a star subdivision of $\A^r$ relative to a $2$-dimensional cone.
To conclude,
argue as in \cite[Theorem 7.7.4]{BPO},
but use (4) also.
\end{proof}

\begin{lem}
\label{logSH.23}
Let $f\colon X\to S:=\A_\N^r\times \A^s\times \G_m^t$ be an admissible blow-up in $\SmlSm/\F_1$ with integers $r,s,t\geq 0$.
Then there exists an admissible blow-up $g\colon Y\to X$ in $\SmlSm/\F_1$ such that $fg$ is a composition of admissible blow-ups along smooth centers.
\end{lem}
\begin{proof}
The morphism $\ul{f}\colon \ul{X}\to \ul{S}$ is associated with a subdivision of fans $\Delta\to \Sigma:=\A^{r+s}\times \G_m^t$.
Since $f$ is an admissible blow-up,
the cone $\sigma:=\Cone(e_{r+1},\ldots,e_{r+s})$ is contained in $\Delta$,
where $e_1,\ldots,e_{r+s+t}$ are the standard coordinates in the lattice $\Z^{r+s+t}$ of the fan $\Sigma$.
By Proposition \ref{logmonoid.2}, there exists a finite sequence of star subdivisions
\[
\Delta_m\to \cdots \to \Delta_0:=\Sigma
\]
such that $\Delta_m$ is a subdivision of $\Delta$ and $\sigma\in \Delta_m$.
Consider $Y\in \cSm/\F_1$ such that $\ul{Y}$ is associated with $\Delta_m$ and $Y-\partial Y\cong S-\partial S$.
By Proposition \ref{logmonoid.25},
we have $Y\in \SmlSm/\F_1$,
and note that the induced morphism $Y\to X$ satisfies the desired properties.
\end{proof}

\begin{thm}
\label{logSH.7}
Let $\cC$ be a symmetric monoidal stable $\infty$-category,
and let
\[
M\colon \SmlSm/\F_1\to \cC
\]
be a $\square$-invariant symmetric monoidal functor sending Zariski distinguished squares to cartesian squares.
Then the following conditions are equivalent.
\begin{enumerate}
\item[\textup{(1)}]
$M$ is invariant under dividing covers.
\item[\textup{(2)}]
$M$ is invariant under admissible blow-ups along smooth centers.
\item[\textup{(3)}]
$M$ is invariant under admissible blow-ups.
\item[\textup{(4)}]
$M$ is $(\P^n,\P^{n-1})$-invariant for every integer $n\geq 1$.
\end{enumerate}
\end{thm}
\begin{proof}
To show the equivalence between (1), (2), and (4),
translate the references to \cite{BPO} used in the proof of Theorem \ref{logSH.12} in terms of log monoid schemes,
see Proposition \ref{logSH.24} for the details.
The proofs become simpler for the log monoid scheme case since \cite[Construction 7.2.8]{BPO} is not needed.

Since (3) implies (2) trivially,
it remains to show that (2) implies (3).
Let $f\colon X\to S$ be an admissible blow-up in $\SmlSm/\F_1$.
We need to show that $M(f)\colon M(X)\to M(S)$ is an isomorphism.
For this, we can work Zariski locally on $S$.
Hence we may assume $S=\A_\N^r\times \A^s\times \G_m^t$ for some integers $r,s,t\geq 0$.

Consider the category $\Adm_S$ of admissible blow-ups of $S$,
and consider the smallest class of morphisms $\SmAdm_S$ in $\Adm_S$ containing all admissible blow-ups along smooth centers and closed under compositions.
Argue as in Proposition \ref{logSH.3} and use Lemma \ref{logSH.23} to show that $\SmAdm_S$ admits a calculus of right fractions in $\Adm_S$.
Using \cite[Lemma C.2.1]{BPO} (with $\infty$-category $\cD:=\cC$ instead of a $1$-category) and the condition (2),
we see that $M(f)$ is an isomorphism.
\end{proof}

The next result provides an alternative model of $\logSH_{S^1}(\F_1)$.

\begin{thm}
\label{logSH.8}
There is an equivalence of symmetric monoidal $\infty$-categories
\[
\logSH_{S^1}(\F_1)
\simeq
\Sh_\Zar(\SmlSm/\F_1,\Sp)[(\P^\bullet,\P^{\bullet-1})^{-1}],
\]
where $(\P^\bullet,\P^{\bullet-1})$ in the formulation denotes the set of projections $X\times (\P^n,\P^{n-1})\to X$ for $X\in \SmlSm/\F_1$ and integer $n\geq 1$.
\end{thm}
\begin{proof}
This is a consequence of the implication (1)$\Leftrightarrow$(3) in Theorem \ref{logSH.7}.
\end{proof}

For symmetric monoidal $\infty$-categories $\cC$ and $\cD$,
let $\Fun^{\otimes}(\cC,\cD)$ be the $\infty$-category of symmetric monoidal functors,
and let $\Fun^{\otimes,L}(\cC,\cD)$ be the $\infty$-category of colimit preserving symmetric monoidal functors.

\begin{prop}
\label{logSH.22}
Let $\cD$ be a pointed presentable symmetric monoidal $\infty$-category.
Then the induced functor
\[
\Fun^{\otimes,L}(\logSH(\F_1),\cD)
\to
\Fun^{\otimes}(\SmlSm/\F_1,\cD)
\]
is fully faithful,
and its essential image is spanned by those functors $F\colon \SmlSm/\F_1\to \cD$ satisfying Zariski descent and $(\P^n,\P^{n-1})$-invariance for all integers $n\geq 1$ such that the cofiber of $F(\F_1)\xrightarrow{i_1} F(\P^1)$ is an invertible object of $\cD$,
where $i_1$ is the $1$-section $\Spec(\F_1)\to \P^1$.
\end{prop}
\begin{proof}
Argue as in \cite[Corollary 2.39]{zbMATH06374152}, and use Theorem \ref{logSH.8}.
\end{proof}

\begin{const}
\label{omega.20}
Let $\omega\colon \SmlSm/\F_1\to \Sm/\F_1$
(resp.\ $\omega\colon \RglRg \to \Rg$,
resp.\ $\omega\colon \SmlSm/S\to \Sm/S$ with $S\in \Sch$)
be the functor sending $Y\in \SmlSm/\F_1$ (resp.\ $Y\in \RglRg$, resp.\ $Y\in \SmlSm/S$)
to $Y-\partial Y$.
Its left adjoint sends $X\in \Sm/\F_1$ (resp.\ $X\in \Rg$, resp.\ $X\in \Sm/S$) to $X$.
Hence we have the colimit preserving adjoint functors
\begin{gather*}
\omega^\sharp
:
\PSh(\Sm/\F_1,\Spc)\rightleftarrows
\PSh(\SmlSm/\F_1,\Spc)
:
\omega_\sharp
\\
\text{
(resp.\ $
\omega^\sharp
:
\PSh(\Rg,\Spc)\rightleftarrows
\PSh(\RglRg,\Spc)
:
\omega_\sharp
$,
}
\\
\text{
resp.\ $
\omega^\sharp
:
\PSh(\Sm/S,\Spc)\rightleftarrows
\PSh(\SmlSm/S,\Spc)
:
\omega_\sharp
$)}
\end{gather*}
such that $\omega^\sharp X\cong X$ and $\omega_\sharp Y\cong Y-\partial Y$ for $X\in \Sm/\F_1$ (resp.\ $X\in \Rg$, resp.\ $X\in \Sm/S$) and $Y\in \SmlSm/\F_1$ (resp.\ $Y\in \RglRg$, resp.\ $Y\in \SmlSm/S$).

We have similar adjoint functors for $\Spc_*$ and $\Sp$ too.
\end{const}

\begin{const}
\label{omega.19}
The functor $\omega\colon \SmlSm/\F_1\to \Sm/\F_1$ and its left adjoint send $\A^1$ to $\A^1$ and Zariski distinguished squares to Zariski distinguished squares.
Hence using \cite[Corollary 2.1.1]{logA1},
we have the induced adjoint functors
\[
\omega^\sharp: \rH(\F_1)\rightleftarrows \Sh_\Zar(\SmlSm/\F_1,\Spc)[(\A^1)^{-1}]:\omega_\sharp
\]
that preserve colimits.
\end{const}

\begin{df}
Let $\ver$ be the set of morphisms $f\colon Y\to X$ in $\SmlSm/\F_1$ such that $f-\partial f\colon Y-\partial Y\to X-\partial X$ is an isomorphism.
\end{df}

\begin{prop}
\label{omega.12}
The composite functor
\begin{align*}
\rH(\F_1)
\xrightarrow{\omega^\sharp} &
\Sh_\Zar(\SmlSm/\F_1,\Spc)[(\A^1)^{-1}]
\\
\xrightarrow{L_\ver} &
\Sh_\Zar(\SmlSm/\F_1,\Spc)[(\A^1)^{-1},\ver^{-1}]
\end{align*}
is an equivalence of $\infty$-categories,
where $L_\ver$ is the localization functor.
We also have similar results for $\rH_*(\F_1)$ and $\SH_{S^1}(\F_1)$.
\end{prop}
\begin{proof}
We focus on the case of $\rH(\F_1)$ since the proofs are similar.
Let $\iota_\ver$ be a right adjoint of $L_\ver$,
which is the inclusion functor.
We only need to show that the unit and counit
\[
\id \to \omega_\sharp\iota_\ver L_\ver \omega^\sharp,
\text{ }
L_\ver \omega^\sharp \omega_\sharp \iota_\ver
\to
\id
\]
are isomorphisms.

Since $\omega^\sharp$ and $\omega_\sharp$ preserve colimits and $Y\cong \omega_\sharp\omega^\sharp Y$ for every $Y\in \Sm/\F_1$,
the unit $\id \to \omega_\sharp \omega^\sharp$ is an isomorphism.
Since $\omega\colon \SmlSm/\F_1\to \Sm/\F_1$ sends every morphism in $\ver$ to an isomorphism,
we have a natural isomorphism $\omega_\sharp\cong
\omega_\sharp\iota_\ver L_\ver$.
It follows that we have $\id \cong \omega_\sharp\iota_\ver L_\ver \omega^\sharp$.

It remains to show $L_\ver \omega^\sharp \omega_\sharp \iota_\ver \cong \id$.
Since $L_\ver$ is essentially surjective and $\omega_\sharp\cong \omega_\sharp\iota_\ver L_\ver$,
it suffices to show $L_\ver \omega^\sharp \omega_\sharp \cong L_\ver$.
Since $\omega^\sharp$, $\omega_\sharp$, and $L_\ver$ preserves colimits,
it suffices to show that the induced morphism $L_\ver \omega^\sharp \omega_\sharp X\to L_\ver X$ is an isomorphism for $X\in \SmlSm/\F_1$.
This follows from $\omega^\sharp \omega_\sharp X\cong X-\partial X$.
\end{proof}

\begin{prop}
\label{omega.8}
Let $\cF$ be an $\A^1$-invariant Zariski sheaf of spectra on $\SmlSm/\F_1$.
If the morphism
\[
p^*\colon \cF(X)\to \cF(X\times \square^n)
\]
induced by the projection $p\colon X\times \square^n\to X$ is an isomorphism for every $X\in \Sm/\F_1$ and integer $n\geq 0$,
then the induced morphism $\cF(Y)\to \cF(Y-\partial Y)$ is an isomorphism for every $Y\in \SmlSm/\F_1$.
\end{prop}
\begin{proof}
We proceed by induction on $d:=\max_{y\in Y} \rank \ol{\cM}_{Y,y}^\gp$.
The question is Zariski local on $Y$,
so we may assume $Y\cong X\times \A_\N^d$ with $X\in \Sm/\F_1$.
The claim is trivial if $d=0$,
so assume $d>0$.

Consider the Zariski covering $\{\A^1,\A_\N\}$ of $\square$.
Using cartesian products,
we obtain the Zariski covering of $\square^d$ consisting of $2^d$ open subschemes.
One of them is $U:=\A_\N^d$,
and let $V$ be the union of the other open subschemes.
Since $\cF$ is $\A^1$-invariant and $\cF(X)\cong \cF(X\times \square^n)$,
we have
\[
\cF(X\times (U\cup V))
\cong
\cF(X\times (U\cup V-\partial (U\cup V))).
\]
On the other hand,
we have
\begin{gather*}
\cF(X\times V)
\cong
\cF(X\times (V-\partial V)),
\\
\cF(X\times (U\cap V))
\cong
\cF(X\times ((U\cap V)-\partial (U\cap V)))
\end{gather*}
by induction.
Combine the three isomorphisms above and use the assumption that $\cF$ is a Zariski sheaf to conclude.
\end{proof}

The following result is analogous to \cite[Proposition 2.5.7]{logA1}.

\begin{thm}
\label{omega.13}
There is an equivalence of symmetric monoidal $\infty$-categories
\[
\SH_{S^1}(\F_1)
\simeq
\logSH_{S^1}(\F_1)[(\A^1)^{-1}].
\]
\end{thm}
\begin{proof}
As in Proposition \ref{omega.12},
we have an equivalence of symmetric monoidal $\infty$-categories
\[
\SH_{S^1}(\F_1)
\simeq
\Sh_\Zar(\SmlSm/\F_1,\Sp)[(\A^1)^{-1},\ver^{-1}].
\]
If $\cF$ is an $(\A^1,\ver)$-invariant Zariski sheaf of spectra on $\SmlSm/\F_1$,
then $\cF$ is $(\square,\divi)$-invariant.
If $\cF$ is $(\A^1,\square)$-invariant Zariski sheaf of spectra on $\SmlSm/\F_1$,
then $\cF$ is $\ver$-invariant by Proposition \ref{omega.8}.
Hence we have an equivalence of symmetric monoidal $\infty$-categories
\[
\Sh_\Zar(\SmlSm/\F_1,\Sp)[(\A^1)^{-1},\ver^{-1}]
\simeq
\Sh_\Zar(\SmlSm/\F_1,\Sp)[\square^{-1},\divi^{-1},(\A^1)^{-1}].
\]
To conclude,
observe that the right-hand side is equivalent to $\logSH_{S^1}(\F_1)[(\A^1)^{-1}]$.
\end{proof}

\begin{rmk}
\label{omega.18}
As in \cite[Theorem 3.1.29]{Arndt},
we have an isomorphism $\P^\infty \cong \rB\G_m$ in $\rH_*(\F_1)$.
Using the multiplication $\G_m\times \G_m\to \G_m$, we have the induced morphism $\P^\infty\otimes \P^\infty\to \P^\infty$ in $\rH_*(\F_1)$,
which yields the morphism $\beta\colon \P^1\otimes \P^\infty\to \P^\infty$ in $\rH_*(\F_1)$.
As noted in \cite[\S 1]{Arndt},
we can define the \emph{motivic K-theory spectrum} over $\F_1$ as
\begin{align*}
\KGL&:=\Sigma_{\P^1}^\infty (\P^\infty)_+[\beta^{-1}]
\\ &:=\colim(\Sigma_{\P^1}^\infty (\P^\infty)_+\xrightarrow{\beta} \Sigma^{-2,-1}\Sigma_{\P^1}^\infty (\P^\infty)_+ \xrightarrow{\beta} \cdots)\in \SH(\F_1).
\end{align*}
This definition is reasonable since for $S\in \Sch$,
we have $p^*\KGL\cong \KGL$ in $\SH(S)$ by \cite[Theorem 4.17]{zbMATH05663782} or \cite[Theorem 1.1]{zbMATH05549229},
where $p\colon S\to \Spec(\F_1)$ is the structure morphism.
Unfortunately, the definition of $\KGL$ in $\SH(\F_1)$ is abstract, and we do not know how to compute the  \emph{Snaith K-theory spectrum} of $X$ defined by
\[
\Kth^\Sna(X)
:=
\hom_{\SH(\F_1)}(\Sigma_{\P^1}^\infty X_+,\KGL).
\]
for $X\in \Sm/\F_1$.
This behaves differently from the K-theory spectrum $\Kth(X)$ defined in \cite[\S 5.3]{zbMATH06016605}.
For example,
we have $\Kth^\Sna(\P^1)\cong \Kth^\Sna(\F_1)\oplus \Kth^\Sna(\F_1)$,
while $\Kth(\F_1)\cong \Z$ and $\Kth_0(\P^1)\cong \Z^{\N}$ by \cite[Theorem 5.9, Corollary 5.15]{zbMATH06016605}.
\end{rmk}

\section{Suslin functor}

The Suslin functor for the interval $\A^1$ makes $\A^1$-motivic homotopy theory computationally accessible.
A fundamental example is the $\A^1$-invariant motivic complex in \cite[Definition 3.1]{MVW},
which is obtained by applying the Suslin functor to the sheaf with transfers $\Z_\mathrm{tr}(\G_m^{\wedge q})[-q]$ for integers $q\geq 0$.

In logarithmic motivic homotopy theory,
the Suslin functor for the interval $\square$ plays a similar role. Unlike $\A^1$,
$\square$ is \emph{not} an interval object of $\SmlSm/\F_1$ since we do not have the multiplication morphism $\square\times \square \to \square$ extending the multiplication morphism $\A^1\times \A^1\to \A^1$.
This is the reason why we need to invert admissible blow-ups in $\SmlSm/\F_1$ to use the Suslin functor for $\square$.
Such an inversion of admissible blow-ups makes the Suslin functor for $\square$ more complicated.

The purpose of this section is to review the Suslin functor and its cubical version.
Even though the Suslin and cubical Suslin functors for an interval are equivalent by \eqref{omega.14.2} in a good case,
the formulation of the cubical Suslin functor can have a computational advantage since it does not use the diagonal morphism of the interval.

\

Let $\ECube$ denote the category considered in \cite[Remark B.2.3]{MR4442406} and
\cite[D\'efinition A.6]{MR3259031}.
Its objects are the the symbols $\ul{\unit}^n$ for $n\in \N$.
Its morphisms are generated by
\begin{gather*}
\delta_{i,\epsilon}\colon \ul{\unit}^n \to \ul{\unit}^{n+1}\text{ for }i\in [1,n+1]\text{ and }\epsilon\in [0,1],
\\
p_i\colon \ul{\unit}^n\to \ul{\unit}^{n-1} \text{ for }i\in [1,n],
\\
\mu_i\colon \ul{\unit}^n\to \ul{\unit}^{n-1} \text{ for }i\in [1,n-1].
\end{gather*}
Let $\Cube$ denote the category considered in \cite[\S B.2.1]{MR4442406} and \cite[D\'efinition A.1]{MR3259031},
which is the subcategory of $\ECube$ whose morphisms are only generated by $\delta_{i,\epsilon}$ and $p_i$.
If $\cC$ is a category, then a \emph{cubical object of $\cC$} (resp.\ \emph{extended cubical object of $\cC$}) is a functor $\Cube^\op\to \cC$ (resp.\ $\ECube^\op\to \cC$).

\begin{const}
\label{logSH.25}
Let $\cC$ be a unital monoidal category,
and let $I$ be an object of $\cC$ equipped with morphisms
\[
i_0,i_1\colon \unit \rightrightarrows I,
\text{ }
p\colon I\to \unit
\]
satisfying $pi_0=pi_1=\id$,
where $\unit$ denotes the unit of $\cC$.
We can naturally associate $I^\bullet \colon \Cube\to \cC$ as follows.
We have the morphism
\[
\delta_{i,\epsilon}
\colon
I^n \xrightarrow{\cong} I^{i-1}\otimes \unit \otimes I^{n+1-i}
\xrightarrow{\id \otimes i_\epsilon\otimes \id}
I^{i-1}\otimes I\otimes I^{n+1-i}
\xrightarrow{\cong}
I^{n+1}
\]
for integers $n\geq 0$, $1\leq i\leq n+1$, and $\epsilon=0,1$.
We have the $i$th projection
\[
p_i
\colon
I^n \to I^{n-1}
\]
for integers $n\geq 0$ and $1\leq i\leq n$.

We have the induced functor $\Z[I^n]\colon \Cube \to \PSh(\cC,\Ab)$ for $n\in \N$.
We also have $\Z[I^\bullet]^\flat\in \PSh(\cC,\rD(\Z))$ such that
\[
\Z[I^n]^\flat
:=
\coker\big(\oplus \delta_{i,0} \colon \bigoplus_{i=1}^n \Z[I^{n-1}]\to \Z[I^n]\big)
\]
in degree $n\in \N$ and the differential is induced by $\sum_{i=1}^n (-1)^i \delta_{i,1}$,
see also \cite[Remark B.1.4]{MR4442406} and \cite[Lemma A.3, Remarque A.5]{MR3259031}.

Let $\cF$ be a presheaf of complexes on $\cC$.
The \emph{cubical Suslin complex} $\CSing^I \cF$ in \cite[Definition B.6.2]{MR4442406} can be formulated as
\[
\CSing^I \cF(X)
:=
\hom_{\PSh(\cC,\rD(\Z))}
(\Z[I^\bullet]^\flat \otimes \Z[X],\cF)
\in \rD(\Z)
\]
for $X\in \cC$,
where $\hom$ denotes the hom complex.
For $n\in \N$,
we set
\[
\CSing_n^I \cF(X)
:=
\hom_{\PSh(\cC,\rD(\Z))}
(\Z[I^n]^\flat \otimes \Z[X],\cF)
\in \rD(\Z).
\]
\end{const}

\begin{exm}
With the above notation,
assume that $\cF$ is a presheaf of abelian groups on $\cC$.
We have an explicit description of $\CSing^I\cF(X)$ obtained by \cite[Lemma B.1.3]{MR4442406} as follows.
In degree $n\in \N$,
we have
\[
\CSing_n^I\cF(X)
\cong
\bigcap_{i=1}^n \ker(\cF(X\times I^n)\xrightarrow{\delta_{i,0}^*} \cF(X\times I^{n-1})),
\]
where we regard $\cF(X\times I^{n-1})$ as $0$ if $n=0$.
The differential is given by
\[
\delta^*:=\sum_{i=1}^n (-1)^i \delta_{i,1}^*\colon \CSing_n^I\cF(X)\to \CSing_{n-1}^I\cF(X).
\]
\end{exm}

\begin{exm}
As noted in \cite[Remarque A.23]{MR3259031},
we have the monoidal products on $\Cube$ and $\ECube$ given by
\[
\ul{\unit}^m\otimes \ul{\unit}^n:=\ul{\unit}^{m+n}
\]
for $m,n\in \N$.
Observe that $\ul{\unit}^0$ is the monoidal unit.
Also, $\ul{\unit}^1$ satisfies the condition for $I$ in Construction \ref{logSH.25}.
\end{exm}

\begin{const}
\label{logSH.14}
Let $\cF$ be a cubical object of $\rD(\Z)$.
Then the \emph{cubical geometric realization of $\cF$} is
\[
\lvert \cF \rvert_\mathrm{cube}
:=
\hom_{\PSh(\Cube,\rD(\Z))}(\Z[\ul{\unit}^\bullet]^\flat,\cF).
\]
\end{const}

\begin{prop}
\label{cube.1}
Let $\cF$ be a presheaf of complexes on a unital monoidal category $\cC$,
and let $I$ be an object of $\cC$ equipped with morphisms $i_0,i_1\colon \unit \to I$ and $p\colon I\to \unit$ satisfying $pi_0=pi_1=\id$,
where $\unit$ denotes the unit of $\cC$.
Then for $X\in \cC$,
there is a natural isomorphism of complexes
\[
\CSing^I\cF(X)
\cong
\lvert \cF(X\otimes I^\bullet) \rvert_\mathrm{cube}.
\]
\end{prop}
\begin{proof}
Consider the functor $\alpha\colon \Cube\to \cC$ sending $\ul{\unit}^n$ to $X\otimes I^n$ for $n\in \N$.
Use the induced pair of adjunctions
\[
\alpha_!
:
\PSh(\Cube,\rD(\Z))
\rightleftarrows
\PSh(\cC,\rD(\Z))
:
\alpha^*
\]
to compare $\CSing^I\cF(X)$ and $\lvert \cF(X\otimes I^\bullet) \rvert_\mathrm{cube}$.
\end{proof}

\begin{const}
\label{logSH.13}
Let $\cC$ be a unital monoidal category,
and let $I$ be an interval object of $\cC$ in the sense of \cite[Definition B.2.1]{MR4442406}, which is a variant of \cite[\S 2.2]{VSelecta}.
This means that $I$ is equipped with morphisms
\[
\mu\colon I\otimes I\to I,
\text{ }
i_0,i_1\colon \unit \rightrightarrows I,
\text{ }
p\colon I\to \unit
\]
such that
\begin{equation}
\label{omega.14.1}
pi_0=pi_1=\id,
\text{ }
\mu(i_0\otimes \id)=\mu(\id \otimes i_0)=i_0 p,
\text{ }
\mu(i_1\otimes \id)=\mu(\id \otimes i_1)=\id.
\end{equation}

We can extend $I^\bullet \colon \Cube \to \cC$ to  $I^\bullet \colon \ECube\to \cC$ by adding the $i$th multiplication
\[
\mu_i
\colon
I^n
\xrightarrow{\cong}
I^{i-1}\otimes I^2 \otimes I^{n-i-1}
\xrightarrow{\id \otimes \mu \otimes \id}
I^{i-1}\otimes I\otimes I^{n-i-1}
\xrightarrow{\cong}
I^{n-1}
\]
for integers $n\geq 2$ and $1\leq i\leq n-1$.
\end{const}

\begin{const}
\label{omega.14}
Let $\cC$ be a category with finite product,
and let $I$ be an interval object of $\cC$ with respect to the monoidal product $\times$.
The functor $\Sing^I$ in \cite[\S 2.3]{MV} can be formulated in the stable setting as follows.
Consider the cosimplicial object $\Delta_I^\bullet \colon \bDelta\to \cC$ defined as in \cite[p.\ 88]{MV}.
We have the \emph{Suslin functor}
\[
\Sing^I
\colon
\PSh(\cC,\Sp)
\to
\PSh(\cC,\Sp).
\]
given by
\begin{equation}
\label{omega.14.3}
\Sing^I\cF(X):=\colim_{n\in \bDelta^\op} \cF(X\times \Delta_I^n)
\end{equation}
for $\cF\in \PSh(\cC,\Sp)$ and $X\in \cC$.

Consider the adjoint functors
\[
L_I:
\PSh(\cC,\Sp)\rightleftarrows
\PSh(\cC,\Sp)[I^{-1}]
:
\iota_I,
\]
where $I^{-1}$ in the formulation is the set of projections $X\times I\to X$ for all $X\in \cC$,
and $L_I$ is the localization.
Argue as in \cite[Corollary 2.19]{MVW} and \cite[Corollary 3.8 in \S 2]{MV} to show that $\Sing^I\cF$ is $I$-local and the induced morphism $\cF\to \Sing^I \cF$ is an $I$-local equivalence.
We can also consider $\Sing^I$ for complexes instead of spectra.

Together with \cite[Theorem B.6.3]{MR4442406},
we have natural isomorphisms
\begin{equation}
\label{omega.14.2}
\Sing^I
\cong
\iota_I L_I
\cong
\CSing^I
\colon
\PSh(\cC,\rD(\Z))
\to
\PSh(\cC,\rD(\Z)).
\end{equation}
\end{const}

\begin{prop}
\label{logSH.16}
The monoid scheme $\A^1$ is an interval object of $\Sm/\F_1$.
\end{prop}
\begin{proof}
We have the multiplication morphism $\mu\colon \A^1\times \A^1\to \A^1$ induced by the diagonal map $\N \to \N\oplus \N$.
Let $i_0,i_1\colon \pt \rightrightarrows \A^1$ be the $0$-section and $1$-section.
Check \eqref{omega.14.1} to conclude.
\end{proof}

\begin{const}
\label{logSH.17}
Let $\Bl_{(\infty,0)+(0,\infty)}(\square^2)$ be the log monoid scheme in $\SmlSm/\F_1$ such that its underlying monoid scheme is associated with the fan in $\Z^2$ whose maximal cones are
\begin{gather*}
\Cone(e_1,e_2),
\text{ }
\Cone(e_2,e_2-e_1),
\text{ }
\Cone(-e_1,e_2-e_1),
\\
\text{ }
\Cone(e_1,e_1-e_2),
\text{ }
\Cone(-e_2,e_1-e_2),
\text{ }
\Cone(-e_1,-e_2)
\end{gather*}
and $\Bl_{(\infty,0)+(0,\infty)}(\square^2)-\partial \Bl_{(\infty,0)+(0,\infty)}(\square^2)\cong \A^2$.
We have the induced admissible blow-up $p\colon \Bl_{(\infty,0)+(0,\infty)}(\square^2)\to \square^2$.
We also have the morphism
\[
\mu'\colon \Bl_{(\infty,0)+(0,\infty)}(\square^2)\to \square
\]
induced by the summation map of the lattice $\Z^2\to \Z$.
Let $\mu\colon \square^2\to \square$ be the morphism $gf^{-1}$ in $\SmlSm/\F_1[\Adm^{-1}]$.
\end{const}

\begin{prop}
\label{logSH.18}
The log monoid scheme $\square$ is an interval object of the category $\SmlSm/\F_1[\Adm^{-1}]$.
\end{prop}
\begin{proof}
Let $i_0,i_1\colon \pt \rightrightarrows \square^1$ be the $0$-section and $1$-section.
Consider $\mu$ in Construction \ref{logSH.17}.
Check \eqref{omega.14.1} to conclude.
\end{proof}

\begin{const}
\label{logSH.15}
Argue as in \cite[Lemma 4.3]{loghomotopy} to obtain an equivalence of $\infty$-categories
\[
\PSh(\SmlSm/\F_1[\Adm^{-1}],\Sp)
\simeq
\PSh(\SmlSm/\F_1,\Sp)[\Adm^{-1}].
\]
We have the Suslin functor
\[
\Sing^\square\colon \PSh(\SmlSm/\F_1[\Adm^{-1}],\Sp)
\to
\PSh(\SmlSm/\F_1[\Adm^{-1}],\Sp)
\]
and the localization functor
\[
L_\Adm\colon \PSh(\SmlSm/\F_1,\Sp)
\to
\PSh(\SmlSm/\F_1,\Sp)[\Adm^{-1}].
\]
Combining all these together,
we obtain the functor
\[
\Sing^\square L_\Adm
\colon
\PSh(\SmlSm/\F_1,\Sp)
\to
\PSh(\SmlSm/\F_1,\Sp)[\Adm^{-1}].
\]
For $\cF\in \PSh(\SmlSm/\F_1,\Sp)$,
we have
\begin{equation}
\label{logSH.15.1}
\Sing^\square L_\Adm \cF(X)
\cong
\colim_{n\in \bDelta^\op}
\colim_{Y\in \Adm_{X\times \square^n}^\op}
\cF(Y)
\end{equation}
by \eqref{omega.14.3} and Proposition \ref{logSH.18},
where $\Adm_{X\times \square^n}$ denotes the category of admissible blow-ups of $X\times \square^n$ in $\SmlSm/\F_1$.

Similarly, we obtain the functor
\[
\CSing^\square L_\Adm
\colon
\PSh(\SmlSm/\F_1,\rD(\Z))
\to
\PSh(\SmlSm/\F_1,\rD(\Z))[\Adm^{-1}].
\]
For $\cF\in \PSh(\SmlSm/\F_1,\rD(\Z))$,
Proposition \ref{cube.1} yields
\begin{equation}
\label{logSH.15.2}
\CSing^\square L_\Adm \cF(X)
\cong
\big\lvert \colim_{Y\in \Adm_{X\times \square^\bullet}^\op} \cF(Y)\big\rvert_\mathrm{cube}.
\end{equation}
\end{const}

\section{Logarithm functor: Case of monoid schemes}
\label{boundary}

The purpose of this section is to explain how to naturally make a presheaf on $\Sm/\F_1$ into a presheaf on $\SmlSm/\F_1$ that potentially enjoys motivic properties.

The Suslin functor for $\square$ has a slight computational disadvantage: If $X=X'\times \T_\Sigma$ for some $X'\in \Sm/\F_1$ and a smooth fan $\Sigma$,
then the category $\Adm_{X\times \square^n}$ is unnecessarily more complicated then $\Adm_{\T_\Sigma \times \square^n}$ even though $\Adm_{\T_\Sigma \times \square^n}$ contains the essential data.
We refer to Definition \ref{logmonoid.26} for $\T_\Sigma$.
The idea of the logarithm functor in Definition \ref{boundary.33} is to work with $\Adm_{\T_\Sigma \times \square^n}$ instead of $\Adm_{X\times \square^n}$ for such an $X$.

\

We set $\Gmlog:=(\P^1,0+\infty)\in \SmlSm/\F_1$.
For a fan $\Sigma$,
we often regard it as the associated toric monoid scheme.
We also regard the monoid schemes $\A^1$, $\G_m$, and $\P^n$ for integers $n\geq 1$ as the corresponding fans.

\begin{df}
\label{boundary.9}
Let $X\in \SmlSm/\F_1$.
For $n\in \N$,
let $\SBl_X^n$ be the category of admissible blow-ups $Y\to X\times \square^n$ in $\SmlSm/\F_1$ such that there exists a dividing cover $g\colon V\to X\times (\Gmlog)^n$ in $\SmlSm/\F_1$ satisfying $\ul{f}=\ul{g}$.
Such an admissible blow-up $Y\to X\times \square^n$ is called a \emph{standard blow-up}.

Observe that every morphism in $\SBl_X^n$ is an admissible blow-up.
Furthermore, there exists a commutative square
\[
\begin{tikzcd}
V\ar[d,"g"']\ar[r]&
Y\ar[d,"f"]
\\
X\times (\Gmlog)^n\ar[r]&
X\times \square^n
\end{tikzcd}
\]
in $\SmlSm/\F_1$ by Proposition \ref{logmonoid.5},
where $X\times (\Gmlog)^n\to X\times \square^n$ is the canonical map whose underlying morphism of schemes is the identity.
\end{df}

\begin{df}
\label{boundary.48}
This is a variant of the above definition,
which is needed for technical convenience. For $X\in \SmlSm/\F_1$ and $n\in \N$,
let $\Sdiv_X^n$ be the category of dividing covers $V\to X\times (\Gmlog)^n$ in $\lSm/\F_1$ such that $g^{-1}(X\times \A_\N^n)\to X\times \A_\N^n$ is an isomorphism.
Such a dividing cover $V\to X\times (\Gmlog)^n$ is called a \emph{standard dividing cover.}
\end{df}

\begin{prop}
\label{boundary.62}
Let $X\in \SmlSm/\F_1$ and $n\in \N$.
For $Y,Y'\in \SBl_X^n$ and $V,V'\in \Sdiv_X^n$,
$\Hom_{\SBl_X^n}(Y',Y)$ and $\Hom_{\Sdiv_X^n}(V',V)$ have at most one element.
\end{prop}
\begin{proof}
This is a consequence of Propositions \ref{image.6} and \ref{logmonoid.5}.
\end{proof}

\begin{prop}
\label{boundary.63}
For $X\in \SmlSm/\F_1$ and integer $n\geq 0$,
there is a fully faithful functor of categories
\[
\rho
\colon
\SBl_X^n
\to
\Sdiv_X^n
\]
such that $\ul{Y}=\ul{\rho(Y)}$ for $Y\in \SBl_X^n$.
Furthermore,
$V\in \Sdiv_X^n$ is in the essential image if and only if $V\in \SmlSm/\F_1$.
\end{prop}
\begin{proof}
For $Y\in \SBl_X^n$, there exists a dividing cover $g\colon V\to X\times (\Gmlog)^n$ in $\SmlSm/\F_1$ satisfying $\ul{f}=\ul{g}$.
We set $\rho(Y):=V$.
Since $g^{-1}(X\times \A_\N^n)\cong X\times \A_\N^n$,
we have $\rho(Y)\in \Sdiv_X^n$.
For another $Y'\in \SBl_X^n$, consider the corresponding dividing cover $V'\to X\times (\Gmlog)^n$ too.
Observe that $\Hom_{\SBl_X^n}(Y',Y)$ and $\Hom_{\Sdiv_X^n}(V',V)$ have at most one element by Proposition \ref{boundary.62}.
If $Y'\to Y$ is a morphism in $\SBl_X^n$,
then we have the composite morphism
\[
V'-\partial V'
\xrightarrow{\cong}
Y'\times_{\square^n}\G_m^n
\to
Y\times_{\square^n} \G_m^n
\xrightarrow{\cong}
V-\partial V,
\]
which yields a morphism $V'\to V$ in $\Sdiv_X^n$ by Proposition \ref{logmonoid.5}.
Conversely,
if $V'\to V$ is a morphism in $\Sdiv_X^n$,
then we have the composite morphism
\[
Y'-\partial Y'
\xrightarrow{\cong}
\ul{V'}\times_{(\P^1)^n}\A^n
\to
\ul{V}\times_{(\P^1)^n}\A^n
\to
Y-\partial Y,
\]
which yields a morphism $Y'\to Y$ in $\SBl_X^n$ by Proposition \ref{logmonoid.5}.
It follows that $\rho$ has a functor structure and is fully faithful.

Let us identify the essential image of $\rho$.
By definition,
for every $Y\in \SBl_X^n$,
we have $\rho(Y)\in \SmlSm/\F_1$.
Conversely,
for $V\in \Sdiv_X^n$ such that $V\in \SmlSm/S$,
consider $Y\in \cSm/\F_1$ such that $\ul{Y}:=\ul{V}$ and $Y-\partial Y:=(X-\partial X)\times \A^n$.
Then we have an induced admissible blow-up $Y\to X\times \square^n$,
so we have $Y\in \SmlSm/\F_1$ by Proposition \ref{logmonoid.25} and hence $Y\in \SBl_X^n$.
Furthermore, we have $\rho(Y)\cong V$.
\end{proof}

\begin{exm}
\label{boundary.54}
Assume $X:=\square^r$ with an integer $r\geq 0$.
For an integer $n\geq 0$,
consider $(X\times (\Gmlog)^n)^\sharp$ obtained by Construction \ref{logmonoid.16},
which is isomorphic to $(\square-\{0\})^r \times (\Gmlog)^n$.
A dividing cover $g\colon V\to X\times (\Gmlog)^n$ in $\lSm/\F_1$ corresponds to a subdivision of $(\P^1-\{0\})^r \times (\P^1)^n$,
which also corresponds to a subdivision $\Sigma \to (\P^1)^{r+n}$ such that if $a:=(a_1,\ldots,a_{r+n})$ is a ray of $\Sigma$ such that $a_i>0$ for some $1\leq i\leq r$, then we have $a=e_i$.
The condition
\[
g^{-1}((X-\partial X)\times \A_\N^n)\cong (X-\partial X)\times \A_\N^n
\]
corresponds to the condition $\Cone(e_{r+1},\ldots,e_{r+n})\in \Sigma$.
\end{exm}

Next, we discuss some basic properties of $\SBl_X^n$ and $\Sdiv_X^n$.

\begin{prop}
\label{boundary.29}
Let $X\in \SmlSm/\F_1$ and $n\in \N$.
Then for every $V\in \Sdiv_X^n$,
there exists $Z\in \SBl_X^n$ and a morphism $\rho(Z)\to W$ in $\Sdiv_X^n$.
\end{prop}
\begin{proof}
By Proposition \ref{logmonoid.1},
there exists a dividing cover $p\colon W\to V$ with $W\in \SmlSm/\F_1$ such that for every open subscheme $U$ of $V$ with $U\in \SmlSm/\F_1$,
we have $p^{-1}(U)\cong U$.
In particular,
we have $(gp)^{-1}((X-\partial X)\times \A_\N^n)\cong (X-\partial X)\times \A_\N^n$,
so we have $W\in \Sdiv_X^n$.
Since $W\in \SmlSm/\F_1$,
we have $W\cong \rho(Z)$ for some $Z\in \SBl_X^n$ by Proposition \ref{boundary.63}.
\end{proof}

\begin{prop}
\label{boundary.17}
Let $X\in \SmlSm/\F_1$.
Then the categories $\SBl_X^n$ and $\Sdiv_X^n$ are cofiltered.
\end{prop}
\begin{proof}
Let $V,V'\in \Sdiv_X^n$.
The fiber product $V'':=V\times_{X\times (\Gmlog)^n} V'$ yields $Y''\in \Sdiv_X^n$ since the preimage of $U:=(X-\partial X)\times \A_\N^n$ in $V''$ is $U$.
Hence $\Sdiv_X^n$ is connected.
Together with Proposition \ref{boundary.62},
we see that $\Sdiv_X^n$ is cofiltered.

Now let $Y,Y'\in \SBl_X^n$.
We set $V:=\rho(V)$ and $V':=\rho(V')$,
and form $V''$ as above.
By Proposition \ref{boundary.29},
there exists a morphism $\rho(Z)\to V''$ in $\Sdiv_X^n$ with $Z\in \SBl_X^n$.
Then we have morphisms $Z\to Y,Y'$ by Proposition \ref{boundary.63},
so $\SBl_X^n$ is connected and hence cofiltered using Proposition \ref{boundary.62}.
\end{proof}

\begin{prop}
\label{boundary.31}
Let $X\in \SmlSm/\F_1$ and $n\in \N$.
Then the fully faithful functor $\rho \colon \SBl_X^n \to \Sdiv_X^n$ is cofinal.
\end{prop}
\begin{proof}
This is immediate from Propositions \ref{boundary.29} and \ref{boundary.17}.
\end{proof}

The following two constructions are the key ingredients for establishing the naturality of the functor $\widetilde{\Log}$ in Definition \ref{boundary.32} below.

\begin{const}
\label{boundary.35}
Let $f\colon X'\to X$ be a morphism in $\SmlSm/\F_1$ and $n\in \N$.
If $g\colon V\to X\times (\Gmlog)^n$ is a dividing cover in $\Sdiv_X^n$,
consider its pullback $g'\colon V'\to X'\times (\Gmlog)^n$.
By Proposition \ref{logmonoid.13},
$g'$ is a dividing cover.
Furthermore,
\[
g^{-1}((X-\partial X)\times \A_\N^n)\cong (X-\partial X)\times \A_\N^n
\]
implies
\[
g'^{-1}((X'-\partial X')\times \A_\N^n)\cong (X'-\partial X')\times \A_\N^n
\]
since $f$ sends $X'-\partial X'$ into $X-\partial X$.
Hence we have $V'\in \Sdiv_{X'}^n$.
The construction $V'$ is natural,
so we obtain a functor $f^*\colon \Sdiv_X^n\to \Sdiv_{X'}^n$ given by $f^*(V):=V'$.

Similarly,
let $\alpha\colon [m]\to [m']$ be a morphism in $\bDelta$ that is one of $p_i\colon [n+1] \to [n]$ for $0\leq i\leq n$ and $\delta_i\colon [n]\to [n+1]$ for $0\leq i\leq n$ but not $i\neq n+1$.
Then we have the induced morphism $(\Gmlog)^\alpha \colon (\Gmlog)^{m'}\to (\Gmlog)^m$ such that its underlying morphism $(\P^1)^{m'}\to (\P^1)^m$ extends $\A^\alpha \colon \A^{m'}\to \A^m$.
Indeed, 
this claim is clear for $p_i$.
The morphism $(\Gmlog)^{\delta_0}$ is the $1$-section in the first coordinate.
The morphism $(\Gmlog)^{\delta_i}$ for $1\leq i\leq n$ is identified with
\[
(\Gmlog)^{i-1} \times \Gmlog \times (\Gmlog)^{n-i}
\xrightarrow{\id \times d \times id}
(\Gmlog)^{i-1} \times (\Gmlog)^2 \times (\Gmlog)^{n-i},
\]
where $d\colon \Gmlog\to (\Gmlog)^2$ is the diagonal morphism.
Argue as above to obtain a functor $\alpha^*\colon \Sdiv_X^m\to \Sdiv_X^{m'}$.
\end{const}

\begin{const}
\label{boundary.38}
Let $X\in \SmlSm/\F_1$.
Consider the morphism $\alpha:=\delta_n\colon [n-1]\to [n]$ in $\bDelta$ with an integer $n\geq 1$,
and let $g\colon W\to X\times (\Gmlog)^n$ be a partial dividing cover in $\SmlSm/\F_1$ in the sense of Definition \ref{logmonoid.18} such that the induced morphism
\begin{equation}
\label{boundary.38.2}
g^{-1}((X-\partial X)\times \A_\N^n)\to (X-\partial X)\times \A_\N^n
\end{equation}
is an open immersion.
Observe that the morphism $\square^\alpha\colon \square^{n-1}\to \square^n$ is the $0$-section in the $n$th coordinate.

Let $\ul{\alpha^*(W)}$ be the closure of $(X-\partial X)\times \A^m$ in $\ul{W}\times_{(\P^1)^n}(\P^1)^{n-1}$,
see Definition \ref{image.1} for the notion of the closure for monoid schemes.
Observe that $\ul{\alpha^*(W)}$ is a smooth toric monoid scheme by Lemma \ref{image.8}.

We claim that there exists a unique
$\alpha^*(W)\in \SmlSm/\F_1$ with underlying monoid scheme $\ul{\alpha^*(W)}$ satisfying the following conditions:
\begin{enumerate}
\item[(i)] $\alpha^*(W)-\partial \alpha^*(W):=((X-\partial X)\times \G_m^{n-1})\times_{\ul{X}\times (\P^1)^n}\ul{Y}$.
\item[(ii)]
Let $\alpha^*(g)\colon \alpha^*(W)\to X\times (\Gmlog)^{n-1}$ be the induced morphism in $\SmlSm/\F_1$.
Then $\alpha^*(g)$ is a partial dividing cover.
\item[(iii)] The induced morphism 
\begin{equation}
\label{boundary.38.1}
\alpha^*(g)^{-1}((X-\partial X)\times \A_\N^{n-1})\to (X-\partial X)\times \A_\N^{n-1}
\end{equation}
is an open immersion.
\end{enumerate}
The uniqueness is ensured by (i).
We can work Zariski locally on $X$ for the claim,
so we may assume $X=\A_\N^r\times \A^s \times\G_m^t$ for some $r,s,t\in \N$.
Then $g$ is $\id\colon \A^s \times \G_m^t\to \A^s\times \G_m^t$ times a dividing cover.
We can take out the $\A^s\times \G_m^t$ part,
so we reduce to the case where $s=t=0$.

In this case, let $\alpha^*(W):=\T_\Sigma$, where $\Sigma$ is the fan associated with $\ul{\alpha^*(V)}$.
Then we have (i).
Since $\ul{\alpha^*(g)}$ is birational,
$\alpha^*(g)$ is a partial dividing cover by Examples \ref{msch.13} and \ref{logmonoid.9},
so we have (ii).
Since \eqref{boundary.38.2} is an open immersion.
we have (iii).

Now, consider $U\in \cSm/\F_1$ such that $\ul{U}:=\ul{W}$ and $U-\partial U:=(X-\partial X)\times \A^n$,
and consider $\alpha^*(U)\in \cSm/\F_1$ such that
\[
\ul{\alpha^*(U)}:=\ul{\alpha^*(V)},
\text{ }
\alpha^*(U)-\partial \alpha^*(U):=(X-\partial X)\times \A^{n-1}.
\]
By Proposition \ref{logmonoid.25},
we have $U,\alpha^*(U)\in \SmlSm/\F_1$.
Furthermore,
we have the induced commutative square
\begin{equation}
\label{boundary.38.3}
\begin{tikzcd}
\alpha^*(U)\ar[d]\ar[r]&
U\ar[d]
\\
X\times \square^{n-1}\ar[r,"\id \times \square^\alpha"]&
X\times \square^n.
\end{tikzcd}
\end{equation}
\end{const}

\begin{prop}
\label{boundary.21}
Let $X\in \SmlSm/\F_1$.
Consider the morphism $\alpha:=\delta_n\colon [n-1]\to [n]$ in $\bDelta$ with an integer $n\geq 1$.
Then there exists a functor $\alpha^*\colon \SBl_X^n \to \SBl_X^{n-1}$ satisfying the following property:
For $Y\in \SBl_X^n$,
there exists a commutative square
\begin{equation}
\label{boundary.21.1}
\begin{tikzcd}
\alpha^*(Y)\ar[d]\ar[r]&
Y\ar[d]
\\
X\times \square^{n-1}
\ar[r,"\id \times \square^\alpha"]&
X\times \square^n.
\end{tikzcd}
\end{equation}
\end{prop}
\begin{proof}
Let $V\to X\times (\Gmlog)^n$ be a dividing cover corresponding to $Y$.
Consider the partial dividing $\alpha^*(V)\to X\times (\Gmlog)^m$ obtained by Construction \ref{boundary.38},
which is a dividing cover since it is proper.
The morphism \eqref{boundary.38.1} is a proper open immersion and hence an isomorphism.
It follows that the morphism $\alpha^*(Y)\to X\times \square^{n-1}$ obtained by Construction \ref{boundary.38} is an admissible blow-up, so we have $\alpha^*(Y)\in \SBl_X^{n-1}$.
We also obtain \eqref{boundary.21.1} from \eqref{boundary.38.3}.
The construction of $\alpha^*(Y)$ is natural, so we obtain a desired functor $\alpha^*$.
\end{proof}

\begin{df}
\label{boundary.74}
Let $\SBl(\F_1)$ be the category of triples $(Y, X,[n])$ with $Y\in \SBl_X^n$, $X\in \SmlSm/\F_1$, and $n\in \N$ whose morphisms are of the form
\[
(u,f,\alpha)\colon (Y',X',[n'])\to (Y,X,[n])
\]
such that the square
\[
\begin{tikzcd}
Y'\ar[d]\ar[r,"u"]&
Y\ar[d]
\\
X'\times \square^{n'}\ar[r,"f\times \square^\alpha"]&
X\times \square^n
\end{tikzcd}
\]
commutes.
Equivalently,
the square
\[
\begin{tikzcd}[column sep=huge]
(X'-\partial X')\times \A^{n'}\ar[d]\ar[r,"(f-\partial f)\times \A^{\alpha}"]&
(X-\partial X)\times \A^n\ar[d]
\\
Y'\ar[r,"u"]&
Y
\end{tikzcd}
\]
commutes since $(X'-\partial X')\times \A^{n'}$ is dense in $Y'$.

A morphism $(u,f,\alpha)\colon (Y',X',[n'])\to (Y,X,[n])$ is an \emph{admissible blow-up} if $f$ and $\alpha$ are isomorphisms.
Let $\SAdm$ denote the class of admissible blow-ups.
\end{df}

\begin{prop}
\label{boundary.65}
The class $\SAdm$ in $\SBl(\F_1)$ admits calculus of right fractions.
\end{prop}
\begin{proof}
Let
\[
(u,f,\alpha)\colon (Y',X',[n'])\to (Y,X,[n]),
\text{ }
(v,\id,\id)\colon (Z,X,[n])\to (Y,X,[n])
\]
be morphisms in $\SBl(\F_1)$ so that $(v,\id,\id)\in \SAdm$.
Note that $\alpha$ is a composite of several $\delta_i\colon [l]\to [l+1]$ and $p_i\colon [l]\to [l-1]$.
To check the Ore condition,
we may assume that $f=\id$ and $\alpha$ is one of the above or $f$ is arbitrary and $\alpha=\id$.
By Propositions \ref{boundary.29} and \ref{boundary.21} and Construction \ref{boundary.35},
we have a morphism $(w,f,\alpha)\colon (Z',X',[n'])\to (Z,X,[n])$ for some $w$.
Since $\SBl_{X'}^{n'}$ is cofiltered by Proposition \ref{boundary.17}, replacing $Z'$,
we may assume that there exists a morphism $Z'\to Y'$ in $\SBl_{X'}^{n'}$.
The square
\[
\begin{tikzcd}
Z'\ar[d,"w"']\ar[r]&
Y'\ar[d,"u"]
\\
Z\ar[r,"v"]&
Y
\end{tikzcd}
\]
commutes since $(X'-\partial X')\times \A^{n'}$ is dense in $Z'$.
This shows the Ore condition.

For the right cancellability condition,
let $(u_1,f_1,\alpha_1),(u_2,f_2,\alpha_2)\colon (Y',X',[n'])\to (Y,X,[n])$ be two morphisms, and let $(v,\id,\id)\colon (Y'',X',[n'])\to (Y',X',[n'])$ be an admissible blow-up such that
\[
(u_1,f_1,\alpha_1)\circ (v,\id,\id)
=
(u_2,f_2,\alpha_2)\circ (v,\id,\id).
\]
Then we have $f_1=f_2$ and $\alpha_1=\alpha_2$,
and the square
\[
\begin{tikzcd}[column sep=huge]
(X'-\partial X')\times \A^{n'}\ar[d]\ar[r,"(f-\partial f)\times \A^{\alpha}"]&
(X-\partial X)\times \A^n\ar[d]
\\
Y'\ar[r,"u_i"]&
Y
\end{tikzcd}
\]
commutes for $i=1,2$.
Since $(X'-\partial X')\times \A^{n'}$ is dense in $Y'$,
we have $u_1=u_2$.
This shows the right cancellability condition.
\end{proof}

\begin{prop}
\label{boundary.64}
The functor
\[
\SBl(\F_1)[\SAdm^{-1}]
\to
\SmlSm/\F_1\times \bDelta
\]
induced by the forgetful functor $\SBl(\F_1)\to \SmlSm/\F_1\times \bDelta$
is an equivalence of categories.
\end{prop}
\begin{proof}
The inclusion functor $\SmlSm/\F_1\times \bDelta\to \SBl(\F_1)$ given by
\[
(X,[n])\mapsto (X\times \square^n,X,[n])
\]
induces an essentially surjective functor
\[
\SmlSm/\F_1\times \bDelta\to \SBl(\F_1)[\SAdm^{-1}].
\]
It suffices to show that this is fully faithful.
Let $(X,[n]),(X',[n'])\in \SmlSm/\F_1\times \bDelta$,
and let $(Y',X',[n'])\to (X'\times \square^{n'},X',n')$ be an admissible blow-up.
Then we have
\[
\Hom_{\SBl(\F_1)}((Y',X',[n']),(X\times \square^n,X,[n]))
\simeq
\Hom_{\SmlSm/\F_1\times \bDelta}((X',[n']),(X,[n])),
\]
which finishes the proof.
\end{proof}

\begin{df}
\label{boundary.32}
Let $\cF$ be a presheaf of spectra on $\SmlSm/\F_1$. Consider the presheaf of spectra $\cF$ on $\SBl(\F_1)$ given by $\cF(Y,X,[n]):=\cF(Y)$.
Note that $\SAdm$ admits a calculus of right fractions by Proposition \ref{boundary.65}.
Argue as in \cite[Lemma 4.3]{loghomotopy} to obtain an equivalence of $\infty$-categories
\[
\PSh(\SBl(\F_1)[\SAdm^{-1}],\Sp)
\simeq
\PSh(\SBl(\F_1),\Sp)[\SAdm^{-1}].
\]
Let
\[
L_\SAdm
\colon
\PSh(\SBl(\F_1),\Sp)
\to
\PSh(\SBl(\F_1),\Sp)[\SAdm^{-1}]
\]
be the localization functor.
Together with Proposition \ref{boundary.64},
we obtain
\[
\widetilde{\Log}_\bullet \cF:=
L_\SAdm \cF\in \PSh(\SmlSm/\F_1\times \bDelta,\Sp).
\]
Let $\widetilde{\Log}_n\cF(X)$ denote its value at $X\in \SmlSm/\F_1$ and $n\in \bDelta$.
Then we have
\[
\widetilde{\Log}_n\cF(X)
\cong
\colim_{Y\in (\SBl_X^n)^\op}
\cF(Y)
\]
for $X\in \SmlSm/\F_1$ and integer $n\geq 0$.
We define
\begin{equation}
\label{boundary.32.1}
\widetilde{\Log} \cF
:=
\colim_{n\in \bDelta^\op}\widetilde{\Log}_n\cF
\cong
\colim_{n\in \bDelta^\op}\colim_{Y\in (\SBl_X^n)^\op} \cF(Y).
\end{equation}
The construction of $\widetilde{\Log} \cF$ is natural in $\cF$,
so we obtain a functor
\[
\widetilde{\Log}
\colon
\PSh(\SmlSm/\F_1,\Sp)
\to
\PSh(\SmlSm/\F_1,\Sp)
\]
\end{df}

\begin{rmk}
\label{boundary.66}
For an integer $n\geq 0$,
let $\Sdiv^n(\F_1)$ be the category of pairs $(V,X)$ with $V\in \Sdiv_X^n$ and $X\in \SmlSm/S$ whose morphisms $(h,f)\colon (V',X')\to (V,X)$ are the commutative squares
\[
\begin{tikzcd}
V'\ar[d]\ar[r]&
V\ar[d]
\\
X'\times (\Gmlog)^n\ar[r]&
X\times (\Gmlog)^n.
\end{tikzcd}
\]
Let $\SAdm$ be the class of morphisms $(h,f)\colon (V',X')\to (V,X)$ in $\Sdiv_X^n$ such that $f$ is an isomorphism.
Then argue as in Proposition \ref{boundary.64} to obtain an equivalence of categories
\[
\Sdiv^n(\F_1)[\SAdm^{-1}]
\simeq
\SmlSm/\F_1.
\]

Let $\cF$ be a presheaf of spectra on $\SmlSm/\F_1$.
Consider the full subcategory $\SBl^n(\F_1)$ of $\SBl(\F_1)$ whose entries for $\bDelta$ are $[n]$,
and consider the presheaf of spectra $\cF$ on $\SBl^n(\F_1)$ given by $\cF(Y,X,[n]):=\cF(Y)$.
Argue as in Proposition \ref{boundary.63} to show that $\SBl^n(\F_1)$ is a full subcategory of $\Sdiv^n(\F_1)$.
Let $\cF^\Sdiv$ be its left Kan extension to $\Sdiv^n(\F_1)$.
Argue as in Definition \ref{boundary.32} also to obtain
\[
L_\SAdm \cF^\Sdiv\in \PSh(\SmlSm/\F_1,\Sp),
\]
which agrees with $\widetilde{\log}_n\cF$ since the functor $\SBl_X^n\to \Sdiv_X^n$ is cofinal by Proposition \ref{boundary.31}.
Hence we have an alternative description
\begin{equation}
\label{boundary.66.1}
\widetilde{\log}_n\cF(X)
\cong
\colim_{Y\in (\Sdiv_X^n)^\op}\cF^\Sdiv(Y).
\end{equation}
This is sometimes useful because for every morphism $f\colon X'\to X$ in $\SmlSm/\F_1$,
we have the explicit functor $f^*\colon \Sdiv_X^n\to \Sdiv_{X'}^n$ in Construction \ref{boundary.35}.
\end{rmk}

\begin{df}
\label{boundary.33}
The \emph{logarithm functor} is  
\[
\Log:=\widetilde{\Log}\omega^\sharp
\colon \PSh(\Sm/\F_1,\Sp)
\to
\PSh(\SmlSm/\F_1,\Sp),
\]
where $\omega^\sharp$ is the left Kan extension functor
\[
\omega^\sharp
\colon
\PSh(\Sm/\F_1,\Sp)
\to
\PSh(\SmlSm/\F_1,\Sp).
\]
Explicitly,
for a presheaf $\cF$ of spectra on $\Sm/\F_1$ and $X\in \SmlSm/\F_1$,
\eqref{boundary.32.1} yields a natural isomorphism
\begin{equation}
\label{boundary.33.1}
\Log \cF(X)
\cong
\colim_{n\in \bDelta^\op}
\colim_{Y\in (\SBl_X^n)^\op}
\cF(\ul{Y}).
\end{equation}
Observe also that we have a natural morphism
\[
\cF(X) \to \Log \cF(X).
\]
\end{df}

Our next purpose is to define the cubical version of the logarithm functor.

\begin{prop}
\label{boundary.30}
Let $X\in \SmlSm/\F_1$.
For all integers $n\geq 1$ and $1\leq i\leq n$,
there exists a functor $\mu_i^*\colon \Sdiv_X^n \to \Sdiv_X^{n+1}$ satisfying the following property:
For $V\in \Sdiv_X^n$,
there exists a commutative square
\[
\begin{tikzcd}
(X-\partial X)\times \G_m^{n+1}\ar[r,"\id \times \G_m^{\mu_i}"]\ar[d]&
(X-\partial X)\times \G_m^n\ar[d]
\\
\mu_i^*(V)\ar[r]&
V,
\end{tikzcd}
\]
where $\G_m^{\mu_i}$ is the $i$th multiplication
\[
\G_m^{i-1}\times \G_m^2 \times \G_m^{n-i}
\xrightarrow{\id \times \mu\times \id}
\G_m^{i-1}\times \G_m \times \G_m^{n-i}
\]
induced by the multiplication $\mu\colon \G_m^2\to \G_m$.
\end{prop}
\begin{proof}

We refer to \textup{Construction \ref{logSH.17}} for $\mu'\colon \Bl_{(0,\infty)+(\infty,0)}(\square^2)\to \square$.
Consider $D\in \SmlSm/\F_1$ such that $\ul{D}:=\ul{\Bl_{(0,\infty)+(\infty,0)}(\square^2)}$ and $D-\partial D:=\G_m^2$.
We have the morphism $\mu''\colon D\to \Gmlog$ extending the multiplication morphism $\G_m^2\to \G_m$,
and we have $\ul{\mu''}=\ul{\mu'}$.
Let $g\colon V\to X\times (\Gmlog)^n$ be the dividing cover corresponding to $Y\to X\times \square^n$,
and let $g'\colon V'\to X\times (\Gmlog)^{n+1}$ be its pullback along the morphism
\[
\id \times \mu''\times \id \colon (\Gmlog)^{i-1}\times D \times (\Gmlog)^{n-1-i} \to (\Gmlog)^{i-1}\times \Gmlog \times (\Gmlog)^{n-1-i}.
\]
Observe that $g'$ is a dividing cover by Proposition \ref{logmonoid.13} since $\id \times \mu''\times \id$ is a morphism of the form $\T_\Delta\to \T_\Sigma$ for some subdivision of fans $\Delta\to \Sigma$.
We have
\[
g^{-1}((X-\partial X)\times \A_\N^n)\cong (X-\partial X)\times \A_\N^n,
\]
which implies
\[
g'^{-1}((X-\partial X)\times \A_\N^{n+1})\cong (X-\partial X)\times \A_\N^{n+1}.
\]
Hence we have $V'\in \Sdiv_X^{n+1}$.
The construction of $V'$ is natural, so we have a functor $\mu_i^*\colon \Sdiv_X^n \to \Sdiv_X^{n+1}$ such that $\mu_i^*(V):=V'$,
and $\mu_i^*$ satisfies the desired property.
\end{proof}

\begin{df}
Let $\SBl^\cube(\F_1)$ be the category of triples $(Y, X,\ul{\unit}^n)$ with $Y\in \SBl_X^n$, $X\in \SmlSm/\F_1$, and $n\in \N$ whose morphisms are of the form
\[
(u,f,\alpha)\colon (Y',X',\ul{\unit}^{n'})\to (Y,X,\ul{\unit}^n)
\]
such that the square
\[
\begin{tikzcd}[column sep=huge]
(X'-\partial X')\times \A^{n'}\ar[d]\ar[r,"(f-\partial f)\times \A^{\alpha}"]&
(X-\partial X)\times \A^n\ar[d]
\\
Y'\ar[r,"u"]&
Y
\end{tikzcd}
\]
commutes.

A morphism $(u,f,\alpha)\colon (Y',X',\ul{\unit}^{n'})\to (Y,X,\ul{\unit}^n)$ is an \emph{admissible blow-up} if $X'\to X$ and $\ul{\unit}^{n'}\to \ul{\unit}^{n}$ are isomorphisms.
Let $\SAdm$ denote the class of admissible blow-ups.
\end{df}

\begin{df}
\label{boundary.27}
Argue as in Propositions \ref{boundary.64} and \ref{boundary.65} but use Proposition \ref{boundary.30} also to obtain an equivalence of categories
\[
\SBl^\cube(\F_1)[\SAdm^{-1}]
\simeq
\SmlSm/\F_1\times \ECube.
\]
Let $\cF$ be a presheaf of complexes on $\SmlSm/\F_1$.
Consider the presheaf of complexes $\cF$ on $\SBl^\cube(\F_1)$ given by $\cF(Y,X,\ul{\unit}^n):=\cF(Y)$ for $Y\in \SBl_X^n$, $X\in \SmlSm/\F_1$, and $n\in \N$.
Argue as in Definition \ref{boundary.32} to obtain
\[
\logCLog_\bullet \cF:=
L_\SAdm \cF\in \PSh(\SmlSm/\F_1\times \ECube,\Sp).
\]
Observe that we have a natural isomorphism
\begin{equation}
\label{boundary.27.2}
\logCLog_n\cF(X)
\cong
\widetilde{\Log}_n\cF(X)
\end{equation}
for all $X\in \SmlSm/\F_1$ and $n\in \N$.
Let $\logCLog\cF$ be the presheaf of complexes on $\SmlSm/\F_1$ given by
\[
\logCLog\cF(X)
:=
\lvert \logCLog_\bullet\cF(X)\rvert_\mathrm{cube}
\cong
\big\lvert \colim_{Y\in (\SBl_X^\bullet)^\op}\cF(Y)\big\rvert_\mathrm{cube}
\]
for $X\in \SmlSm/\F_1$,
where $\lvert - \rvert_\mathrm{cube}$ is the cubical geometric realization in Construction \ref{logSH.14}.
The construction of $\logCLog \cF$ is natural in $\cF$,
so we obtain a functor
\[
\logCLog
\colon
\PSh(\SmlSm/\F_1,\rD(\Z))
\to
\PSh(\SmlSm/\F_1,\rD(\Z)).
\]
\end{df}

\begin{df}
The \emph{cubical logarithm functor} is
\[
\CLog:=\logCLog\omega^\sharp
\colon
\PSh(\Sm/\F_1,\rD(\Z))
\to
\PSh(\SmlSm/\F_1,\rD(\Z)).
\]
Explicitly,
for a presheaf $\cF$ of complexes on $\Sm/\F_1$ and $X\in \SmlSm/\F_1$,
we have
\begin{equation}
\label{boundary.27.1}
\CLog \cF(X)
:=
\lvert \Log_\bullet\cF(X)\rvert_\mathrm{cube}
=
\big\lvert \colim_{Y\in (\SBl_X^\bullet)^\op}\cF(\ul{Y})\big\rvert_\mathrm{cube}
\end{equation}
using Proposition \ref{boundary.31}.
Observe also that we have a natural morphism
\[
\cF(X) \to \CLog \cF(X).
\]
\end{df}

\begin{exm}
\label{boundary.60}
Assume $X=\square^r$ with an integer $r\geq 0$.
Let $1\leq i\leq n$ be integers,
and consider the morphisms $\delta_{i,0},\delta_{i,1}\colon \ul{\unit}^{n-1}\to \ul{\unit}^n$.

Let $Y\in \SBl_X^n$,
and consider the fan $\Sigma$ corresponding to $Y$.
We have the functor $\delta_{i,0}^*\colon \SBl_X^n\to \SBl_X^{n-1}$ as in Proposition \ref{boundary.21}.
By Lemma \ref{image.8},
$\ul{\delta_{i.0}^*(Y)}$ corresponds to the fan $V(e_{r+i})$ if $\{e_1,\ldots,e_{r+n}\}$ denotes the standard coordinates in $\square^{r+n}$.

On the other hand,
we have the functor $\delta_{i,1}^*\colon \Sdiv_X^n\to \Sdiv_X^{n-1}$ as in Construction \ref{boundary.35}.
Consider $V\in \Sdiv_X^n$ corresponding to $Y$.
Then $\ul{\delta_{i.1}^*(V)}$ corresponds to the fan that is the restriction of $\Sigma$ to the lattice $\Z^{r+i-1}\times 0 \times \Z^{n-i}$ since $(\P^1)^{r+i-1}\times \{1\}\times (\P^1)^n$ can be identified with the restriction of the fan $(\P^1)^{n+r}$ to the lattice $\Z^{r+i-1}\times 0 \times \Z^{n-i}$.
It follows that we have $\delta_{i,1}^*(V)\in \SmlSm/\F_1$,
so we have an induced functor $\delta_{i,1}^*\colon \SBl_X^n \to \SBl_X^{n-1}$.
\end{exm}

In the remaining part of this section,
we discuss several properties of $\Log \cF$ and $\CLog\cF$.

\begin{prop}
\label{boundary.22}
Let $\cF$ be a presheaf of spectra on $\SmlSm/\F_1$.
Then $\widetilde{\Log}_n \cF$ is dividing invariant for every integer $n\geq 0$.
\end{prop}
\begin{proof}
Let $X'\to X$ be a dividing cover in $\SmlSm/\F_1$.
Then we have the functor $f_!\colon \Sdiv_{X'}^n\to \Sdiv_X^n$ given as follows.
Let $V'\in \Sdiv_{X'}^n$.
Then the composite $V'\to X'\times (\Gmlog)^n\to X\times (\Gmlog)^n$ is a dividing cover,
and hence we have $V'\in \Sdiv_X^n$.
This construction is natural, so we obtain $f_!$.

Together with the description of the functor $f^*\colon \Sdiv_X^n\to \Sdiv_{X'}^n$ in Construction \ref{boundary.35},
we see that $f_!$ is left adjoint to $f^*$.
Since $X'\cong X'\times_X X'$,
we also have $\id \cong f^*f_!$. 
Using the counit $f_!f^*\to \id$,
we see that $f_!$ is cofinal.
Consider the morphisms
\[
\colim_{Y'\in (\Sdiv_{X'}^n)^{\op}}\cF^\Sdiv(Y')
\xrightarrow{v}
\colim_{Y\in (\Sdiv_X^n)^{\op}}\cF^\Sdiv(Y)
\xrightarrow{u}
\colim_{Y'\in (\Sdiv_{X'}^n)^{\op}}\cF^\Sdiv(Y')
\]
induced by $f_!$ and $f^*$.
Since $f_!$ is cofinal, $v$ is an isomorphism.
Furthermore,
$uv$ is obtained by taking the colimits of the identity morphisms,
so $uv$ is an isomorphism.
It follows that $u$ is an isomorphism,
which implies the claim together with \eqref{boundary.66.1}.
\end{proof}

\begin{lem}
\label{boundary.18}
Let $j\colon X'\to X$ be an open immersion in $\SmlSm/\F_1$.
Then the functor $j^*\colon \Sdiv_{X}^n\to \Sdiv_{X'}^n$ is essentially surjective and hence cofinal for every integer $n\geq 0$.
\end{lem}
\begin{proof}
Let $V'\in \Sdiv_{X'}^n$.
We have the induced morphism $V'^\sharp \to X'^\sharp \times (\Gmlog)^n$,
see Construction \ref{logmonoid.16} for $X'^\sharp$ and $V'^\sharp$.
For the local case $X=\A_\N^d \times \G_m^e \times \A^m$ and $X'=\A_\N^{d'} \times \G_m^{e'}\times \A^{m'}$ with integers $0\leq d'\leq d$, $0\leq e\leq e'$, and $0\leq m'\leq m$ such that $d+e+m=d'+e'+m'$,
we have $X^\sharp \cong \A_\N^d \times \G_m^{e+m}$ and $X'^\sharp \cong \A_\N^{d'}\times \G_m^{e'+m'}$.
There exists a dividing cover $W'\to \A_\N^{d'} \times (\Gmlog)^n$ such that $V'\cong W'\times \G_m^{e'+m'}$.
Since $(X'-\partial X')\times \A_\N^n$ is an open subscheme of $V'$,
$\G_m^{d'}\times \A_\N^n$ is an open subscheme of $W'$.
Consider the induced dividing cover $V:=\A_\N^{d-d'}\times \G_m^e \times \A^m\times W'\to X\times (\Gmlog)^n$.
Then $(X-\partial X)\times \A_\N^n$ is an open subscheme of $V$,
so we have $V\in \Sdiv_X^n$.
Furthermore, we have $j^*Y\cong Y'$.

The above local construction of $V$ is natural, so we can glue the local constructions to construct $V$ for general $X$ such that $j^*V\cong V'$.
\end{proof}

\begin{lem}
\label{boundary.36}
Let $p\colon X\times \A^1\to X$ be the projection with $X\in \SmlSm/\F_1$.
Then the functor $p^*\colon \Sdiv_X^n\to \Sdiv_{X\times \A^1}^n$ is an equivalence of categories for every integer $n\geq 0$.
\end{lem}
\begin{proof}
It suffices to show that the categories of dividing covers of $X\times (\Gmlog)^n$ and $X\times \A^1\times (\Gmlog)^n$ are equivalent.
In the local case $X=\A_P\times \A^m$ with a toric monoid $P$ and integer $m\geq 0$, any dividing cover of $X\times (\Gmlog)^n$ (resp.\ $X\times \A^1\times (\Gmlog)^n$) is of the form $\T_\Sigma\times \A^m$ (resp.\ $\T_\Sigma\times \A^1\times \A^m$) for a subdivision $\Sigma$ of $\A_P \times (\P^1)^n$.
Hence the claim holds.
In particular,
for every dividing cover $W\to X\times \A^1\times (\Gmlog)^n$,
we can naturally construct a dividing cover $V\to X\times (\Gmlog)^n$ such that $W\cong V\times \A^1$.

For the general case, glue the local constructions to conclude.
\end{proof}

\begin{prop}
\label{boundary.1}
Let $\cF$ be a presheaf of spectra (resp.\ complexes) on $\SmlSm/\F_1$.
If $\cF$ is $\A^1$-invariant,
then $\widetilde{\Log} \cF$ (resp.\ $\logCLog \cF$) is $\A^1$-invariant.
If $\cF$ is a Zariski sheaf,
then $\widetilde{\Log} \cF$ (resp.\ $\logCLog \cF$) is a Zariski sheaf.
\end{prop}
\begin{proof}
Assume that $\cF$ is $\A^1$-invariant.
Then \eqref{boundary.66.1} and Lemma \ref{boundary.36} implies that  $\widetilde{\Log}$ (resp.\ $\logCLog$) is $\A^1$-invariant.

Assume that $\cF$ is a Zariski sheaf.
Let
\[
Q:=
\begin{tikzcd}
W\ar[d]\ar[r]&
V\ar[d]
\\
U\ar[r]&
X
\end{tikzcd}
\]
be a Zariski distinguished square in $\SmlSm/\F_1$.
To show that $\widetilde{\Log}$ (resp.\ $\logCLog$) is a Zariski sheaf,
we need to show that $\widetilde{\Log}(Q)$ (resp.\ $\logCLog(Q)$) is cocartesian.
For this,
it suffices to show that the square $\widetilde{\Log}_n(Q)$ is cartesian for every integer $n\geq 0$ by \eqref{boundary.27.2}.
This square can be identified with the induced square
\[
\begin{tikzcd}
\colim_{Y\in (\Sdiv_X^n)^\op}
\cF^\Sdiv(Y)\ar[d]\ar[r]&
\colim_{Y\in (\Sdiv_X^n)^\op}
\cF^\Sdiv(Y\times_X U)\ar[d]
\\
\colim_{Y\in (\Sdiv_X^n)^\op}
\cF^\Sdiv(Y\times_X V)\ar[r]&
\colim_{Y\in (\Sdiv_X^n)^\op}
\cF^\Sdiv(Y\times_X W)
\end{tikzcd}
\]
by \eqref{boundary.66.1} and Lemma \ref{boundary.18},
which is cocartesian since $\cF$ is a Zariski sheaf.
\end{proof}

As a consequence of Propositions \ref{boundary.22} and \ref{boundary.1},
we obtain an induced functor
\begin{equation}
\Log
\colon \SH_{S^1}(\F_1)
\to
\Sh_\Zar(\SmlSm/\F_1,\Sp)[(\A^1)^{-1},\divi^{-1}].
\end{equation}

A \emph{smooth toric monoid $P$} is an fs monoid $P$ such that $\ol{P}:=P/P^*$ is free and $P^\gp$ is torsion free.
In this case, there is an isomorphism $P\cong \N^r\times \Z^s$ for some integers $r,s\geq 0$.
Let $\Adm_{\A_P\times \square^n}$ denote the category of admissible blow-ups of $\A_P\times \square^n$ in $\SmlSm/\F_1$.

\begin{lem}
\label{boundary.23}
Let $P$ be a smooth toric monoid.
Then for every integer $n\geq 0$,
there is an equivalence of categories
\[
\SBl_{\A_P}^n
\simeq
\Adm_{\A_P\times \square^n}.
\]
\end{lem}
\begin{proof}
Let $f\colon Y\to \A_P\times \square^n$ be an admissible blow-up.
Then $\ul{f}$ is proper birational, so $\ul{f}$ is associated with a subdivision of fans $\Delta\to \Sigma$ as noted in Example \ref{msch.13}.
We have the dividing cover $g\colon \T_\Delta\to \T_\Sigma\cong \A_P\times (\Gmlog)^n$ such that $\ul{g}$ is identified with $\ul{f}$.
Hence the inclusion $\SBl_{\A_P}^n\subset \Adm_{\A_P\times \square^n}$ is an equivalence of categories.
\end{proof}

\begin{lem}
\label{boundary.24}
Let $P$ be a smooth toric monoid,
and let $n\in \N$.
If $\cF$ is a presheaf of complexes on $\SmlSm/\F_1$,
then we have natural isomorphisms
\begin{gather*}
\Log \cF(\A_P\times \A^n)
\cong
\Sing^\square L_{\Adm}\ul{\Hom}(\Z[\A^n],\cF)(\A_P),
\\
\CLog \cF(\A_P\times \A^n)
\cong
\CSing^\square L_{\Adm}\ul{\Hom}(\Z[\A^n],\cF)(\A_P),
\end{gather*}
where $\ul{\Hom}$ denotes the internal hom.
\end{lem}
\begin{proof}
Compare \eqref{logSH.15.1} and \eqref{logSH.15.2} with \eqref{boundary.33.1} and \eqref{boundary.27.1},
and use Lemma \ref{boundary.23}.
\end{proof}

\begin{df}
\label{boundarify.1}
A presheaf $\cF$ of spectra on $\Sm/\F_1$ is \emph{logarithmic} if
\[
\cF(X)\cong \Log \cF(X\times \square^n)
\]
for all $X\in \Sm/\F_1$ and integer $n\geq 0$.

A presheaf $\cF$ of complexes on $\Sm/\F_1$ is \emph{cubically logarithmic} if
\[
\cF(X) \cong \CLog\cF(X\times \square^n)
\]
for all $X\in \Sm/\F_1$ and integer $n\geq 0$.
\end{df}

\begin{prop}
\label{boundary.26}
Let $\cF$ be a Zariski sheaf of complexes on $\Sm/\F_1$.
Then we have a natural isomorphism of presheaves of complexes
\[
\Log \cF
\cong
\CLog \cF.
\]
In particular,
$\cF$ is logarithmic if and only if it is cubically logarithmic.
\end{prop}
\begin{proof}
We need to show $\Log \cF(X)\cong \CLog\cF(X)$ for $X\in \SmlSm/\F_1$.
By Proposition \ref{boundary.1},
we can work Zariski locally on $X$.
Hence we may assume $X=\A_P\times \A^n$ for some smooth toric monoid $P$ and integer $n\geq 0$.
Use \eqref{omega.14.2} and Lemma \ref{boundary.24} to conclude.
\end{proof}

\section{Logarithm functor: Case of schemes}

In this section,
we define the logarithm functor $\Log$ for schemes by adapting Definition \ref{boundary.33}.
The functor $\Log$ makes a sheaf of spectra $\cF$ on $\Rg$ (resp.\ $\Sm/S$ with $S\in \Sch$) into a sheaf of spectra on $\RglRg$ (resp.\ $\SmlSm/S$).
We can view $\Log \cF$ as a natural logarithmic extension of $\cF$ if $\cF(X) \cong \Log \cF(X)$ for every $X\in \Rg$ (resp.\ $X\in \Sm/S$).

The following two definitions are direct analogs of Definition \ref{boundary.9} and \ref{boundary.48}.

\begin{df}
\label{boundary.49}
Let $X\in \RglRg$ (resp.\ $X\in \SmlSm/S$ with $S\in \Sch$).
For $n\in \N$,
let $\SBl_X^n$ be the category of admissible blow-ups $f\colon Y\to X\times \square^n$ in $\RglRg$ (resp.\ $\SmlSm/S$) such that there exists a commutative square
\begin{equation}
\label{boundary.49.1}
\begin{tikzcd}
V\ar[d,"g"']\ar[r]&
Y\ar[d,"f"]
\\
X\times (\Gmlog)^n\ar[r]&
X\times \square^n
\end{tikzcd}
\end{equation}
in $\RglRg$ (resp.\ $\SmlSm/S$)
with a dividing cover $g$ satisfying $\ul{g}=\ul{f}$,
where $X\times (\Gmlog)^n\to X\times \square^n$ is the canonical map whose underlying morphism of schemes is the identity.
Observe that every morphism in $\SBl_X^n$ is an admissible blow-up.
\end{df}

\begin{df}
\label{boundary.50}
For $X\in \RglRg$ (resp.\ $X\in \SmlSm/S$ with $S\in \Sch$) and $n\in \N$,
let $\Sdiv_X^n$ be the category of dividing covers $g\colon V\to X\times (\Gmlog)^n$ such that $g^{-1}(X\times \A_\N^n)\to X\times \A_\N^n$ is an isomorphism.
\end{df}

We need the following basic result since we often need to check that two morphisms of fs log schemes are equal.

\begin{prop}
\label{boundary.67}
Let $X\in \RglRg$ (resp.\ $X\in \SmlSm/S$ with $S\in \Sch$).
Then the log structure map $\cM_X\to \cO_X$ is injective.
Furthermore,
for every fs log scheme $Y$,
the map
\[
\Hom(X,Y)
\to
\Hom(\ul{X},\ul{Y}),
\text{ }
f\mapsto \ul{f}
\]
is injective.
\end{prop}
\begin{proof}
If $X\in \RglRg$ or $S$ is normal, then the log structure on $X$ is the compactifying log structure by \cite[Proposition III.1.7.3(4)]{Ogu},
so the log structure map $\cM_X\to \cO_X$ is injective by definition, and \cite[Proposition III.1.6.2]{Ogu} implies the other claim.

For general $S$,
let $f_1,f_2\colon X\rightrightarrows Y$ be two morphisms such that $\ul{f_1}=\ul{f_2}$.
We have the induced commutative square
\[
\begin{tikzcd}
\cM_Y\ar[d]\ar[r]&
\cO_Y\ar[d]
\\
f_{i*}\cM_X\ar[r]&
f_{i*}\cO_X
\end{tikzcd}
\]
for $i=1,2$. If $\cM_X\to \cO_X$ is injective, then the two morphisms $\cM_Y\to f_{1*}\cM_X$ and $\cM_Y\to f_{2*}\cM_X$ agree,
so we have $f_1=f_2$.
Hence it suffices to show that $\cM_X\to \cO_X$ is injective.
Equivalently, it suffices to show that $\cM_{X,x}\to \cO_{X,x}$ is injective for every point $x\in X$.

We set $A:=\cO_{X,x}$.
Let $P$ be a neat chat of $X$ at $x$ in the sense of \cite[Definition II.2.3.1(2)]{Ogu}.
Then the induced morphism $P\to \ol{\cM}_{X,x}$ is an isomorphism,
and $P\cong \N^d$ for some integer $d$.
Let $\alpha \colon P\to A$ be the log structure map.
Then we have $\alpha^{-1}(A^*)=0$ since $P\cong \ol{\cM}_{X,x}$.
The log structure $\cM_{X,x}$ associated with $P$ is $P\oplus_{\alpha^{-1}(A^*)}A^*\cong P\oplus A^*$, see \cite[(II.1.1.1)]{Ogu}.
Hence it suffices to show that the induced map $\eta\colon P\oplus A^* \to A$ is injective.
Let $e_1,\ldots,e_d$ be the standard coordinates in $\N^d$.
Observe that $\eta(e_i)$ is a nonzero divisor for each $i$ since $X-\partial X$ is dense in $X$.

Let $S'$ be the normalization of $S$,
and let $f\colon A\to B$ be the map of rings whose spectrum is $\Spec(\cO_{X,x})\times_S S'\to \Spec(\cO_{X,x})$.
Assume that
\[
\eta(u+a_1e_1+\cdots+a_de_d)=\eta(v+b_1e_1+\cdots+b_de_d)
\]
with $u,v\in A^*$ and $a_1,\ldots,a_d,b_1,\ldots,b_d\in \N$.
Then we have
\begin{equation}
\label{boundary.67.1}
f(\eta(u+a_1e_1+\cdots+a_de_d))
=
f(\eta(v+b_1e_1+\cdots+b_de_d)).
\end{equation}
The log structure on $X':=X\times_S S'$ is the compactifying log structure since $X'\in \SmlSm/S'$ and $S'$ is normal.
Hence the log structure map $\cM_{X'}\to \cO_{X'}$ is injective.
Together with \eqref{boundary.67.1},
we have $a_i=b_i$ for each $i$.
Since $\eta(e_i)$ is a nonzero divisor for each $i$,
we have $\eta(u)=\eta(v)$,
which implies $u=v$.
This shows that $\eta$ is injective.
\end{proof}

\begin{rmk}
A typical application of Proposition \ref{boundary.67} is as follows.
Let $X\in \RglRg$ (resp.\ $X\in \SmlSm/S$ with $S\in \Sch$),
and let $f_1,f_2\colon X\rightrightarrows Y$ be two morphisms of fs log schemes.
If $j\colon U\to X$ is a dense open immersion of fs log schemes such that $f_1j=f_2j$,
then we have $f_1=f_2$ by Proposition \ref{boundary.67}.
\end{rmk}

We are in a better situation for dividing covers as follows.

\begin{prop}
\label{boundary.68}
Let $X\in \RglRg$ (resp.\ $X\in \SmlSm/S$ with $S\in \Sch$).
Then for all dividing covers $Y,Y'\to X$,
the map
\[
\Hom_X(Y',Y)
\to
\Hom_{\ul{X}}(\ul{Y'},\ul{Y}),
\text{ }
f\mapsto \ul{f}
\]
is an isomorphism.
\end{prop}
\begin{proof}
Assume $X\in \RglRg$ or $X\in \SmlSm/S$ with normal $S$.
Then the log structure on $X$ is the compactifying log structure by \cite[Proposition III.1.7.3(4)]{Ogu},
so \cite[Proposition III.1.6.2]{Ogu} implies the claim, noting that we have $X-\partial X \cong Y-\partial Y\cong Y'-\partial Y'$.

For general $S$,
let $g\colon \ul{Y}\to \ul{Y'}$ be a morphism of schemes over $\ul{X}$.
By Proposition \ref{boundary.67},
it suffices to show that there exists a morphism of fs log schemes $f\colon Y'\to Y$ over $X$ such that $\ul{f}=g$.

Let $P\to Q,Q'$ be maps of fine monoids.
The pushout $Q\oplus_P^\mathrm{int} Q'$ in the category of integral monoids is the image of $Q\oplus_P^\mathrm{mon} Q'\to Q^\gp\oplus_{P^\gp} Q'^\gp$,
where $Q\oplus_P^\mathrm{mon} Q'$ denotes the pushout in the category of monoids.
Hence $\Spec(\Z[Q\oplus_P^\mathrm{int} Q'])$ is the closure of $\Spec(\Z[Q^\gp]\otimes_{\Spec[P^\gp]}\Spec[Q'^\gp])$ in  $\Spec(\Z[Q])\times_{\Spec(\Z[P])}\Spec(\Z[Q'])$.
Together with the local description of dividing covers in \cite[Lemma A.11.3]{BPO},
we see that the underlying scheme of the fiber product $Y\times_X^\mathrm{int} Y'$ in the category of fine log schemes is the closure of $X-\partial X\cong (Y-\partial Y)\times_{X-\partial X}(Y'-\partial Y')$ in $\ul{Y}\times_{\ul{X}}\ul{Y'}$.

Now, let $g\colon \ul{Y'}\to \ul{Y}$ be a morphism of schemes over $\ul{X}$.
Then we have the graph morphism $\ul{Y'}\to \ul{Y}\times_{\ul{X}}\ul{Y'}$, which is a closed immersion.
Since $Y'-\partial Y'$ is dense in $\ul{Y'}$,
we have $\ul{Y'}\cong \ul{Y\times_X^\mathrm{int}Y'}$.
Hence we only need to show that the induced map
\begin{equation}
\label{boundary.68.1}
\ol{\cM}_{Y',y}\to \ol{\cM}_{Y\times_X^\mathrm{int}Y',y}
\end{equation}
is an isomorphism for every point $y\in Y'$.

Let $s$ be the image of $y$ in $S$,
and let $k$ be the residue field at $s$.
Since \eqref{boundary.68.1} is invariant under the base change from $S$ to $\Spec(k)$,
it suffices to show $Y\times_X^\mathrm{int}Y'\cong Y'$ under the assumption $S=\Spec(k)$.

In this case,
since $S$ is normal,
we already know that there exists a morphism of fs log schemes $f\colon Y'\to Y$ over $X$ such that $\ul{f}=g$.
Since $Y\to X$ is a monomorphism,
the second projection $Y\times_X Y\to Y$ is an isomorphism,
and hence the projection $Y\times_X Y'\to Y'$ is an isomorphism.
It follows that the composite $\ol{\cM}_{Y',y}\to \ol{\cM}_{Y\times_X^\mathrm{int}Y',y}\to \ol{\cM}_{Y\times_X Y',y}$ is an isomorphism.
In particular,
the map $\ol{\cM}_{Y\times_X^\mathrm{int}Y',y}\to \ol{\cM}_{Y\times_X Y',y}$ is surjective.
To conclude,
observe that the fiber product $Y\times_X Y'$ in the category of fs log schemes is the saturation 
of $Y\times_X^\mathrm{int}Y'$ and hence the map $\ol{\cM}_{Y\times_X^\mathrm{int}Y',y}\to \ol{\cM}_{Y\times_X Y',y}$ is injective.
\end{proof}

\begin{prop}
\label{boundary.69}
Let $X\in \RglRg$ (resp.\ $X\in \SmlSm/S$ with $S\in \Sch$) and $n\in \N$.
For $Y,Y'\in \SBl_X^n$ and $V,V'\in \Sdiv_X^n$,
$\Hom_{\SBl_X^n}(Y',Y)$ and $\Hom_{\Sdiv_X^n}(V',V)$ have at most one element.
\end{prop}
\begin{proof}
This is a consequence of Proposition \ref{boundary.67}.
\end{proof}

\begin{lem}
\label{boundary.75}
Let $X'\to X$ be a morphism in $\RglRg$ (resp.\ $\SmlSm/S$ with $S\in \Sch$) and $n\in \N$.
For $Y\in \SBl_X^n$ and $Y'\in \SBl_{X'}^n$,
consider $V\in \Sdiv_X^n$ and $V'\in \Sdiv_{X'}^n$ obtained by \eqref{boundary.49.1}.
If $h\colon V'\to V$ is a morphism over $X\times (\Gmlog)^n$,
then there exists a morphism $u\colon Y'\to Y$ over $X\times \square^n$ such that the square
\begin{equation}
\label{boundary.75.1}
\begin{tikzcd}
V'\ar[d]\ar[r,"h"]&
V\ar[d]
\\
Y'\ar[r,"u"]&
Y
\end{tikzcd}
\end{equation}
commutes.
\end{lem}
\begin{proof}
We have the induced commutative diagram
\[
\begin{tikzcd}
\cM_Y\ar[r]&
\cM_V\ar[r]\ar[d]&
\cO_V\ar[d]
\\
h_*\cM_{Y'}\ar[r]&
h_*\cM_{V'}\ar[r]&
\cO_{V'}.
\end{tikzcd}
\]
We need to show that there exists a morphism $\cM_Y\to h_*\cM_{Y'}$ making the above diagram still commutative.
As a consequence of Proposition \ref{boundary.67},
the morphisms $\cM_Y\to \cM_V$ and $h_*\cM_{Y'}\to h_*\cM_{V'}$ are injective.
Hence it suffices to show that the image of $\cM_Y\to h_*\cM_{V'}$ is contained in the image of $h_*\cM_{Y'}\to h_*\cM_{V'}$.
Equivalently,
it suffices to show that the image of $\ol{\cM}_{Y,v}\to (h_*\ol{\cM}_{V'})_v$ is contained in the image of $(h_*\ol{\cM}_{Y'})_v\to (h_*\ol{\cM}_{V'})_v$ for every point $v$ of $V$.
Consider the image $s$ of $v$ in $S$,
and let $k$ be the residue field at $s$.
Since the maps $\ol{\cM}_{Y,v}\to (h_*\ol{\cM}_{V'})_v$ and $(h_*\ol{\cM}_{Y'})_v\to (h_*\ol{\cM}_{V'})_v$ invariant under the base change from $S$ to $\Spec(k)$,
we reduce to the case of $S=\Spec(k)$ for showing this.

In this case,
we obtain \eqref{boundary.75.1} by \cite[Proposition III.1.6.2]{Ogu} since the log structure on $V'$ is a compactifying log structure.
Hence the image of $\ol{\cM}_{Y,v}\to (h_*\ol{\cM}_{V'})_v$ is contained in the image of $(h_*\ol{\cM}_{Y'})_v\to (h_*\ol{\cM}_{V'})_v$.
\end{proof}

\begin{prop}
\label{boundary.70}
For $X\in \RglRg$ (resp.\ $X\in \SmlSm/S$ with $S\in \Sch$) and integer $n\geq 0$,
there is a fully faithful functor of categories
\[
\rho
\colon
\SBl_X^n
\to
\Sdiv_X^n
\]
such that $\ul{Y}=\ul{\rho(Y)}$ for $Y\in \SBl_X^n$.
\end{prop}
\begin{proof}
For $Y\in \SBl_X^n$,
consider $V$ obtained by \eqref{boundary.49.1},
and we set $\rho(Y):=V$.
If $Y'\to Y$ is a morphism in $\SBl_X^n$,
then Proposition \ref{boundary.68} yields a morphism $\rho(Y')\to \rho(Y)$.

Conversely,
let $h\colon V':=\rho(Y')\to V:=\rho(Y)$ be a morphism in $\Sdiv_X^n$.
By Lemma \ref{boundary.75},
we obtain a commutative square of the form \eqref{boundary.75.1}.
In particular,
$\Hom_{\SBl_X^n}(Y',Y)$ is nonempty if and only if $\Hom_{\Sdiv_X^n}(\rho(Y'),\rho(Y))$ is nonempty.

Together with Proposition \ref{boundary.69},
we see that for every morphism $f\colon Y'\to Y$ we can uniquely assign $\rho(f)$ and $\rho$ satisfies the functor axioms.
\end{proof}

We refer to Definition \ref{logmonoid.21} for the notion of strict morphisms $X\to X_0$ when $X\in \lSch$ and $X_0\in \lSm/\F_1$.

\begin{df}
\label{boundary.71}
Let $\RglRg'$ ($\SmlSm'/S$ with $S\in \Sch$) be the full subcategory of $\RglRg$ (resp.\ $\SmlSm/S$) spanned by those $X$ admitting a strict morphism $X\to X_0$ with $X_0\in \SmlSm/\F_1$.
\end{df}

\begin{df}
For $X\in \lSch$,
a \emph{fan chart of $X$} is a fan $\Sigma$ with a strict morphism $X\to \T_\Sigma$.
\end{df}

\begin{prop}
\label{boundary.40}
Let $X\in \lSch$ with a fan chart $\Sigma$,
and let $\Sigma'$ be a subfan of $\Sigma$.
Then for every log \'etale monomorphism $f\colon Y\to X$ such that the induced morphism $Y\times_{\T_\Sigma}\T_{\Sigma'}\to X\times_{\T_\Sigma}\T_{\Sigma'}$ is an open immersion,
there exists a subdivision of fans $\Delta\to \Sigma$ such that $\Delta$ contains $\Sigma'$ as a subfan and the induced morphism
\[
f'\colon
X'\times_{\T_\Sigma}\T_{\Delta}
\to
X\times_{\T_\Sigma}\T_{\Delta}
\]
is an open immersion.
If we further assume that $f$ is a dividing cover,
then $f'$ is an isomorphism.
\end{prop}
\begin{proof}
This is a slight improvement of \cite[Proposition A.11.5]{BPO}.

Assume that we have shown the claim for log \'etale monomorphisms.
If $f$ is a dividing cover,
then $f'$ is surjective by \cite[Proposition A.11.6]{BPO} and proper,
so $f'$ is an isomorphism.
Hence it suffices to show the claim for log \'etale monomorphisms.

Let $\{Y_1,\ldots,Y_n\}$ be a Zariski covering of $Y$ such that we can find a desired fan $\Delta_i$ for each $Y_i\to X$.
Then
\[
\Delta:=\{\delta_1\cap \cdots \cap \delta_n:\delta_1\in \Delta_1,\ldots,\delta_n\in \Delta_n\}
\]
is a desired fan for $Y\to X$.
Hence we can work Zariski locally on $Y$.
By \cite[Lemma A.11.3]{BPO},
we may assume that there exists a cone $\sigma$ contained in the support $\lvert \Sigma \rvert$ such that $Y$ is an open subscheme of $X\times_{\T_\Sigma}\T_{\langle \sigma\rangle }$,
where $\langle \sigma \rangle$ denotes the fan with the single largest cone $\sigma$.
We may also assume that the dual monoid $\sigma^\vee$ is a neat chart of $Y$ at a point $y\in Y$.

Since the induced morphism $Y\times_{\T_\Sigma}\T_{\Sigma'}\to X\times_{\T_\Sigma}\T_{\Sigma'}$ is an open immersion,
the projection
$\T_{\langle \sigma \rangle}\times_{\T_\Sigma} \T_{\Sigma'}\to \T_{\Sigma'}$ is an open immersion too.
It follows that the fan
\[
\{\tau\cap \sigma':\tau\in \langle \sigma \rangle
,\sigma'\in \Sigma'\}
\]
is a subfan of $\Sigma'$.
Hence the union of fans $\langle \sigma \rangle \cup \Sigma'$ is a fan.
By \cite[p.\ 18]{TOda},
there exists a subdivision $\Delta$ of $\Sigma$ containing $\langle \sigma \rangle \cup \Sigma'$ as a subfan,
which is a desired fan.
\end{proof}

The following result enables us to utilize the arguments in \S \ref{boundary} to schemes.

\begin{prop}
\label{boundary.44}
Let $X\in \RglRg'$ (resp.\ $X\in \SmlSm'/S$ with $S\in \Sch$),
and let $X\to X_0$ be a strict morphism with $X_0\in \SmlSm/\F_1$.
Then for every $V\in \Sdiv_X^n$ and integer $n\geq 0$,
there exists $V_0\in \Sdiv_{X_0}^n$ and a morphism $X\times_{X_0}V_0\to V$ in $\Sdiv_X^n$.
\end{prop}
\begin{proof}
Let $U_0$ be an open subscheme of $X_0$ of the form $\A_P\times \A^m$ for a smooth toric monoid $P$ and integer $m\geq 0$,
and let $U:=X\times_{X_0} U_0$.
Then $U$ has a chart $P$.
We set $W:=U\times_X V$.
Observe that the induced morphism
\[
W\times_{U_0 \times (\Gmlog)^n}((U_0-\partial U_0) \times \A_\N^n)
\to
(U\times (\Gmlog)^n) \times_{U_0 \times (\Gmlog)^n}((U_0-\partial U_0) \times \A_\N^n)
\]
is an isomorphism.
Proposition \ref{boundary.40} implies that
there exists a dividing cover $W_0\to U_0\times (\Gmlog)^n$ such that the induced morphism
\[
q\colon
W\times_{U_0\times (\Gmlog)^n}W_0\to (U\times (\Gmlog)^n)\times_{U_0\times (\Gmlog)^n} W_0
\]
and the projection
\[
W_0\times_{U_0\times (\Gmlog)^n}((U_0-\partial U_0)\times \A_\N^n)\to (U_0-\partial U_0)\times \A_\N^n
\]
are isomorphisms.
In particular,
we have $W_0\in \Sdiv_{U_0}^n$.
Compose $q^{-1}$ with the projection $W\times_{U_0\times (\Gmlog)^n}W_0\to W$ to obtain
\begin{equation}
\label{boundary.44.1}
U\times_{U_0} W_0\to W.
\end{equation}

Now, let $\{U_{01},\ldots,U_{0r}\}$ be a Zariski covering of $X_0$ such that each $U_{0i}$ is of the form $\A_{P_i}\times \A^{m_i}$ for some smooth toric monoid $P_i$ and integer $m_i\geq 0$.
Construct $W_{0i}\in \Sdiv_{U_{0i}}^n$ for each $U_{0i}$ as the above paragraph.
Then we have a morphism $X\times_{X_0} W_{0i}\to V\times_{X_0} U_{0i}$ from \eqref{boundary.44.1}.
By Lemma \ref{boundary.18}, there exists $V_{0i}\in \Sdiv_{X_0}^n$ with a morphism $V_{0i}\times_{X_0}U_{0i}\to W_{0i}$ in $\Sdiv_{U_{0i}}^n$.
Since $\Sdiv_{X_0}^n$ is cofiltered by Proposition \ref{boundary.17},
there exists $V_0\in \Sdiv_{X_0}^n$ with morphisms $V_0\to V_{01},\ldots,V_{0r}$ in $\Sdiv_{X_0}^n$.
Together with Proposition \ref{boundary.67},
we can glue the induced morphisms
\[
X\times_{X_0} V_0\times_{X_0} U_{0i}\to V\times_{X_0}U_{0i}
\]
for all $i$ to obtain a morphism $X\times_{X_0} V_0\to V$ in $\Sdiv_X^n$.
\end{proof}

Using this result,
we obtain some properties of $\SBl_X^n$ and $\Sdiv_X^n$ as follows.

\begin{prop}
\label{boundary.41}
Let $X\in \RglRg'$ (resp.\ $X\in \SmlSm'/S$ with $S\in \Sch$).
Then the categories $\SBl_X^n$ and $\Sdiv_X^n$ are cofiltered for every integer $n\geq 0$.
\end{prop}
\begin{proof}
Let $X\to X_0$ be a strict morphism with $X_0\in \SmlSm/\F_1$.
Since $\Sdiv_{X_0}^n$ is cofiltered by Proposition \ref{boundary.17},
Propositions \ref{boundary.29} and \ref{boundary.44} imply that $\SBl_X^n$ and $\Sdiv_X^n$ are connected.
To conclude,
observe that there exists at most one morphism $Y'\to Y$ for all $Y,Y'\in \Sdiv_X^n$ by Proposition \ref{boundary.69}.
\end{proof}

\begin{prop}
\label{boundary.45}
Let $X\in \RglRg'$ (resp.\ $X\in \SmlSm'/S$ with $S\in \Sch$),
and let $X\to X_0$ be a strict morphism with $X_0\in \SmlSm/\F_1$.
Then the induced functors $\SBl_{X_0}^n\to \SBl_X^n$, $\Sdiv_{X_0}^n\to \Sdiv_{X}^n$, and $\SBl_X^n\to \Sdiv_X^n$ are cofinal for every $n\in \N$.
In particular,
for every presheaf of spectra $\cF$ on $\RglRg'$ (resp.\ $\SmlSm'/S$),
we have
\begin{equation}
\label{boundary.45.1}
\colim_{Y_0\in (\SBl_{X_0}^n)^\op}\cF(\ul{X}\times_{\ul{X_0}}\ul{Y_0})
\cong
\colim_{Y\in (\SBl_X^n)^\op}\cF(\ul{Y}).
\end{equation}
\end{prop}
\begin{proof}
The categories $\SBl_{X_0}^n$, $\SBl_X^n$, $\Sdiv_{X_0}^n$, and $\Sdiv_X^n$ are cofiltered by Propositions \ref{boundary.17} and \ref{boundary.41}.
Use Propositions \ref{boundary.29} and \ref{boundary.44} to show that the functors are cofinal.
The last claim is a consequence of this.
\end{proof}

\begin{prop}
\label{boundary.46}
Let $Y\to X$ be a strict morphism in $\RglRg'$ (resp.\ $\SmlSm'/S$ with $S\in \Sch$).
Then the induced functor $\Sdiv_X^n \to \Sdiv_Y^n$ is cofinal for every $n\in \N$.
\end{prop}
\begin{proof}
Let $X\to X_0$ be a strict morphism with $X_0\in \SmlSm/\F_1$.
Then the composite morphism $Y\to X_0$ is strict too.
Proposition \ref{boundary.45} finishes the proof.
\end{proof}

Next, we establish analogs of Construction \ref{boundary.35}, Definition \ref{boundary.74}, and Proposition \ref{boundary.65}.

\begin{const}
\label{boundary.39}
Let $f\colon X'\to X$ be a morphism in $\RglRg$ (resp.\ $\SmlSm/S$ with $S\in \Sch$) and $n\in \N$.
If $g\colon V\to X\times (\Gmlog)^n$ is a dividing cover corresponding to $Y\in \Sdiv_X^n$,
consider the pullback $g'\colon V'\to X'\times (\Gmlog)^n$ of $g$,
which is a dividing cover.
We have
\[
g^{-1}((X-\partial X)\times \A_\N^n)\cong (X-\partial X)\times \A_\N^n,
\]
which implies
\[
g'^{-1}((X'-\partial X')\times \A_\N^n)\cong (X'-\partial X')\times \A_\N^n.
\]
Hence we have $V'\in \Sdiv_{X'}^n$.
The construction of $V'$ is natural,
so we obtain a functor $f^*\colon \Sdiv_X^n\to \Sdiv_{X'}^n$ given by $f^*(V):=V'$.
\end{const}

\begin{df}
Let $\SBl(\Rg)$ (resp.\ $\SBl(S)$ with $S\in \Sch$) be the category of triples $(Y, X,[n])$ with $Y\in \SBl_X^n$, $X\in \RglRg'$ (resp.\ $X\in \SmlSm'/S$), and $n\in \N$ whose morphisms are of the form
\[
(u,f,\alpha)\colon (Y',X',[n'])\to (Y,X,[n])
\]
such that the square
\[
\begin{tikzcd}
Y'\ar[d]\ar[r,"u"]&
Y\ar[d]
\\
X'\times \square^{n'}\ar[r,"f\times \square^\alpha"]&
X\times \square^n
\end{tikzcd}
\]
commutes.

A morphism $(u,f,\alpha)\colon (Y',X',[n'])\to (Y,X,[n])$ is an \emph{admissible blow-up} if $f\colon X'\to X$ and $\alpha\colon [n']\to [n]$ are isomorphisms.
Let $\SAdm$ denote the class of admissible blow-ups.
\end{df}

\begin{prop}
\label{boundary.73}
The class $\SAdm$ in $\SBl(\Rg)$ (resp.\ $\SBl(S)$ with $S\in \Sch$) admits calculus of right fractions.
\end{prop}
\begin{proof}
For the Ore condition,
let
\[
(u,f,\alpha)\colon (Y',X',[n'])\to (Y,X,[n]),
\text{ }
(v,\id,\id)\colon (Z,X,[n])\to (Y,X,[n])
\]
be morphisms in $\SBl(\Rg)$ (resp.\ $\SBl(S)$) so that $(v,\id,\id)\in \SAdm$,
and let $X\to X_0$ be a strict morphism with $X_0\in \SmlSm/\F_1$.
We may assume $f=\id$ or $\alpha=\id$ since the pair $(f,\alpha)$ is the composite $(f,\id)\circ (\alpha,\id)$.

Assume $\alpha=\id$.
Consider $W\in \Sdiv_X^n$ corresponding to $Z\in \SBl_X^n$.
Construction \ref{boundary.39} yields $W'\in \Sdiv_{X'}^n$ with a morphism $W'\to W$ over $X\times (\Gmlog)^n$.
By Proposition \ref{boundary.45},
we may assume that there exists $Z'\in \SBl_{X'}^n$ corresponding to $W'$.
Lemma \ref{boundary.75} yields a morphism $Z'\to Z$ over $X\times \square^n$.
The square
\[
\begin{tikzcd}
Z'\ar[d]\ar[r]&
Z\ar[d]
\\
X'\times \square^n\ar[r]&
X\times \square^n
\end{tikzcd}
\]
commutes by Proposition \ref{boundary.67}.
Hence we obtain a morphism $(Z',X',[n])\to (Z,X,[n])$.
Since $\SBl_{X'}^{n}$ is cofiltered by Proposition \ref{boundary.41},
there exist morphisms $Z''\to Y',Z'$ in $\SBl_X'^{n}$.
The square
\begin{equation}
\label{boundary.73.1}
\begin{tikzcd}
Z''\ar[d]\ar[r]&
Y'\ar[d,"u"]
\\
Z\ar[r,"v"]&
Y
\end{tikzcd}
\end{equation}
commutes by Proposition \ref{boundary.67}.
This shows the Ore condition for the case of $\alpha=\id$.

Assume $f=\id$.
Consider $W\in \Sdiv_X^n$ corresponding to $Z\in \SBl_X^n$.
By Proposition \ref{boundary.45},
there exists $W_0\in \Sdiv_{X_0}^n$ with a morphism $X\times_{X_0} W_0\to W$ in $\Sdiv_X^n$.
Using Proposition \ref{boundary.31},
we may assume that there exists $Z_0\in \SBl_{X_0}^n$ corresponding to $W_0$.
Then there exists a morphism $X\times_{X_0} Z_0\to Z$ in $\SBl_X^n$ by Proposition \ref{boundary.70}.
Also, Proposition \ref{boundary.65} yields a morphism $(Z_0',X_0,[n'])\to (Z_0,X_0,[n])$ in $\SBl(\F_1)$.
Hence we obtain a morphism $(X\times_{X_0} Z_0',X,[n'])\to (Z,X,[n])$.
Since $\SBl_X^{n'}$ is cofiltered by Proposition \ref{boundary.41},
there exist morphisms $Z''\to Y',X\times_{X_0} Z_0'$ in $\SBl_X^{n'}$.
The square \eqref{boundary.73.1} commutes by Proposition \ref{boundary.67}.
This shows the Ore condition for the case of $f=\id$.

For the right cancellability condition,
let $(u_1,f_1,\alpha_1),(u_2,f_2,\alpha_2)\colon (Y',X',[n'])\to (Y,X,[n])$ be two morphisms, and let $(v,\id,\id)\colon (Y'',X',[n'])\to (Y',X',[n'])$ be an admissible blow-up such that
\[
(u_1,f_1,\alpha_1)\circ (v,\id,\id)
=
(u_2,f_2,\alpha_2)\circ (v,\id,\id).
\]
Then we have $f_1=f_2$ and $\alpha_1=\alpha_2$,
and we also have $u_1=u_2$ by Proposition \ref{boundary.67}.
This shows the right cancellability condition.
\end{proof}

\begin{prop}
\label{boundary.76}
There are equivalence of categories
\begin{gather*}
\SBl(\Rg)[\SAdm^{-1}]
\simeq
\RglRg'\times \bDelta,
\\
\SBl(\SmlSm/S)[\SAdm^{-1}]\simeq \SmlSm'/S\times \bDelta.
\end{gather*}
\end{prop}
\begin{proof}
Argue as in Proposition \ref{boundary.64}.
\end{proof}

\begin{df}
\label{boundary.43}
Let $\cF$ be a presheaf of spectra on $\RglRg'$ (resp.\ $\SmlSm'/S$ with $S\in \Sch$).
Argue as in Definition \ref{boundary.32} but use Propositions \ref{boundary.73} and \ref{boundary.76} instead of Propositions \ref{boundary.65} and \ref{boundary.64} to obtain
\[
\widetilde{\log}_\bullet \cF\in \PSh(\RglRg\times \bDelta,\Sp)
\text{ (resp.\ $\widetilde{\log}_\bullet \cF\in \PSh(\SmlSm/S\times \bDelta,\Sp)$).}
\]
Let $\widetilde{\Log}_n\cF(X)$ denote its value at $X$ and $n\in \bDelta$.
Then we have
\[
\widetilde{\Log}_n\cF(X)
\cong
\colim_{Y\in (\SBl_X^n)^\op}
\cF(Y)
\]
for $X\in \RglRg$ (resp.\ $X\in \SmlSm/\F_1$) and integer $n\geq 0$.
We define
\begin{equation}
\label{boundary.43.1}
\widetilde{\Log} \cF
:=
\colim_{n\in \bDelta^\op}\widetilde{\Log}_n\cF
\cong
\colim_{n\in \bDelta^\op}\colim_{Y\in (\SBl_X^n)^\op} \cF(Y).
\end{equation}
The construction of $\widetilde{\Log} \cF$ is natural in $\cF$,
so we obtain the functor
\begin{gather*}
\widetilde{\Log}
\colon
\PSh(\RglRg,\Sp)
\to
\PSh(\RglRg',\Sp)
\\
\text{
(resp.\
$\widetilde{\Log}
\colon
\PSh(\SmlSm/S,\Sp)
\to
\PSh(\SmlSm'/S,\Sp)$).
}
\end{gather*}
\end{df}

\begin{df}
Let $S\in \Sch$.
The \emph{logarithm functor} is  
\begin{gather*}
\Log:=\widetilde{\Log}\omega^\sharp
\colon \PSh(\Rg,\Sp)
\to
\PSh(\RglRg',\Sp)
\\
\text{(resp.\ 
$
\Log:=\widetilde{\Log}\omega^\sharp
\colon \PSh(\Sm/S,\Sp)
\to
\PSh(\SmlSm'/S,\Sp)
$)},
\end{gather*}
see Construction \ref{omega.20} for $\omega^\sharp$.
Explicitly,
for a presheaf $\cF$ of spectra on $\Rg$ (resp.\ $\Sm/S$) and $X\in \RglRg'$ (resp.\ $\SmlSm'/S$),
\eqref{boundary.43.1} yields a natural isomorphism
\begin{equation}
\label{boundary.12.1}
\Log \cF(X)
\cong
\colim_{n\in \bDelta^\op}
\colim_{Y\in (\SBl_X^n)^\op}
\cF(\ul{Y}).
\end{equation}
If $X\to X_0$ is a strict morphism with $X_0\in \SmlSm/\F_1$,
then Proposition \ref{boundary.45} yields a natural isomorphism
\begin{equation}
\label{boundary.52.1}
\Log \cF(X)
\cong
\colim_{n\in \bDelta^\op}
\colim_{Y_0\in (\SBl_{X_0}^n)^\op}
\cF(\ul{X}\times_{\ul{X_0}}\ul{Y_0}).
\end{equation}
Observe that if $X_0=\A_P$ for some sharp fs monoid $P$, then $P$ is free since $X\in \SmlSm/S$, so we have
\begin{equation}
\label{boundary.52.2}
\Log \cF(X)
\cong
\colim_{n\in \bDelta^\op}
\colim_{Y_0\in \Adm_{\A_P\times \square^n}^\op}
\cF(\ul{X}\times_{\ul{\A_P}}\ul{Y_0})
\end{equation}
together with Lemma \ref{boundary.23}.
\end{df}

\begin{prop}
\label{boundary.72}
Let $X\in \RglRg$ (resp.\ $X\in \SmlSm/\F_1$ with $S\in \Sch$).
For all integers $n\geq 0$ and $1\leq i\leq n$,
there exists a functor $\mu_i^*\colon \Sdiv_X^n \to \Sdiv_X^{n+1}$ satisfying the following property:
For $V\in \Sdiv_X^n$,
there exists a commutative square
\[
\begin{tikzcd}
(X-\partial X)\times \G_m^{n+1}\ar[r,"\id \times \G_m^{\mu_i}"]\ar[d]&
(X-\partial X)\times \G_m^n\ar[d]
\\
\alpha^*(V)\ar[r]&
V.
\end{tikzcd}
\]
\end{prop}
\begin{proof}
Argue as in Proposition \ref{boundary.30}.
\end{proof}

\begin{df}
Let $\cF$ be a presheaf of complexes on $\RglRg$ (resp.\ $\SmlSm/S$ with $S\in \Sch$).
Argue as in Definition \ref{boundary.27} but use Proposition \ref{boundary.72} also to obtain
\begin{gather*}
\logCLog_\bullet \cF\in \PSh(\RglRg\times \ECube,\Sp)
\\
\text{(resp.\ $\logCLog_\bullet \cF\in \PSh(\SmlSm/S\times \ECube,\Sp)$)}
\end{gather*}
such that
\[
\logCLog_n\cF(X)
\cong
\widetilde{\Log}_n\cF(X)
\]
for all $X\in \RglRg$ (resp.\ $X\in \SmlSm/S$) and $n\in \N$.
Let $\logCLog\cF$ be the presheaf of complexes given by
\[
\logCLog\cF(X)
:=
\lvert \logCLog_\bullet\cF(X)\rvert_\mathrm{cube}
\cong
\big\lvert \colim_{Y\in (\SBl_X^\bullet)^\op}\cF(Y)\big\rvert_\mathrm{cube}
\]
for $X\in \RglRg$ (resp.\ $X\in \SmlSm/S$).
The construction of $\logCLog \cF$ is natural in $\cF$,
so we obtain the functor
\begin{gather*}
\logCLog
\colon
\PSh(\RglRg,\rD(\Z))
\to
\PSh(\RglRg',\rD(\Z))
\\
\text{(resp.\ $\logCLog
\colon
\PSh(\SmlSm/S,\rD(\Z))
\to
\PSh(\SmlSm'/S,\rD(\Z))$).}
\end{gather*}
The \emph{cubical logarithm functor} is
\begin{gather*}
\CLog:=\logCLog\omega^\sharp
\colon
\PSh(\Rg,\rD(\Z))
\to
\PSh(\RglRg',\rD(\Z))
\\
\text{(resp.\ $\CLog:=\logCLog\omega^\sharp
\colon
\PSh(\Sm/S,\rD(\Z))
\to
\PSh(\SmlSm'/S,\rD(\Z))$).}
\end{gather*}
\end{df}

\begin{prop}
\label{boundary.51}
Let $\cF$ be a presheaf of spectra on $\RglRg$ (resp.\ $\SmlSm/S$ with $S\in \Sch$).
If $\cF$ is $\A^1$-invariant,
then $\widetilde{\Log} \cF$ on $\RglRg'$ (resp.\ $\SmlSm'/S$) is $\A^1$-invariant.
If $\cF$ is a Zariski (resp.\ strict Nisnevich) sheaf,
then $\widetilde{\Log} \cF$ on $\RglRg'$ (resp.\ $\SmlSm'/S$) is a Zariski (resp.\ strict Nisnevich) sheaf.
\end{prop}
\begin{proof}
Argue as in Proposition \ref{boundary.1} but use Proposition \ref{boundary.46} also.
\end{proof}

\begin{df}
\label{boundary.53}
Every $X\in \RglRg$ (resp.\ $X\in \SmlSm/S$ with $S\in \Sch$) admits a Zariski covering $\{U_i\}_{i\in I}$ with finite $I$ such that $U_i\in \RglRg'$ (resp.\ $U_i\in \SmlSm'/S$) since $X$ has a chart Zariski locally.
As a consequence of Proposition \ref{boundary.51},
we have the induced functors
\begin{gather*}
\widetilde{\Log}
\colon
\Sh_\Zar(\RglRg,\Sp)
\to
\Sh_\Zar(\RglRg,\Sp),
\\
\widetilde{\Log}
\colon
\Sh_\Zar(\SmlSm/S,\Sp)
\to
\Sh_\Zar(\SmlSm/S,\Sp),
\\
\Log
\colon
\Sh_\Zar(\Rg,\Sp)
\to
\Sh_\Zar(\RglRg,\Sp),
\\
\Log
\colon
\Sh_\Zar(\Sm/S,\Sp)
\to
\Sh_\Zar(\SmlSm/S,\Sp)
\end{gather*}
and similarly for the strict Nisnevich topology and cubical logarithm functor.
\end{df}

\begin{df}
A Zariski sheaf $\cF$ of spectra on $\Rg$ (resp.\ $\Sm/S$ with $S\in \Sch$) is \emph{logarithmic} if
\[
\cF(X)\cong \Log \cF(X\times \square^n)
\]
for $X\in \Rg$ (resp.\ $X\in \Sm/S$) and $n\in \N$.
\end{df}

\begin{prop}
\label{boundarify.5}
Let $S\in \Sch$,
and let $\cF$ be a dividing invariant strict Nisnevich sheaf of spectra on $\SmlSm/S$ such that the induced morphism
\[
\cF(X)
\to
\cF(X\times \square^n)
\]
is an isomorphism for every $X\in \Sm/S$ and integer $n\geq 0$.
Then we have $\cF\in \logSH_{S^1}(S)$.
\end{prop}
\begin{proof}
We need to show that the induced morphism
\[
\cF(Y)\to \cF(Y\times \square)
\]
is an isomorphism for every $Y\in \SmlSm/S$.
We proceed by induction on $d:=\max_{y\in Y} \rank \ol{\cM}_{Y,y}^\gp$.
The question is Zariski local on $Y$,
so we may assume that $Y$ admits a chart $\N^d$ for some $d\in \N$.
The claim is trivial if $d=0$,
so assume $d>0$.

Consider the strict Nisnevich distinguished squares
\[
\begin{tikzcd}
Y''-W\ar[d]\ar[r]&
Y''\ar[d]
\\
Y-W\ar[r]&
Y,
\end{tikzcd}
\text{ }
\begin{tikzcd}
Y''-W\ar[d]\ar[r]&
Y''\ar[d]
\\
Y'-W\ar[r]&
Y'
\end{tikzcd}
\]
with $W\in \Sm/S$ and $Y,Y',Y''\in \SmlSm/S$
in \cite[Proof of Proposition 2.4.3]{logA1}.
Note that we have $Y'\cong W\times \A_\N^d$.
By induction,
the induced morphisms
\[
\cF(Y-W)\to \cF((Y-W)\times \square),
\text{ }
\cF(Y''-W'')\to \cF((Y''-W'')\times \square)
\]
are isomorphisms.
By strict Nisnevich descent,
the induced morphism
\[
\cF(Y)\to \cF(Y\times \square)
\]
is an isomorphism if and only if
the induced morphism
\[
\cF(Y'')\to \cF(Y''\times \square)
\]
is an isomorphism.
Arguing similarly for $Y'$ and $Y''$,
we only need to show that the induced morphism
\[
\cF(Y')\to \cF(Y'\times \square)
\]
is an isomorphism.

Consider the Zariski covering $\{\A^1,\A_\N\}$ of $\square$.
Using cartesian products,
we obtain the Zariski covering of $\square^d$ consisting of $2^d$ open subschemes.
One of them is $U:=\A_\N^d$,
and let $V$ be the union of the other open subschemes.
Using the assumption on $\cF$,
we have
\[
\cF(W\times \square^n)
\cong
\cF(W\times \square^n\times \square).
\]
On the other hand,
we have
\begin{gather*}
\cF(W\times V)\cong
\cF(W\times V\times \square),
\\
\cF(W\times (U\cap V))\cong
\cF(W\times (U\cap V)\times \square)
\end{gather*}
by induction.
Combine these three isomorphisms above and use the assumption that $\cF$ is a Zariski sheaf to have
\[
\cF(W\times U)\cong \cF(W\times U \times \square).
\]
To conclude, observe that $W\times U\cong Y'$.
\end{proof}

\begin{prop}
\label{boundary.47}
Let $S\in \Sch$,
and let $\cF$ be a Nisnevich sheaf of spectra on $\Sm/S$.
If $\cF$ is logarithmic,
then we have $\Log \cF\in \logSH_{S^1}(S)$.
\end{prop}
\begin{proof}
Proposition \ref{boundary.51} implies that $\Log \cF$ is a strict Nisnevich sheaf.
By Proposition \ref{boundarify.5} and the assumption that $\cF$ is logarithmic,
it suffices to show that $\Log \cF$ is a dividing invariant,
i.e.,
for every dividing cover $X'\to X$ in $\SmlSm/S$,
the induced morphism $\Log \cF(X)\to \Log \cF(X')$ is an isomorphism.
We can work Zariski locally on $X$,
so we may assume that $X$ has a fan chart $\Sigma$.
We set $X_0:=\T_\Sigma$.
Using \cite[Proposition A.11.5, Lemma C.2.1]{BPO} (with $\infty$-category $\cD:=\Sp$ instead of a $1$-category),
we reduce to the case where $X'\cong X\times_{X_0}X_0'$ for some dividing cover $X_0'\to X_0$.
Argue as in Proposition \ref{boundary.22} but use Proposition \ref{boundary.45} also to show
\[
\colim_{Y_0\in (\SBl_{X_0}^n)^\op} \cF(\ul{X}\times_{\ul{X_0}}\ul{Y_0})
\cong
\colim_{Y_0'\in (\SBl_{X_0'}^n)^\op} \cF(\ul{X}\times_{\ul{X_0}}\ul{Y_0'}).
\]
We conclude by \eqref{boundary.45.1}.
\end{proof}

\begin{df}
For $S\in \Sch$,
we denote by $\LogSh(S)$ the full subcategory of $\Sh_\Nis(\Sm/S,\Sp)$ spanned by the logarithmic sheaves,
and we have its $\P^1$-stabilization $\Sp_{\P^1}(\LogSh(S))$.
By Proposition \ref{boundary.47}, we have the induced \emph{logarithm functor}
\[
\Log\colon \Sp_{\P^1}(\LogSh(S))
\to
\logSH(S).
\]
An object of $\Sp_{\P^1}(\Sh_\Nis(\Sm/S,\Sp))$ is \emph{logarithmic} if it belongs to $\Sp_{\P^1}(\LogSh(S))$.
\end{df}

\begin{rmk}
Let $S\in \Sch$,
and let $Z\to X$ be a closed immersion in $\Sm/S$.
We have the induced cartesian square
\[
Q:=
\begin{tikzcd}
Z\times_X \Bl_Z X\ar[d]\ar[r]&
\Bl_Z X\ar[d]
\\
Z\ar[r]&
X.
\end{tikzcd}
\]
For $\cF\in \LogSh(S)$,
$\cF(Q)\cong \Log \cF(Q)$ is cocartesian by \cite[Theorem 7.3.3]{BPO}.

Consider the $\infty$-category
\[
\MS_\Nis(S):=
\Sp_{\P^1}(\Sh_\mathrm{ebu,Nis}(\Sm/S,\Sp)),
\]
which is a Nisnevich variant of $\mathrm{MS}$ due to  Annala-Hoyois-Iwasa \cite[\S 4]{zbMATH07935738},
and see \cite[Definition 2.1]{zbMATH07935738} for $\mathrm{ebu}$.
Then $\Sp_{\P^1}(\LogSh(S))$ is a full subcategory of $\MS_\Nis(S)$ due to the above paragraph.

Note that we have the functor
\[
\omega^\sharp\colon  \MS_\Nis(S)
\to
\logSH(S)
\]
sending $\Sigma^{p,q}\Sigma_{\P^1}^\infty X_+$ to $\Sigma^{p,q}\Sigma_{\P^1}^\infty X_+$ for $X\in \Sm/S$ and integers $p$ and $q$.
\end{rmk}

\begin{quest}
Let $S\in \Sch$.
Then do we have an equivalence of $\infty$-categories
\[
\Sp_{\P^1}(\LogSh(S))\simeq \MS_\Nis(S)?
\]
See Conjecture \ref{boundarify.33} for what we expect in this direction.
\end{quest}

\section{Logarithm of \texorpdfstring{$\A^1$}{A1}-invariant sheaves}

In this section,
we explore what happens if we apply the logarithm functor $\Log$ to logarithmic $\A^1$-invariant sheaves of spectra.
See Theorem \ref{boundarify.8} for the smooth log smooth case and Theorem \ref{boundary.6} for the regular log regular case.

\

Recall that for $S\in \Sch$,
we have the adjoint functors
\[
\omega_\sharp : \logSH_{S^1}(S) \rightleftarrows \SH_{S^1}(S):\omega^*
\]
such that $\omega_\sharp(\Sigma^{n}\Sigma_{S^1}^\infty X_+)\cong \Sigma^n\Sigma_{S^1}^\infty (X-\partial X)_+$ for $X\in \SmlSm/S$ and integer $n$.

\begin{thm}
\label{boundarify.8}
Let $S\in \Sch$ and $\cF\in \SH_{S^1}(S)$.
If $\cF$ is logarithmic,
then there is a natural isomorphism
\[
\Log \cF
\cong
\omega^* \cF
\]
in $\logSH_{S^1}(S)$.
In particular, for every $X\in \SmlSm/S$,
there is a natural isomorphism
\[
\Log \cF(X)
\cong
\cF(X-\partial X).
\]
\end{thm}
\begin{proof}
By Propositions \ref{boundary.51} and \ref{boundary.47},
we see that $\Log \cF\in \logSH_{S^1}(S)$ is $\A^1$-invariant.
Hence by \cite[Proposition 2.5.7]{logA1},
we see that the induced morphism $\Log \cF(X)\to \Log \cF(X-\partial X)$ is an isomorphism.
To conclude,
observe that the canonical morphism $\cF(X-\partial X)\to \Log \cF(X-\partial X)$ is an isomorphism since $\cF$ is logarithmic and $X-\partial X$ has the trivial log structure.
\end{proof}

For $X\in \lSch$,
recall that a \emph{vector bundle $\cE\to X$} is a strict morphism such that $\ul{\cE}\to \ul{X}$ is a vector bundle.
We have the \emph{Thom space of $\cE\to X$} given by the quotient notation $\Th(\cE):=\cE/(\Bl_Z \cE,E)$, where $Z$ is the zero section, and $E$ is the exceptional divisor on $\Bl_Z \cE$.
If $\cF$ is a presheaf of spectra on a category containing the induced morphism $(\Bl_Z \cE,E)\to \cE$,
then we set
\[
\cF(\Th(\cE))
:=
\fib(\cF(\cE)\to \cF(\Bl_Z \cE,E)).
\]

\begin{prop}
\label{boundary.19}
Let $\cF$ be a $(\square,\divi)$-invariant strict Nisnevich sheaf on $\RglRg$,
and let $\cE\to X$ be a vector bundle with $X\in \RglRg$.
Consider the $0$-section $Z$ of $\cE$.
Then there is a natural isomorphism
\[
\cF(\Th(\cE))
\cong
\fib(\cF(\P(\cE\oplus \cO))\to \cF(\P(\cE))).
\]
\end{prop}
\begin{proof}
Consider $\cG\in \logSH_{S^1}(\ul{X})$ given by
\[
\cG(Y):=\cF(Y\times_{\ul{X}}X)
\]
for $Y\in \SmlSm/\ul{X}$.
Observe that $\ul{\cE}\to \ul{X}$ is a vector bundle.
By \cite[Proposition 7.4.5]{BPO},
we have a natural isomorphism
\[
\Sigma_{S^1}^\infty \Th(\ul{\cE})
\cong
\Sigma_{S^1}^\infty \P(\ul{\cE}\oplus \cO)/\P(\ul{\cE})
\]
in $\logSH_{S^1}(\ul{X})$.
Apply $\hom_{\logSH_{S^1}(\ul{X})}(-,\cG)$ to this isomorphism to conclude.
\end{proof}

Let $Z\to X$ be a strict closed immersion in $\lSch$.
Recall that the \emph{blow-up of $X$ along $Z$} is
\[
\Bl_Z X:=\Bl_{\ul{Z}}\ul{X}\times_{\ul{X}}X.
\]
The \emph{normal bundle of $Z$ in $X$} is
\[
\rN_Z X
:=
\rN_{\ul{Z}}\ul{X} \times_{\ul{Z}} Z.
\]

\begin{df}
Let $X\in \Rg$ (resp.\ $X\in \Sm/S$ with $S\in \Sch$),
and let $Z_1+\cdots+Z_n$ be a strict normal crossing divisor on $X$.
If $Y:=(X,Z_1+\cdots+Z_{n-1})$ and $Z:=(Z_n,Z_1\cap Z_n+\cdots+Z_{n-1}\cap Z_n)$,
then we will use the convenient notation
\[
(Y,Z):=(X,Z_1+\cdots+Z_n).
\]
\end{df}

\begin{const}
Let $X\in \RglRg$,
and let $Z$ be an effective Cartier divisor on $X$ such that $(X,Z)\in \RglRg$.
The \emph{$\A^1$-deformation space} is
\[
\rD_Z^{\A^1}X
:=
\Bl_Z (X\times \A^1)-\Bl_Z (X\times \{0\}).
\]
Then we have the induced commutative diagram
\[
\begin{tikzcd}
(X,Z)\ar[d]\ar[r]&
(\rD_Z^{\A^1}X,Z\times \A^1)\ar[d]\ar[r,leftarrow]&
(\rN_Z X,Z)\ar[d]
\\
X\ar[r]&
\rD_Z^{\A^1} X\ar[r,leftarrow]&
\rN_Z X.
\end{tikzcd}
\]
See \cite[Definition 7.5.1]{BPO} for the version using $\square$ instead of $\A^1$.
\end{const}

\begin{const}
\label{boundary.4}
Let $\Sigma \to \A^{d+1}\times (\P^1)^n$ be a subdivision of smooth fans in the lattice $\Z^{d+n+1}$ with $d,n\in \N$ such that $\Cone(e_{d+2},\ldots,e_{d+n+1})\in \Sigma$,
where $e_1,\ldots,e_{d+n+1}$ denotes the standard coordinates in $\Z^{d+n+1}$.
Let $\Sigma_0$ be the restriction of $\Sigma$ to the lattice $0 \times \Z^{d+n}$,
which is smooth.
Then we have
\[
\Cone(e_{d+2},\ldots,e_{d+n+1})\in \Sigma_0.
\]
Also, $\A^1\times \Sigma_0$ is a subdivision of $\A^{d+1}\times (\P^1)^n$,
and $\Sigma$ and $\A^1\times \Sigma_0$ contain $\G_m \times \Sigma_0$ as a common subfan.
Furthermore, we have 
\[
\Cone(e_{d+2},\ldots,e_{d+n+1})\in \A^1\times \Sigma_0.
\]
Consider the fan
\[
\Delta:=\{\sigma\cap \sigma':\sigma\in \Sigma,\sigma'\in \A^1\times \Sigma_0\}.
\]
Then $\Delta$ contains $\G_m\times \Sigma_0$ as a subfan,
and we have
\[
\Cone(e_{d+2},\ldots,e_{d+n+1})\in \Delta.
\]
By Proposition \ref{logmonoid.1},
there exists a subdivision $\Gamma\to \Delta$ such that $\Gamma$ is smooth and for every smooth cone $\delta$ of $\Delta$,
we have $\delta\in \Gamma$.
Then $\Gamma$ contains $\G_m\times \Sigma_0$ as a subfan,
and we have $\Cone(e_{d+2},\ldots,e_{d+n+1})\in \Gamma$.

We can interpret this in terms of standard blow-ups as follows.
For $d,n\in \N$,
let $\cD_{d,n}$ be the full subcategory of $\SBl_{\A_\N^{d+1}}^n$ spanned by those $Y'$ satisfying the following condition: There exists $Y\in \SBl_{\A_\N^d}^n$ with a morphism $p\colon Y'\to \A_\N\times Y$ in $\SBl_{\A_\N^{d+1}}^n$ such that $p$ is an isomorphism on $\G_m \times Y$,
i.e., $p^{-1}(\G_m \times Y)\to \G_m\times Y$ is an isomorphism.
Then $\cD_{d,n}$ is cofinal in $\SBl_{\A_\N^{d+1}}^n$.
Here, $Y$ and $Y'$ correspond to $\Sigma_0$ and $\Gamma$, the claim that $\Gamma$ contains $\G_m \times \Sigma_0$ as a subfan corresponds to the claim that $p$ is an isomorphism on $\G_m\times Y$,
and the claim that $Y\to \A_\N^d\times \square^n$ and $Y'\to \A_\N^{d+1}\times \square^n$ are admissible blow-ups corresponds to the claim that $\Cone(e_{d+2},\ldots,e_{d+n+1})\in \Sigma_0,\Gamma$.
In particular,
we have a natural isomorphism
\begin{equation}
\label{boundary.4.1}
\colim_{Y\in \SBl_{\A_\N^d}^n}\colim_{Y'\in \cC_Y}
\cong
\colim_{Y'\in \SBl_{\A_\N^{d+1}}^n},
\end{equation}
where $\cC_Y$ is the full subcategory of $\SBl_{\A_\N^{d+1}}^n$ spanned by those $Y'$ satisfying the following condition: There exists a morphism $p\colon Y'\to \A_\N\times Y$ in $\SBl_{\A_\N^{d+1}}^n$ such that $p$ is an isomorphism on $\G_m\times Y$.
\end{const}

\begin{thm}
\label{boundary.5}
Let $\cF$ be an $\A^1$-invariant cdh sheaf of spectra on $\Sch$.
Then for every $X\in \RglRg$ and effective divisor $Z$ on $X$ such that $(X,Z)\in \RglRg$,
the squares in the induced commutative diagram
\begin{equation}
\label{boundary.5.1}
\begin{tikzcd}
\Log \cF(X)\ar[d]\ar[r,leftarrow]&
\Log \cF(\rD_Z^{\A^1} X)\ar[d]\ar[r]&
\Log \cF(\rN_Z X)\ar[d]
\\
\Log \cF(X,Z)\ar[r,leftarrow]&
\Log \cF(\rD_Z^{\A^1} X,Z\times \A^1)\ar[r]&
\Log \cF(\rN_Z X,Z)
\end{tikzcd}
\end{equation}
are cartesian.
In particular,
there is a natural Gysin fiber sequence
\begin{equation}
\label{boundary.5.2}
\Log \cF(\Th(\rN_Z X))
\to
\Log \cF(X)
\to
\Log \cF(X,Z).
\end{equation}
\end{thm}
\begin{proof}
We focus on showing that the left square of \eqref{boundary.5.1} is cartesian since the proof for the right square is similar.
After that, we obtain \eqref{boundary.5.2}.

The question is Zariski local on $X$ by Proposition \ref{boundary.4},
so we may assume that the induced morphism $(X,Z)\to X$ is the pullback of the induced morphism $(X_0,Z_0)\to X_0$ with $X_0:=\A^1\times \A_\N^d$ and $Z_0:=\{0\} \times \A_\N^d$ along a strict morphism $X\to X_0$.
With this assumption,
we only need to show that the square
\[
\begin{tikzcd}
\Log_n \cF(\rD_Z^{\A^1} X)\ar[d]\ar[r]&
\Log_n
\cF(X)\ar[d]
\\
\Log_n \cF(\rD_Z^{\A^1} X,Z\times \A^1)\ar[r]&
\Log_n \cF(X,Z)
\end{tikzcd}
\]
is cartesian for every integer $n\geq 0$.

Form the log monoid scheme $\rD_{Z_0}^{\A^1}X_0$ as we did for fs log schemes.
Explicitly,
we have an isomorphism $\rD_{Z_0}^{\A^1}X_0\cong X_0 \times \A^1$.
We also have the induced strict morphism of squares
\[
\begin{tikzcd}
(X,Z)\ar[r]\ar[d]&
X\ar[d]
\\
(\rD_Z^{\A^1}X,Z\times \A^1)\ar[r]&
\rD_Z^{\A^1} X
\end{tikzcd}
\to
\begin{tikzcd}
(X_0,Z_0)\ar[r]\ar[d]&
X_0\ar[d]
\\
(\rD_{Z_0}^{\A^1}X_0,Z_0\times \A^1)\ar[r]&
\rD_{Z_0}^{\A^1} X_0.
\end{tikzcd}
\]
Furthermore, $X_0$ and $\rD_{Z_0}^{\A^1} X_0$ are strict over $\A_\N^d$,
and $(X_0,Z_0)$ and $(\rD_{Z_0}^{\A^1} X_0,Z_0\times \A^1)$ are strict over $\A_\N^{d+1}$.

Using Proposition \ref{boundary.45} and \eqref{boundary.4.1},
we have
\begin{gather*}
\Log_n\cF(X)
\cong
\colim_{Y_0\in \SBl_{\A_\N^d}^n}\cF(\ul{Y}),
\text{ }
\Log_n\cF(X,Z)
\cong
\colim_{Y_0\in \SBl_{\A_\N^d}^n}\colim_{Y_0'\in \cC_{Y_0}}\cF(\ul{Y'}),
\\
\Log_n\cF(\rD_Z^{\A^1}X)
\cong
\colim_{Y_0\in \SBl_{\A_\N^d}^n}\cF(\ul{W}),
\text{ }
\Log_n\cF(\rD_Z^{\A^1}X,Z\times \A^1)
\cong
\colim_{Y_0\in \SBl_{\A_\N^d}^n}\colim_{Y_0'\in \cC_{Y_0}}\cF(\ul{W'}),
\end{gather*}
where the notation $\cC_{Y_0}$ is due to Construction \ref{boundary.4}, and
\begin{gather*}
Y:=X\times_{\A^1\times \A_\N^d}(\A^1\times Y_0),
\text{ }
Y':=(X,Z)\times_{\A_\N^{d+1}}Y_0',
\\
W:=\rD_Z^{\A^1}X\times_{\A^1\times \A_\N^d}(\A^1\times Y_0),
\text{ }
W':=(\rD_Z^{\A^1}X,Z\times \A^1)\times_{\A_\N^{d+1}}Y_0'.
\end{gather*}

Hence it suffices to show that the induced square
\[
\begin{tikzcd}
\cF(\ul{W})\ar[d]\ar[r]&
\cF(\ul{Y})\ar[d]
\\
\cF(\ul{W'})\ar[r]&
\cF(\ul{Y'})
\end{tikzcd}
\]
is cartesian.
The morphism $Y_0'\to \A_\N \times Y_0$ is an isomorphism on $\G_m\times Y_0$ by the definition of $\cC_{Y_0}$.
It follows that the induced squares
\[
\begin{tikzcd}
\cF(\ul{Y})\ar[d]\ar[r]&
\cF(\ul{Y\times_X Z})\ar[d]
\\
\cF(\ul{Y'})\ar[r]&
\cF(\ul{Y'\times_X Z}),
\end{tikzcd}
\text{ }
\begin{tikzcd}
\cF(\ul{W})\ar[d]\ar[r]&
\cF(\ul{W\times_{\rD_Z^{\A^1}X}( Z\times \A^1)})\ar[d]
\\
\cF(\ul{W'})\ar[r]&
\cF(\ul{W'\times_{\rD_Z^{\A^1} X}(Z\times \A^1)})
\end{tikzcd}
\]
are cartesian since $\cF$ is a cdh sheaf.
Hence it suffices to show that the induced square
\[
\begin{tikzcd}
\cF(\ul{W\times_{\rD_Z^{\A^1}X}( Z\times \A^1)})\ar[d]\ar[r]&
\cF(\ul{Y\times_X Z})\ar[d]
\\
\cF(\ul{W'\times_{\rD_Z^{\A^1} X}(Z\times \A^1)})\ar[r]&
\cF(\ul{Y'\times_X Z})
\end{tikzcd}
\]
is cartesian.

We have the induced commutative diagram with cartesian squares
\[
\begin{tikzcd}
\ul{Z}\ar[d]\ar[r]&
\ul{Z}\times \A^1\ar[d]
\\
\ul{Z_0}\ar[r,"i_1"]\ar[d]&
\ul{Z_0}\times \A^1\ar[r]\ar[d]&
\{1\}\times \A^d\ar[d]
\\
\ul{X_0}\ar[r]&
\ul{\rD_{Z_0}^{\A^1}X_0}\ar[r]&
\A^{d+1},
\end{tikzcd}
\]
where $i_1$ is the $1$-section.
We also have
\[
\ul{Y\times_X Z}\cong
\ul{Z}\times_{\A^{d+1}}(\A^1\times \ul{Y_0}),
\text{ }
\ul{W\times_{\rD_Z^{\A^1} X}(Z\times \A^1)}
\cong
(\ul{Z}\times \A^1)\times_{\A^{d+1}}(\A^1\times \ul{Y_0}).
\]
From these, we see that the morphism $\ul{Y\times_X Z}\to \ul{W\times_{\rD_Z^{\A^1} X}(Z\times \A^1)}$ can be identified with the $1$-section $\ul{Y\times_X Z} \to \ul{Y\times_X Z} \times \A^1$.
The same holds if $Y$ and $W$ are replaced by $Y'$ and $W'$.
We conclude by the $\A^1$-invariance of $\cF$.
\end{proof}

As in \cite[Definition 1.3.2]{MR3930052},
we say that a presheaf of spectra $\cF$ on $\Rg$ satisfies \emph{absolute purity} if the squares in
\begin{equation}
\label{boundary.6.1}
\begin{tikzcd}
\cF(U)\ar[d]\ar[r,leftarrow]&
\cF(\rD_V^{\A^1} U)\ar[d]\ar[r]&
\cF(\rN_V U)\ar[d]
\\
\cF(U-V)\ar[r,leftarrow]&
\cF(\rD_V^{\A^1} U-V\times \A^1)\ar[r]&
\cF(\rN_V U- V)
\end{tikzcd}
\end{equation}
are cartesian for every regular embedding $V\to U$ in $\Rg$.

\begin{thm}
\label{boundary.6}
Let $\cF$ be an $\A^1$-invariant cdh sheaf of spectra on $\Sch$.
If $\cF$ satisfies absolute purity and the restriction of $\cF$ to $\Sm/S$ is logarithmic for every $S\in \Rg$,
then there is a natural isomorphism
\[
\Log \cF(X)\cong \cF(X-\partial X)\]
for every $X\in \RglRg$.
\end{thm}
\begin{proof}
We proceed by induction on the number $d$ of irreducible components of $\partial X$.
If $d=0$,
then the claim holds since $\cF$ is logarithmic.
Assume that the claim holds for $X$.
For every effective divisor $Z$ on $X$ such that $(X,Z)\in \RglRg$,
we need to show the claim for $(X,Z)$.

We set $U:=X-\partial X$ and $V:=Z-\partial Z$ for simplicity of notation.
Since $\cF$ satisfies absolute purity,
the squares in \eqref{boundary.6.1} are cartesian.
We have
\[
\Log \cF(X)\cong \cF(U)
\]
since the claim holds for $X$.
On the other hand,
we have natural isomorphisms
\begin{align*}
&\fib(
\Log \cF(\rN_Z X)\to \Log \cF(\rN_Z X,Z))
\\
\cong &
\fib(\Log \cF(\P(\rN_Z X\oplus \cO))\to \Log \cF(\P(\rN_Z X)))
\\
\cong &
\fib(\cF(\P(\rN_V U\oplus \cO))\to \cF(\P(\rN_V U)))
\\
\cong &
\fib(\cF(\rN_V U)\to \cF(\rN_V U-V)),
\end{align*}
where the first (resp.\ second, resp.\ third) isomorphism is due to Proposition \ref{boundary.19} (resp.\ induction, resp.\ \cite[Theorem 2.23 in \S 3]{MV}).
Since we have a natural morphism from \eqref{boundary.5.1} to \eqref{boundary.6.1},
we have $\Log \cF(X,Z)\cong \cF(U-V)$ by Theorem \ref{boundary.5},
i.e., the claim holds for $(X,Z)$.
\end{proof}

\section{Logarithmic sheaves of spectra}
\label{boundarify}

In this section, we show that various sheaves of spectra are logarithmic assuming $\CH^q(\C)\cong \CLog\CH^q(\square_{\C}^n)$ for all integers $n,q\geq 0$,
whose proof is postponed to \cite{logSHF2}.

For $X\in \Sch$,
let $\Kth(X)$ be the Bass $K$-theory spectrum of $X$.
For an abelian group like $\pi_i \Kth(X)$ for $X\in \Sch$ and integer $i$,
we often regard it as the associated Eilenberg-MacLane spectrum especially when taking $\otimes$ in $\Sp$.

\begin{lem}
\label{boundarify.3}
Let $X\to S$ be a morphism in $\Sch$.
Then for every proper $Y\in \Sm/\F_1$ and integer $i$,
we have natural isomorphisms
\begin{gather}
\label{boundarify.3.1}
\Kth(X)\otimes_{\Kth(S)} \Kth(S\times Y)
\cong
\Kth(X\times Y),
\\
\label{boundarify.3.2}
\pi_i \Kth(X)\otimes_{\pi_0\Kth(S)}
\pi_0\Kth(S\times Y)
\cong
\pi_i \Kth(X\times Y).
\end{gather}
Furthermore,
if $X\in \Rg$,
then we have a natural isomorphism
\begin{equation}
\label{boundarify.3.3}
\pi_0\Kth(X)\otimes_{\Kth(S)}\Kth(S\times Y)
\cong
\pi_0\Kth(X\times Y).
\end{equation}
\end{lem}
\begin{proof}
If $X\in \Rg$, then $\Kth(X)$ is $(-1)$-connected, so we can view $\pi_0 \Kth(X)$ as a $\Kth(S)$-module and hence we can define \eqref{boundarify.3.3}.

We proceed by induction on $n := \dim Y$. The claim is trivial if $n = 0$.
Assume $n > 0$. If $Y' \to Y$ is the blow-up along a smooth center $Z$, then the claim holds for $Z$ and $Z\times_Y Y'$ by induction.
Hence the claim holds for $Y$ if and only if the
claim holds for $Y'$ using the regular blow-up formula for K-theory \cite[Th\'eor\`eme 2.1]{MR1207482}.
Since $Y$ is obtained from $(\P^1)^n$ by several blow-ups and blow-downs along smooth centers by \cite[Theorem A]{zbMATH00963264}, we reduce to the case where $Y = (\P^1)^n$. The projective bundle formula for K-theory  \cite[Theorem 7.3]{TT} finishes the proof.
\end{proof}

\begin{lem}
\label{boundarify.4}
Let $X\to S$ be a morphism in $\Sch$.
Then for every proper $Y\in \SmlSm/\F_1$ and integer $i$,
we have natural isomorphisms
\begin{gather}
\label{boundarify.4.1}
\Kth(X)
\otimes_{\Kth(S)}
\Log \Kth(S\times Y)
\to
\Log \Kth(X\times Y),
\\
\label{boundarify.4.2}
\pi_i \Kth(X)\otimes_{\pi_0\Kth(S)}
\Log \pi_0\Kth(S\times Y)
\to
\Log \pi_i \Kth(X\times Y).
\end{gather}
Furthermore
if $X\in \Rg$,
then we have a natural isomorphism
\begin{equation}
\label{boundarify.4.3}
\pi_0\Kth(X)\otimes_{\Kth(S)}\Log \Kth(S\times Y)
\to
\Log \pi_0\Kth(X\times Y).
\\
\end{equation}
\end{lem}
\begin{proof}
Use \eqref{boundary.52.1} and Lemma \ref{boundarify.3}.
\end{proof}

For a spectrum $A$ and integer $i$,
let $\tau_{\geq i}A$ and $\tau_{\leq i}A$ be the truncations of $A$.
The following result shows that K-theory is logarithmic if and only if $\pi_0\Kth$ is logarithmic.

\begin{lem}
\label{boundarify.7}
Let $S\in \Rg$.
Then the following conditions are equivalent.
\begin{enumerate}
\item[\textup{(1)}]
$\Log \Kth(S\times \square^n)\cong \Kth(S)$ for all $n\in \N$.
\item[\textup{(2)}]
$\Log \tau_{\leq i}\Kth(S\times \square^n)\cong \tau_{\leq i}\Kth(S)$ for all $i\in \Z$ and $n\in \N$.
\item[\textup{(3)}]
$\Log \pi_i\Kth(S\times \square^n)\cong \pi_i \Kth(S)$ for all $i\in \Z$ and $n\in \N$.
\item[\textup{(4)}]
$\Log \pi_0\Kth(S\times \square^n)\cong \pi_0 \Kth(S)$ for all $n\in \N$.
\end{enumerate}
\end{lem}
\begin{proof}
By \eqref{boundarify.4.3},
(1) implies (4).
By \eqref{boundarify.4.2},
(4) implies (3).
From \eqref{boundary.12.1},
we see that $\Log \tau_{\geq i}\Kth(S\times \square^n)$ is $(i-1)$-connected.
Hence we have
\[
\lim_{i} \Log \tau_{\geq i} \Kth(S\times \square^n)\cong 0.
\]
It follows that (2) implies (1).
By induction on the degree, one can show that (2) is equivalent to (3).
\end{proof}

For a perfect field $k$,
consider Voevodsky's stable $\infty$-category of motives $\DM(k)$, see \cite[Definition 14.2]{MVW} for the triangulated version.
For a simplicial smooth scheme $X_\bullet$ over $k$,
we take $M(X_\bullet):=\colim_{i\in \bDelta^\op} M(X_i)$.
We have a similar definition for an ind-smooth scheme.
We have the commutative ring structure on the object $M(B \G_m)$ of $\DM(k)$ induced by the multiplication $\G_m\times \G_m\to \G_m$.
We have isomorphisms
\[
M(B\G_m)\cong M(\P^\infty)\cong \bigoplus_{q=0}^\infty \Z(q)[2q]
\]
in $\DM(k)$ by \cite[Proposition 3.7 in \S 4]{MV} and \cite[Corollary 15.5]{MVW},
so we have a commutative ring structure on $\bigoplus_{q=0}^\infty \Z(q)[2q]$.
Using this,
we have a commutative ring structure on
\[
R\Gamma_\mot(X,\L^*)
:=
\bigoplus_{q=0}^\infty R\Gamma_\mot(X,\Z(q)[2q])
\]
for $X\in \Sm/k$,
where $R\Gamma_\mot(X,\Z(q))$ denotes the weight $q$ motivic cohomology complex,
and $\L:=\Z(1)[2]$ is the notation for the Lefschetz motive.
Observe that $R\Gamma_\mot(X,\L^*)$ is $(-1)$-connected by \cite[Vanishing Theorem 19.3]{MVW}.

\begin{lem}
\label{boundarify.30}
Let $X\in \Sm/k$,
where $k$ is a perfect field.
Then for every proper $Y\in \Sm/\F_1$ and integer $i$,
we have natural isomorphisms
\begin{gather*}
R\Gamma_\mot(X,\L^*)
\otimes_{R\Gamma_\mot(k,\L^*)}
R\Gamma_\mot(Y_k,\L^*)
\cong
R\Gamma_\mot(X\times Y,\L^*),
\\
\pi_i R\Gamma_\mot(X,\L^*)
\otimes_{R\Gamma_\mot(k,\L^*)}
\pi_0 R\Gamma_\mot(Y_k,\L^*)
\cong
\pi_i R\Gamma_\mot(X\times Y,\L^*),
\\
\pi_0 R\Gamma_\mot(X,\L^*)
\otimes_{R\Gamma_\mot(k,\L^*)}
R\Gamma_\mot(Y_k,\L^*)
\cong
\pi_0 R\Gamma_\mot(X\times Y,\L^*),
\end{gather*}
where $Y_k:=\Spec(k)\times Y$.
\end{lem}
\begin{proof}
Argue as in Lemma \ref{boundarify.3},
but use the projective bundle formula \cite[Theorem 15.12]{MVW} and blow-up formula \cite[Corollary 15.13]{MVW} in $\DM(k)$ instead.
\end{proof}

\begin{lem}
\label{boundarify.10}
Let $X\in \Sm/k$, where $k$ is a perfect field.
Then the following conditions are equivalent.
\begin{enumerate}
\item[\textup{(1)}]
$\CLog R\Gamma_\mot(X\times \square^n,\Z(q))\cong R\Gamma_\mot(X,\Z(q))$ for all $n,q\in \N$.
\item[\textup{(2)}]
$\CLog \tau_{\leq i}R\Gamma_\mot(X\times \square^n,\Z(q))\cong \tau_{\leq i}R\Gamma_\mot(X,\Z(q))$ for all $i\in \Z$ and $n,q\in \N$.
\item[\textup{(3)}]
$\CLog H_\mot^i(X\times \square^n,\Z(q))\cong H_\mot^i(X,\Z(q))$ for all $i\in \Z$ and $n,q\in \N$.
\item[\textup{(4)}]
$\CLog \CH^q(X\times \square^n)\cong \CH^q(X)$ for all $n,q\in \N$.
\end{enumerate}
\end{lem}
\begin{proof}
Argue as in Lemmas \ref{boundarify.4} and \ref{boundarify.7} for $R\Gamma_\mot(-,\L^*)$, but use Lemma \ref{boundarify.30} instead.
\end{proof}

\begin{lem}
\label{boundarify.11}
Let $\cF_\bullet$ be a simplicial spectrum.
If there exists $d\in \Z$ such that $\cF_n$ is $(d-n-1)$-connected for all $n\in \N$,
then the geometric realization $\lvert \cF_\bullet\rvert:=\colim_{n\in \bDelta^\op} \cF_n$ is $(d-1)$-connected.
\end{lem}
\begin{proof}
Recall from \cite[Theorem X.2.9]{MR1417719} that we have the strongly convergent skeleton spectral sequence
\[
E_{p,q}^2=
H_p (\pi_q \cF_\bullet)
\Rightarrow
\pi_{p+q} \lvert \cF_\bullet \rvert.
\]
The assumption on $\cF_\bullet$ implies that $E_{p,q}^2=0$ whenever $p+q<d$.
We obtain the desired claim from this.
\end{proof}

The proof of the following theorem occupies \cite{logSHF2}.
We will use this as a black box in this paper.
For a fan $\Sigma$ and ring $R$,
let $\Sigma_R$ denote the associated toric scheme over $R$.
For $Y\in \SmlSm/\F_1$,
let $Y_R:=\Spec(R)\times Y$.

\begin{thm}
\label{boundarify.9}
We have $\CH^q(\C)\cong \CLog\CH^q(\square_{\C}^r)$ for all integers $q,r\geq 0$.
\end{thm}

\begin{rmk}
Let us give an explicit description of $\CLog\CH^q(\square_{\C}^r)$.
Consider the category $\cC_{n,r}$ of smooth subdivisions $\Sigma \to (\P^1)^{n+r}$ satisfying the following two conditions:
\begin{enumerate}
\item[(i)] If $a:=(a_1,\ldots,a_{n+r})$ is a ray of $\Sigma$ such that $a_i>0$ for some $n+1\leq i\leq n+r$,
then we have $a=e_i$.
\item[(ii)] $\Cone(e_{1},\ldots,e_{n})\in \Sigma$.
\end{enumerate}
Note that there is an equivalence of categories
\(
\cC_{n,r}\simeq \SBl_X^n
\)
by Example \ref{boundary.54}.

Let $e_1,\ldots,e_{n+r}$ be the standard coordinates in $\Z^{n+r}$.
For integers $1\leq i\leq n$ and $\Sigma\in \cC_{n,r}$,
let $D_{i,0}(\Sigma)$ be the fan $V(e_{i})$,
and let $D_{i,1}(\Sigma)$ be the fan consisting of the cones of $\Sigma$ contained in the lattice $\Z^{i-1}\times 0 \times \Z^{n+r-i}$.
Observe that $D_{i,0}(\Sigma)_{\C}$ and $D_{i,1}(\Sigma)_{\C}$ are closed subschemes of $\Sigma_{\C}$.
By Example \ref{boundary.60}, the map
\[
\delta_{i,0}^*\colon \Log_n\CH^q(\square_\C^r)\to \Log_{n-1}\CH^q(\square_\C^r)
\]
agrees with the induced map
\[
\colim_{\Sigma\in \cC_{r,n}} \CH^q(\Sigma_{\C})\to \colim_{\Sigma\in \cC_{r,n}} \CH^q(D_{i,0}(\Sigma)_{\C}),
\]
and the map
\[
\delta_{i,1}^*\colon  \Log_n\CH^q(\square_\C^r)\to \Log_{n-1}\CH^q(\square_\C^r)
\]
agrees with the induced map
\[
\colim_{\Sigma\in \cC_{r,n}} \CH^q(\Sigma_{\C})\to \colim_{\Sigma\in \cC_{r,n}} \CH^q(D_{i,1}(\Sigma)_{\C}).
\]
We also have the maps
\begin{gather*}
p_i^*
\colon
\Log_{n-1} \CH^q(\square_\C^r)
\to
\Log_n \CH^q(\square_\C^r)
\text{ for $1\leq i\leq n$},
\\
\mu_i^*
\colon
\Log_{n-1} \CH^q(\square_\C^r)
\to
\Log_n \CH^q(\square_\C^r)
\text{ for $1\leq i\leq n-1$},
\end{gather*}
and the maps $\delta_{i,0}^*$, $\delta_{i,1}^*$, $p_i^*$, and $\mu_i^*$ form an extended cubical abelian groups.
Now, $\CLog\CH^q(\square_{\C}^r)$ is the complex
\[
\cdots \xrightarrow{\delta^*} \Log_1^\flat \CH^q(\square_\C^r)
\xrightarrow{\delta^*}
\Log_0^\flat \CH^q(\square_\C^r),
\]
where
\[
\Log_n^\flat \CH^q(\square_\C^r)
:=
\bigcap_{i=1}^n \ker(\delta_{i,0}^*)
\]
for $n\in \N$,
and $\delta^*:=\sum_{i=1}^n (-1)^i \delta_{i,1}^*$.
We will use this explicit combinatorial description of $\CLog\CH^q(\square_{\C}^r)$ in \cite{logSHF2}.
\end{rmk}

\begin{thm}
\label{boundarify.31}
Let $k$ be a perfect field.
Then the Nisnevich sheaf $R\Gamma_\mot(-,\Z(q))$ on $\Sm/k$ is logarithmic for every integer $q\geq 0$.
\end{thm}
\begin{proof}
Let $X\in \Sm/k$.
Consider the Zariski sheaf of complexes $\cF_X$ on $\SmlSm/\F_1$ given by $\cF_X(Y):=R\Gamma_\mot(X\times Y,\L^*)$.
Proposition \ref{boundary.26} yields a natural isomorphism
\[
\Log \cF_X(Y)
\cong
\CLog \cF_X(Y).
\]
Together with \eqref{boundary.45.1},
we have a natural isomorphism
\[
\Log R\Gamma_\mot(X\times Y,\L^*)
\cong
\CLog  R\Gamma_\mot(X\times Y,\L^*)
\]
for all $Y\in \SmlSm/\F_1$. 
Hence it suffices to show
\[
R\Gamma_\mot(X,\L^*)\cong \CLog R\Gamma_\mot(X\times \square^n,\L^*)
\]
for all integers $n\geq 0$.
Argue as in Lemma \ref{boundarify.4} but use Lemma \ref{boundarify.30} to show
\[
R\Gamma_\mot (X,\L^*)
\otimes_{R\Gamma_\mot(k,\L^*)}
\CLog R\Gamma_\mot(\square_k^n,\L^*)
\cong
\CLog R\Gamma_\mot(X\times \square^n,\L^*).
\]
Hence it suffices to show 
\[
R\Gamma_\mot(k,\Z(q))\cong \CLog R\Gamma_\mot(\square_k^n,\Z(q))
\]
for every integer $q\geq 0$.
By Lemma \ref{boundarify.10},
it suffices to show $\CH^q(k)\cong \CLog\CH^q(\square_k^n)$.
For every $V\in \Sm/\F_1$, $\CH^q(V_k)$ is independent of the choice of the field $k$ by \cite[Proposition 2.1]{MR1415592}.
Hence we may assume $k=\C$,
and then Theorem \ref{boundarify.9} finishes the proof.
\end{proof}

\begin{thm}
\label{boundarify.6}
Let $S\in \Sch$.
Then the Nisnevich sheaf $\Kth$ of spectra on $\Sm/S$ is logarithmic.
\end{thm}
\begin{proof}
By \eqref{boundarify.4.1},
it suffices to show $\Kth(\Z) \cong \Log \Kth(\square_{\Z}^n)$ for every integer $n\geq 0$.
By Lemma \ref{boundarify.7},
it suffices to show $\pi_0\Kth(\Z) \cong \Log \pi_0\Kth(\square_{\Z}^n)$.
Since $\pi_0\Kth(\Z)\cong \pi_0\Kth(\C)\cong \Z$,
\eqref{boundarify.4.2} implies $\Log \pi_0\Kth(\square_\Z^n)\cong \Log \pi_0\Kth(\square_\C^n)$.
Hence it suffices to show $\pi_0\Kth(\C) \cong \Log \pi_0\Kth(\square_{\C}^n)$.
By Lemma \ref{boundarify.7} again,
it suffices to show $\Kth(\C) \cong \Log \Kth(\square_{\C}^n)$.

Consider the slice filtration $\fil_\bullet\Kth$ of $\Kth\in \SH(\C)$ in \cite[\S 7.1]{MR1977582}.
By \cite[Theorem 7.1.1]{MR2365658},
this agrees with the Friedlander-Suslin filtration on $\Kth$ obtained by \cite[Proposition 13.1]{MR1949356}  since the homotopy coniveau filtration is an evident extension of the Friedlander-Suslin filtration, see \cite[p.\ 217]{MR2365658}.
The graded pieces $\slice_d\Kth$ are isomorphic to $R\Gamma_\mot(-,(d)[2d])$ for all integers $d$ by \cite[Theorem 6.4.2]{MR2365658}.

Consider $\fil^d\Kth:=\fib(\Kth \to \fil_{d-1}\Kth)$ for $d\in \Z$.
By \cite[Lemma 13.5]{MR1949356},
$\fil_d \Kth(V)$ is $(d-\dim V-1)$-connected for $V\in \Sm/\C$.
Hence for $X\in \SmlSm'/\C$ (see Definition \ref{boundary.71} for this notation) and integer $n\geq 0$,
\eqref{boundary.12.1} implies that $\Log_n \fil_d \Kth(X)$ is $(d-\dim X-n-1)$-connected.
Together with Lemma \ref{boundarify.11},
we see that $\Log \fil_d \Kth(X)$ is $(d-\dim X-1)$-connected,
which implies
\begin{equation}
\label{boundarify.6.1}
\lim_d \Log \fil_d \Kth(X)
\cong
0.
\end{equation}

Assume that $\fil^d\Kth:=\Kth \to \fil_{d-1}\Kth$ on $\Sm/\C$ satisfies $\fil^d\Kth(\C)\cong \Log \fil^d\Kth(\square_{\C}^n)$ for every $d\in \Z$ and $n\in \N$.
We have
\[
\Log \Kth(\square_{\C}^n)
\cong
\lim_d \Log \fil^d \Kth(\square_{\C}^n)
\cong
\lim_d \fil^d \Kth(\C)
\cong
\Kth(\C),
\]
where the first isomorphism is due to \eqref{boundarify.6.1},
and the third isomorphism is due to the completeness of the slice filtration $\fil_\bullet \Kth$ \cite[Lemma 3.11]{RSO_hopf}.
It follows that we have $\Kth(\C) \cong \Log \Kth(\square_{\C}^n)$.
Hence it suffices to show $\fil^d\Kth(\C)\cong \Log \fil^d\Kth(\square_{\C}^n)$ for every integer $d$.
By induction on degree $d$, it suffices to show that the slice $\slice_d \Kth\cong R\Gamma_\mot(-,(d)[2d])$ satisfies $R\Gamma_\mot(\C,(d)[2d])\cong \Log R\Gamma_\mot(\square_\C^n,(d)[2d])$,
which is due to Theorem \ref{boundarify.31}.
\end{proof}

\begin{thm}
\label{boundarify.29}
Let $X\in \RglRg$.
Then we have a natural isomorphism of $\E_\infty$-rings
\[
\Log\Kth(X)
\cong
\Kth(X-\partial X).
\]
\end{thm}
\begin{proof}
The K-theory of schemes in $\Sch$ satisfies cdh descent by \cite[Th\'eor\`eme 3.9]{zbMATH06156613} and absolute purity as observed in \cite[Remark 6.1.3]{MR3930052}.
Theorems \ref{boundary.6} and \ref{boundarify.6} finish the proof since we have a natural morphism of $\E_\infty$-rings $\Log \Kth(X)\to \Kth(X-\partial X)$.
\end{proof}

For a scheme $X$,
let $\THH(X)$ and $\TC(X)$ denote the topological Hochschild homology and topological cyclic homology of $X$.

\begin{thm}
\label{boundarify.21}
Let $\cF$ be a $\Kth$-module object of $\Sh_\Zar(\Sch,\Sp)$.
Assume the following conditions:
\begin{enumerate}
\item[\textup{(i)}]
The induced morphism
\[
\Kth(X\times \P^1)\otimes_{\Kth(X)}\cF(X)\to \cF(X\times \P^1)
\]
is an isomorphism for every $X\in \Sch$ and $n\in \N$.
\item[\textup{(ii)}]
For every $S\in \Sch$ and a closed immersion $Z\to X$ in $\Sm/S$,
the induced square
\[
\begin{tikzcd}
\cF(X)\ar[d]\ar[r]&
\cF(Z)\ar[d]
\\
\cF(\Bl_Z X)\ar[r]&
\cF(\Bl_Z X\times_X Z)
\end{tikzcd}
\]
is cartesian.
\end{enumerate}
Then the restriction of $\cF$ to $\Sm/S$ is logarithmic for every $S\in \Sch$.
In particular, $\THH$ and $\TC$ are logarithmic.
\end{thm}
\begin{proof}
Use the conditions (i) and (ii) and argue as in Lemmas \ref{boundarify.3} and \ref{boundarify.4} to show that the induced morphism
\begin{equation}
\label{boundarify.21.1}
\cF(X)\otimes_{\Kth(X)}L_\partial \Kth(X\times Y)
\to
L_\partial \cF(X\times Y)
\end{equation}
is an isomorphism for every $X\in \Sm/S$ and $Y\in \SmlSm/\F_1$.
Since $\Kth$ is logarithmic by Theorem \ref{boundarify.6},
we have $L_\partial \Kth(X\times \square^n)\cong \Kth(X)$ for every $n\in \N$.
Together with \eqref{boundarify.21.1},
we have $L_\partial \cF(X\times \square^n)\cong \cF(X)$, i.e., the restriction of $\cF$ to $\Sm/S$ is logarithmic.

The condition (i) for $\THH$ and $\TC$ follows from the fact that the cyclotomic trace is compatible with the projective bundle formula,
see \cite[Theorem 1.5]{BM12}.
The condition (ii) for $\THH$ and $\TC$ is a consequence of \cite[Theorem 7.3.3]{BPO} and the log motivic representability \cite[Construction 8.4.6]{BPO2} or by \cite[Theorem 1.4]{BM12}.
\end{proof}

Recall from \cite[Definition III.1.2.3]{Ogu} that a \emph{log ring $(R,P)$} (often called pre-log ring) is a pair of a ring $R$ and a monoid $P$ equipped with a map of monoids $P\to R$,
where the monoid operation on $R$ is the multiplication.
A \emph{map of log rings} is an obvious commutative square.

For an fs log scheme $X$,
let $\logTHH(X)$ and $\logTC(X)$ denote the topological Hochschild homology and topological cyclic homology of $X$ in \cite[Definition 8.3.7]{BPO2},
which are obtained by globalizing Rognes' topological Hochschild homology of log rings \cite[Definition 8.11]{zbMATH05610455} and its cyclotomic structure in \cite[Definition 3]{Ob18}.
Since $\logTHH(X)$ and $\logTC(X)$ agree with the original ones if $X$ has the trivial log structure,
we have canonical isomorphisms
\begin{equation}
\label{boundarify.20.2}
\omega_\sharp \logTHH\cong \THH,
\text{ }
\omega_\sharp \logTC\cong \TC
\end{equation}
in $\PSh(\Sm/S,\Sp)$ for $S\in \Sch$.
We refer to Construction \ref{omega.20} for
\[
\omega_\sharp\colon \PSh(\SmlSm/S,\Sp)\to \PSh(\Sm/S,\Sp).
\]

For a ring $R$ and $X\in \lSch/R$,
consider the Hochschild homology
\[
\logHH(X/R):=\logTHH(X)\otimes_{\THH(R)}R.
\]
If $X$ has the trivial log structure,
then $\logHH(X/R)$ agrees with the Hochschild homology $\HH(X/R)$.
By \cite[Theorem 5.15, \S 8.6]{BLPO},
we have the complete exhaustive HKR-filtration $\Fil_\bullet^\HKR$ on $\logHH(-/R)$ whose graded pieces satisfies
\begin{equation}
\label{boundarify.28.1}
\gr_i^\HKR\logHH(X/R)
\cong
R\Gamma_{\Zar}(X,\Omega_{X/R}^i)
\end{equation}
for $X\in \lSm/R$ and integer $i$.

\begin{prop}
For a noetherian ring $R$,
$\HH(-/R)\in \Sh_\Zar(\Sm/S,\Sp)$ is logarithmic.
\end{prop}
\begin{proof}
This is a consequence of Theorem \ref{boundarify.21}.
\end{proof}

\begin{lem}
\label{boundarify.27}
Let $R$ be a ring,
and let $X\in \Sm/R$.
Then for every proper $Y\in \Sm/\F_1$ and integer $i$,
the induced morphisms
\begin{gather*}
\HH(X/R)\otimes_R \HH(Y_R/R)
\to
\HH(X\times Y/R),
\\
\gr_0^\HKR\HH(X/R)\otimes_R \HH(Y_R/R)
\to
\gr_0^\HKR\HH(X\times Y/R),
\\
\gr_i^\HKR \HH(X/R)
\otimes_R
\gr_0^\HKR\HH(Y_R/R)
\to
\gr_i^\HKR \HH(X\times Y/R)
\end{gather*}
are isomorphisms,
where $Y_R:=\Spec(R)\times Y$.
\end{lem}
\begin{proof}
Argue as in Lemma \ref{boundarify.3},
but use the blow-up and projective bundle formulas for $\HH(-/R)$ obtained by the formulas for $\THH$ in \cite[Theorems 1.4, 1.5]{BM12} and the formulas for Hodge cohomology.
\end{proof}

\begin{thm}
\label{boundarify.23}
Let $S\in \Sch$.
Then the presheaf of complexes $R\Gamma_\Zar(-,\Omega_{-/S}^i)$ on $\Sm/S$ is logarithmic for every integer $i\geq 0$.
\end{thm}
\begin{proof}
We can work Zariski locally on $S$,
so we may assume that $S=\Spec(R)$ for some ring $R$.
Argue as in Lemma \ref{boundarify.7} and use Lemma \ref{boundarify.27} to show that the following conditions are equivalent.
\begin{enumerate}
\item[(1)] $\Log \HH(X\times \square^n/R)\cong \HH(X/R)$ for all $X\in \Sm/S$ and $n\in \N$.
\item[(2)] $\Log\, \gr_i^\HKR\HH(X\times \square^n/R)\cong \gr_i^\HKR\HH(X)$ for all $X\in \Sm/S$, $i\in \Z$, and $n\in \N$.
\end{enumerate}
We have checked that $\HH(-/R)$ is logarithmic, so we have (1).
Furthermore,
(2) is what we need to show due to \eqref{boundarify.28.1}.
\end{proof}

\begin{conj}
\label{boundarify.32}
Let $S\in \Rg$.
Then the algebraic cobordism spectrum $\MGL\in \SH(S)$ due to Voevodsky \cite[\S 6.3]{zbMATH01194164} is logarithmic.
\end{conj}

Conjecture \ref{boundarify.32} is conceptually reasonable but technically more demanding than Theorem \ref{boundarify.6}.
The proof of Theorem \ref{boundarify.9} heavily uses known combinatorial computations of Chow groups of toric varieties.
Hence we would need to walk through more complicated computations of algebraic cobordism of toric varieties,
see \cite[Corollary 1.2]{MR3142260} for such a computation.
If we prove Conjecture \ref{boundarify.32},
then the following one would be also reachable:

\begin{conj}
\label{boundarify.33}
Let $S\in \Rg$.
Then every $\MGL$-module in $\MS_\Nis(S)$ is logarithmic.
\end{conj}

This would indicate that there is a close relationship between the full subcategories of orientable objects of $\MS_\Nis(S)$ and $\logSH(S)$.

\section{Motivic representability of logarithmic K-theory}

In this section,
we show that K-theory is representable in $\logSH(S)$ for every $S\in \Sch$,
see Theorem \ref{boundarify.19}.
We first need to set up a natural transformation $u$ as follows.

\begin{const}
\label{boundarify.12}
Let $S\in \Sch$.
Then we have the natural transformation
\[
u\colon \id\to \widetilde{\Log}
\colon
\Sh_{\sNis}(\SmlSm/S,\Sp)
\to
\Sh_{\sNis}(\SmlSm/S,\Sp)
\]
given as follows.
The morphism of spectra $u(\cF)(X)$, which we denote by $u_{\cF}(X)$, is given by
\begin{equation}
\colim_{n\in \bDelta^\op}
\cF(X)
\to
\colim_{n\in \bDelta^\op}
\colim_{Y\in (\SBl_X^n)^\op}
\cF(Y)
\end{equation}
for $\cF\in \Sh_\sNis(\SmlSm/S,\Sp)$ and $X\in \SmlSm'/S$,
see Definition \ref{boundary.71} for $\SmlSm'/S$.
Here, the colimit on the source is the constant colimit,
and when $n$ is fixed,
the morphism $\cF(X)\to \colim_{Y\in (\SBl_X^n)^\op}\cF(Y)$ factors through $\cF(X)\to \cF(X\times \square^n)$ induced by the projection $X\times \square^n\to X$.
\end{const}

\begin{prop}
\label{boundarify.24}
Let $P$ be a sharp fs monoid,
$S\in \Sch$, $n\in \N$, and $Y\in \SBl_{\A_P}^n$.
Then there exists a sequence of admissible blow-ups along smooth centers
\[
Y_m \to \cdots \to Y_0:=\A_P\times \square^n
\]
in $\SBl_{\A_P}^n$ such that there is a morphism $Y_m\to Y$ in $\SBl_{\A_P}^n$.
\end{prop}
\begin{proof}
We have $\ul{Y}\cong \ul{\T_\Delta}$
for some subdivision $\Delta$ of $\Spec(\F_1[P]) \times (\P^1)^n$.
Consider the cone $\sigma:=\Cone(e_1,\ldots,e_n)$ of $\Spec(\F_1[P]) \times (\P^1)^n$,
where $e_1,\ldots,e_n$ are the standard coordinates for $(\P^1)^n$.
The condition that $Y\to \A_P\times \square^n$ is an admissible blow-up is equivalent to the condition that $\sigma\in \Delta$.
By Proposition \ref{logmonoid.2},
there exists a sequence of star subdivisions
\[
\Delta_m \to \cdots \to \Delta_0:=\Spec(\F_1[P])\times (\P^1)^n
\]
such that $\sigma\in \Delta_m$ and there is a subdivision $\Delta_m\to \Delta$.
Consider $Y_i\in \SBl_{\A_P}^n$ for each $i$ such that $\ul{Y_i}:=\ul{\T_{\Delta_i}}$.
Then we have a desired sequence.
\end{proof}

\begin{prop}
\label{boundarify.13}
Let $S\in \Sch$ and $\cF\in \logSH_{S^1}(S)$.
Then $u_\cF\colon \cF\to \widetilde{\Log}\cF$ is an isomorphism and hence admits an inverse $v_\cF\colon \widetilde{\Log}\cF\to \cF$.
\end{prop}
\begin{proof}
We can work Zariski locally on $X$ by Proposition \ref{boundary.51},
so we may assume that $X$ has a chart $P$ such that $P$ is a sharp fs monoid.
By Proposition \ref{boundary.45},
it suffices to show that the morphism
\[
u_\cF\colon \cF(X)
\to
\colim_{n\in \bDelta^\op}\colim_{Y\in \SBl_{\A_P}^n} \cF(X\times_{\A_P} Y)
\]
is an isomorphism for every integer $n\geq 0$.
By Proposition \ref{boundarify.24},
it suffices to show that the morphism $\cF(X)\to \cF(X\times_{\A_P} Y)$ is an isomorphism if $Y\to \A_P\times \square^n$ is a composite of admissible blow-ups along smooth centers.
Since $\cF$ is invariant under admissible blow-up by Theorem \ref{logSH.12} and $\square$-invariant, we deduce this claim.
\end{proof}

\begin{prop}
\label{boundarify.22}
Let $S\in \Sch$.
Then $\logSH_{S^1}(S)$ is generated under colimits by $\Sigma^n\Sigma_{S^1}^\infty X_+$ for all $X\in \Sm/S$ and integers $n$.
\end{prop}
\begin{proof}
Let $\cC$ be the full subcategory of $\logSH_{S^1}(S)$ generated under colimits by those elements.
It suffices to show that for every $Y\in \SmlSm/S$, $\Sigma^\infty Y_+$ is in $\cC$.
We proceed by induction on the number $d$ of irreducible components of $\partial Y$.
The claim is clear if $d=0$.
Assume $d>0$.
There exists $Y_0\in \SmlSm/S$ and its smooth divisor $Z_0$ such that $\ul{Y_0}=\ul{Y}$ and $\partial Y=\partial Y_0+Z_0$.
By \cite[Proposition 7.4.5, Theorem 7.5.4]{BPO},
there is an isomorphism
\[
\Sigma_{S^1}^\infty Y/Y_0
\cong
\Sigma_{S^1}^\infty \P(\rN_{Z_0}Y_0\oplus \cO)/\P(\rN_{Z_0}Y_0).
\]
By induction,
we have $\Sigma_{S^1}^\infty Y_{0+},\Sigma_{S^1}^\infty \P(\rN_{Z_0}Y_0\oplus \cO)_+,\Sigma_{S^1}^\infty \P(\rN_{Z_0}Y_0)_+\in \cC$.
It follows that we have $\Sigma_{S^1}^\infty Y_+\in \cC$.
\end{proof}

\begin{thm}
\label{boundarify.20}
Let $\cF\in \Sh_\sNis(\SmlSm/S,\Sp)$.
If $\omega_\sharp \cF\in \Sh_\Nis(\Sm/S,\Sp)$ is logarithmic and $\cF\in \logSH_{S^1}(S)$,
then there is a natural isomorphism
\[
\Log \omega_\sharp \cF
\cong
\cF
\]
in $\logSH_{S^1}(S)$.
As a consequence,
for $X\in \SmlSm/S$,
we have a natural isomorphism of $\E_\infty$-rings in cyclotomic spectra
\[
\Log \THH(X) \cong \logTHH(X),
\]
and natural isomorphism of $\E_\infty$-rings
\[
\Log \TC(X) \cong \logTC(X).
\]
\end{thm}
\begin{proof}
Apply $\widetilde{\Log}$ to the counit $\omega^\sharp \omega_\sharp \cF\to \cF$ to obtain $\Log \omega_\sharp \cF\to \widetilde{\Log} \cF$.
By Proposition \ref{boundarify.13},
we have an isomorphism $\cF\cong \widetilde{\Log}\cF$.
Also, by Proposition \ref{boundarify.22}, it suffices to show that the induced morphism
\begin{equation}
\label{boundarify.20.1}
\Log \omega_\sharp \cF(X)
\to
\cF(X)
\end{equation}
is an isomorphism for $X\in \Sm/S$.
This follows from the assumption that $\omega_\sharp \cF$ is logarithmic.

Let $\CycSp$ denote the $\infty$-category of cyclotomic spectra.
The claims for $\cF=\THH$ follow from \eqref{boundarify.20.2} since the morphisms
\[
\Log \omega_\sharp \logTHH(X) \to \widetilde{\Log} \logTHH(X),
\text{ }
\logTHH(X) \to \widetilde{\Log} \logTHH(X)
\]
are morphisms of $\E_\infty$-rings in cyclotomic spectra due to the following reasons:
We used sifted colimits to form $\widetilde{\Log}$,
and the forgetful functors $\CAlg(\CycSp)\to \CycSp\to \Sp$ preserve sifted colimits by \cite[Corollary 3.2.3.2]{HA} and \cite[Corollary II.1.7]{zbMATH07009201}.
The proof of the claim for $\cF=\TC$ is similar.
\end{proof}

Next, we discuss a functorial property relating $u$ and $\widetilde{\Log}$.

\begin{lem}
\label{boundarify.14}
Let $S\in \Sch$.
Then the two natural transformations
\[
u\circ \widetilde{\Log},
\widetilde{\Log}\circ u
\colon
\widetilde{\Log}\to \widetilde{\Log}\widetilde{\Log}
\colon
\Sh_\sNis(\SmlSm/S)
\to
\Sh_\sNis(\SmlSm/S)
\]
are homotopic.
\end{lem}
\begin{proof}
Let $\cF\in \Sh_\sNis(\SmlSm/S,\Sp)$, and let $X\in \SmlSm/S$.
Consider the bisimplicial spectra
\[
\cA_{\bullet\bullet}:=
\begin{tikzcd}[column sep=small, row sep=small]
&
\vdots
\ar[d,shift right=1ex]
\ar[d]
\ar[d,shift left=1ex]
&
\vdots
\ar[d,shift right=1ex]
\ar[d]
\ar[d,shift left=1ex]
\\
\cdots
\ar[r,shift right=1ex]
\ar[r]
\ar[r,shift left=1ex]&
\displaystyle
\colim_{Y\in (\SBl_X^1)^\op}
\cF(Y)
\ar[d,shift right=0.5ex]
\ar[d,shift left=0.5ex]
\ar[r,shift right=0.5ex]
\ar[r,shift left=0.5ex]&
\displaystyle
\colim_{Y\in (\SBl_X^1)^\op}
\cF(Y)
\ar[d,shift right=0.5ex]
\ar[d,shift left=0.5ex]
\\
\cdots
\ar[r,shift right=1ex]
\ar[r]
\ar[r,shift left=1ex]&
\displaystyle
\colim_{Y\in (\SBl_X^0)^\op}
\cF(Y)
\ar[r,shift right=0.5ex]
\ar[r,shift left=0.5ex]&
\displaystyle
\colim_{Y\in (\SBl_X^0)^\op}
\cF(Y),
\end{tikzcd}
\]
\[
\cB_{\bullet\bullet}:=
\begin{tikzcd}[column sep=small, row sep=small]
&
\vdots
\ar[d,shift right=1ex]
\ar[d]
\ar[d,shift left=1ex]
&
\vdots
\ar[d,shift right=1ex]
\ar[d]
\ar[d,shift left=1ex]
\\
\cdots
\ar[r,shift right=1ex]
\ar[r]
\ar[r,shift left=1ex]&
\displaystyle
\colim_{Y\in (\SBl_X^1)^\op}
\cF(Y)
\ar[d,shift right=0.5ex]
\ar[d,shift left=0.5ex]
\ar[r,shift right=0.5ex]
\ar[r,shift left=0.5ex]&
\displaystyle
\colim_{Y\in (\SBl_X^0)^\op}
\cF(Y)
\ar[d,shift right=0.5ex]
\ar[d,shift left=0.5ex]
\\
\cdots
\ar[r,shift right=1ex]
\ar[r]
\ar[r,shift left=1ex]&
\displaystyle
\colim_{Y\in (\SBl_X^1)^\op}
\cF(Y)
\ar[r,shift right=0.5ex]
\ar[r,shift left=0.5ex]&
\displaystyle
\colim_{Y\in (\SBl_X^0)^\op}
\cF(Y),
\end{tikzcd}
\]
\[
\cC_{\bullet\bullet}:=
\begin{tikzcd}[column sep=small, row sep=small]
&
\vdots
\ar[d,shift right=1ex]
\ar[d]
\ar[d,shift left=1ex]
&
\vdots
\ar[d,shift right=1ex]
\ar[d]
\ar[d,shift left=1ex]
\\
\cdots
\ar[r,shift right=1ex]
\ar[r]
\ar[r,shift left=1ex]&
\displaystyle
\colim_{V\in (\SBl_X^1)^\op}
\colim_{W\in (\SBl_V^1)^\op}
\cF(W)
\ar[d,shift right=0.5ex]
\ar[d,shift left=0.5ex]
\ar[r,shift right=0.5ex]
\ar[r,shift left=0.5ex]&
\displaystyle
\colim_{V\in (\SBl_X^0)^\op}
\colim_{W\in (\SBl_V^1)^\op}
\cF(W)
\ar[d,shift right=0.5ex]
\ar[d,shift left=0.5ex]
\\
\cdots
\ar[r,shift right=1ex]
\ar[r]
\ar[r,shift left=1ex]&
\displaystyle
\colim_{V\in (\SBl_X^1)^\op}
\colim_{W\in (\SBl_V^0)^\op}
\cF(W)
\ar[r,shift right=0.5ex]
\ar[r,shift left=0.5ex]&
\displaystyle
\colim_{V\in (\SBl_X^0)^\op}
\colim_{W\in (\SBl_V^0)^\op}
\cF(W).
\end{tikzcd}
\]
For $m,n\in \N$,
we have the morphism $\cA_{mn}\to \cC_{mn}$ induced by the composite
\[
\cF(Y) \to \cF(Y\times \square^m) \to \colim_{V\in (\SBl_X^m)^\op}\colim_{W\in (\SBl_V^n)^\op}\cF(W)
\]
for $Y\in \SBl_X^n$ whose second morphism corresponds to $V=X\times \square^m$ and $W=Y\times \square^m$.
We have the morphism $\cB_{mn}\to \cC_{mn}$ induced by the composite
\[
\cF(Y) \to \cF(Y\times \square^n) \to \colim_{V\in (\SBl_X^m)^\op}\colim_{W\in (\SBl_V^n)^\op}\cF(W)
\]
for
$Y\in \SBl_X^n$ whose second morphism corresponds to $V=Y$ and $W=Y\times \square^n$.

When we evaluate $u\circ \widetilde{\Log}$ and $\widetilde{\Log}\circ u$ on $\cF(X)$,
we have the morphisms of bisimplicial colimits
\[
\colim_{m\in \bDelta^\op} \colim_{n \in \bDelta^\op}
\cA_{mn}
\to
\colim_{m\in \bDelta^\op} \colim_{n \in \bDelta^\op}
\cC_{mn},
\text{ }
\colim_{m\in \bDelta^\op} \colim_{n \in \bDelta^\op}
\cB_{mn}
\to
\colim_{m\in \bDelta^\op} \colim_{n \in \bDelta^\op}
\cC_{mn}.
\]
It suffices to show that these two are naturally homotopic.
A bisimplicial colimit is isomorphic to the diagonal simplicial colimit \cite[Definition 5.5.8.1, Lemma 5.5.8.4]{HTT}.
Hence it suffices to show that $\cA_{nn}\to \cC_{nn}$ and $\cB_{nn}\to \cC_{nn}$ are naturally homotopic for $n\in \N$.

Consider $Y\in \SBl_X^n$.
By Proposition \ref{boundary.73},
there exists a commutative square
\[
\begin{tikzcd}
Y'\ar[d]\ar[r]&
Y\ar[d]
\\
Y\times \square^n\ar[r]&
X\times \square^n,
\end{tikzcd}
\]
where $Y'\in \SBl_Y^n$,
and the morphism $Y\times \square^n\to X\times \square^n$ is a pullback of the composite $Y\to X\times \square^n\to X$.
Then the morphisms $\cA_{nn}\to \cC_{nn}$ and $\cB_{nn}\to \cC_{nn}$ agree since both have a factorization
\[
\cF(Y)
\to
\cF(Y')
\to
\colim_{V\in (\SBl_X^m)^\op}\colim_{W\in (\SBl_V^n)^\op}\cF(W)
\]
whose second morphism corresponds to $V=Y$ and $W=Y'$.
\end{proof}

\begin{lem}
\label{K.1}
Let $\cF\in \Sh_\sNis(\SmlSm/S,\Sp)$ and $\cG\in \logSH_{S^1}(S)$,
where $S\in \Sch$.
If $\widetilde{\Log}\cF\in \logSH_{S^1}(S)$,
then the map
\[
\alpha\colon
\pi_0\Hom_{\Sh_\sNis(\SmlSm/S,\Sp)}(\widetilde{\Log}\cF,\cG)\to \pi_0\Hom_{\Sh_\sNis(\SmlSm/S,\Sp)}(\cF,\cG)
\]
obtained by precomposing $u_{\cF}\colon  \cF\to \widetilde{\Log}\cF$ is an isomorphism.
\end{lem}
\begin{proof}
Consider the map
\[
\beta
\colon
\pi_0\Hom_{\Sh_\sNis(\SmlSm/S,\Sp)}(\cF,\cG)
\to
\pi_0\Hom_{\Sh_\sNis(\SmlSm/S,\Sp)}(\widetilde{\Log} \cF,\cG)
\]
sending $f\colon \cF\to \cG$ to the composite morphism
\[
\widetilde{\Log}  \cF \xrightarrow{\widetilde{\Log} f} \widetilde{\Log} \cG \xrightarrow{v_{\cG}} \cG,
\]
see Proposition \ref{boundarify.13} for $v_{\cG}$.
The composite morphism
\[
 \cF\xrightarrow{u_{ \cF}}
\widetilde{\Log}  \cF \xrightarrow{\widetilde{\Log} f} \widetilde{\Log} \cG \xrightarrow{v_{\cG}} \cG
\]
is homotopic to $f$ since we have a natural transformation $\id \to \widetilde{\Log}$ and $v_{\cG}$ is an inverse of $u_{\cG}$.
Hence we have $\alpha(\beta(f))\cong f$.

Let $g\colon \widetilde{\Log} \cF\to \cG$ be a morphism in $\Sh_\sNis(\SmlSm/S,\Sp)$.
We have the induced commutative diagram
\begin{equation}
\label{K.1.6}
\begin{tikzcd}
 \cF\ar[d,"u_{ \cF}"']\ar[r,"u_{ \cF}"]&
\widetilde{\Log} \cF\ar[d,"u_{\widetilde{\Log} \cF}"]\ar[r,"g"]&
\cG\ar[d,"u_\cG"]
\\
\widetilde{\Log} \cF\ar[r,"\widetilde{\Log}u_{ \cF}"]&
\widetilde{\Log}\widetilde{\Log} \cF\ar[r,"\widetilde{\Log}g"]&
\widetilde{\Log} \cG.
\end{tikzcd}
\end{equation}
To show $\beta(\alpha(g))\cong g$,
we need to show that the composite
\[
\widetilde{\Log} \cF
\xrightarrow{\widetilde{\Log}u_{ \cF}}
\widetilde{\Log}\widetilde{\Log} \cF\xrightarrow{\widetilde{\Log}g}
\widetilde{\Log}\cG
\xrightarrow{v_{\cG}}
\cG
\]
is homotopic to $g$.
For this,
it suffices to show $u_{\widetilde{\Log} \cF}\cong \widetilde{\Log}u_{ \cF}$ using the commutativity of the right square of \eqref{K.1.6}.
This is a consequence of Lemma \ref{boundarify.14}.
\end{proof}

For $S\in \Sch$,
let
\(
L_{\lmot} \colon
\Sh_\sNis(\SmlSm/S,\Sp)
\to
\logSH_{S^1}(S)
\)
denote the localization functor.
The constructions of $\widetilde{\Log}$ and $L_\lmot$ are different, but they are related as follows.

\begin{thm}
\label{boundarify.15}
Let $\cF\in \Sh_\sNis(\SmlSm/S,\Sp)$,
where $S\in \Sch$.
If $\widetilde{\Log}\cF\in \logSH_{S^1}(S)$,
then there exists a natural isomorphism 
\[
\widetilde{\Log} \cF\cong L_{\lmot}  \cF
\]
in $\logSH_{S^1}(S)$.
\end{thm}
\begin{proof}
By Lemma \ref{K.1},
we have a natural isomorphism
\[
\pi_0 \Hom_{\Sh_\sNis(\SmlSm/S,\Sp)}(\widetilde{\Log} \cF,L_\lmot   \cF)
\xrightarrow{\cong}
\pi_0 \Hom_{\Sh_\sNis(\SmlSm/S,\Sp)}(  \cF,L_\lmot   \cF).
\]
Let $p\colon   \cF\to L_\lmot   \cF$ be the canonical morphism,
and this corresponds to a morphism $f\colon \widetilde{\Log} \cF\to L_\lmot   \cF$.

Since $\widetilde{\Log} \cF\in \logSH_{S^1}(S)$ and $L_\lmot$ is a localization functor,
we also have a natural isomorphism
\[
\pi_0 \Hom_{\Sh_\sNis(\SmlSm/S,\Sp)}(L_{\lmot}  \cF,\widetilde{\Log} \cF)
\xrightarrow{\cong}
\pi_0 \Hom_{\Sh_\sNis(\SmlSm/S,\Sp)}(  \cF,\widetilde{\Log} \cF).
\]
Let $q\colon   \cF\to \widetilde{\Log} \cF$ be the canonical morphism,
and this corresponds to a morphism $g\colon L_\lmot   \cF\to \widetilde{\Log} \cF$.

We have $gp\cong q$ and $fq\cong p$.
Hence we have $fgp\cong p$ and $gfq\cong q$.
Use the natural isomorphisms
\begin{gather*}
\pi_0 \Hom_{\Sh_\sNis(\SmlSm/S,\Sp)}(\widetilde{\Log} \cF,\widetilde{\Log} \cF)
\xrightarrow{\cong}
\pi_0 \Hom_{\Sh_\sNis(\SmlSm/S,\Sp)}(  \cF,\widetilde{\Log} \cF),
\\
\pi_0 \Hom_{\Sh_\sNis(\SmlSm/S,\Sp)}(L_\lmot   \cF,L_\lmot   \cF)
\xrightarrow{\cong}
\pi_0 \Hom_{\Sh_\sNis(\SmlSm/S,\Sp)}(  \cF,L_\lmot   \cF)
\end{gather*}
to show $fg\cong \id$ and $gf\cong\id$.
Hence $f$ and $g$ are isomorphisms.
\end{proof}

\begin{cor}
\label{boundarify.16}
Let $\cF\in \Sh_\Nis(\Sm/S,\Sp)$,
where $S\in \Sch$.
If $\cF$ is logarithmic,
then there exists a natural isomorphism 
\[
\Log \cF\cong L_{\lmot}  \omega^\sharp \cF.
\]
\end{cor}
\begin{proof}
Combine Theorem \ref{boundarify.15} with Proposition \ref{boundary.47}.
\end{proof}

\begin{df}
\label{boundarify.17}
For $S\in \Sch$,
the \emph{log algebraic K-theory spectrum} is $\Log \KGL \in \logSH(S)$,
where
\[
\KGL:=(\Kth,\Kth,\ldots)\in \Sp_{\P^1}(\Sh_\Nis(\Sm/S,\Sp))
\]
is the $\P^1$-spectrum whose bonding morphisms $\Kth\to \Omega_{\P^1}\Kth$ are obtained by the projective bundle formula.
Observe that we have a natural isomorphism
\begin{equation}
\Log \KGL\cong \Sigma^{2,1}\Log \KGL.
\end{equation}

On the other hand,
if $S\in \Rg$,
then we have the definition
\[
\logKGL:=\omega^*\KGL\in \logSH(S)
\]
in \cite[Definition 6.5.6]{BPO2},
where $\KGL\in \SH(S)$ denotes the algebraic K-theory spectrum.
\end{df}

\begin{thm}
Let $S\in \Rg$.
Then we have a natural isomorphism
\[
\Log \KGL \cong \logKGL
\]
in $\logSH(S)$.
\end{thm}
\begin{proof}
This is an immediate consequence of Theorems \ref{boundarify.8} and \ref{boundarify.6}.
\end{proof}

\begin{prop}
\label{boundarify.18}
Let $f\colon X\to S$ be a morphism in $\Sch$.
Then we have natural isomorphisms
\[
f^*\Log \Kth\cong \Log \Kth,
\text{ }
f^*\Log \KGL\cong \Log \KGL.
\]
\end{prop}
\begin{proof}
Consider the induced commutative diagram
\[
\begin{tikzcd}
\Sh_\Nis(\Sm/S,\Sp)\ar[d,"f^*"']\ar[r,"\omega^\sharp"]&
\Sh_\Nis(\SmlSm/S,\Sp)\ar[d,"f^*"]\ar[r,"L_\lmot"]&
\logSH_{S^1}(S)\ar[d,"f^*"]
\\
\Sh_\Nis(\Sm/X,\Sp)\ar[r,"\omega^\sharp"]&
\Sh_\sNis(\SmlSm/X,\Sp)\ar[r,"L_\lmot"]&
\logSH_{S^1}(X).
\end{tikzcd}
\]
As a consequence of \cite[Proposition 1.3 in \S 4]{MV} (see also \cite[Proposition 3.14 in \S 4]{MV}),
we have $f^*\Kth\cong \Kth$ in $\Sh_\Nis(\Sm/X,\Sp)$.
Use the above diagram and Corollary \ref{boundarify.16} to obtain $f^*\Log \Kth\cong \Log \Kth$ in $\logSH_{S^1}(S)$.
From this,
we also obtain $f^*\Log \KGL\cong \Log \KGL$ in $\logSH(S)$.
\end{proof}

\begin{thm}
\label{boundarify.19}
Let $S\in \Sch$ and $X\in \Sm/S$.
Then there are natural isomorphisms
\[
\Log \Kth(X)
\cong
\hom_{\logSH_{S^1}(S)}(\Sigma_{S^1}^\infty X_+,\Log \Kth)
\cong
\hom_{\logSH(S)}(\Sigma_{\P^1}^\infty X_+,\Log \KGL).
\]
\end{thm}
\begin{proof}
The second isomorphism is clear.
The first isomorphism is a consequence of Proposition \ref{boundary.47} and Theorem \ref{boundarify.6}.
\end{proof}

\begin{const}
\label{boundarify.36}
Let $S\in \Sch$.
By \cite[Theorem 5.2.10]{BPO2}, we have an isomorphism
\[
\P^\infty \cong \rB \Gmlog
\]
in $\Sh_\sNis(\SmlSm/S,\Spc_*)[\square^{-1},\SmAdm^{-1}]$.
Using the multiplication $\Gmlog\times \Gmlog \to \Gmlog$ in \cite[Example 5.2.8]{BPO2},
we have the induced morphism $\P^\infty \otimes \P^\infty \to \P^\infty$ and hence a morphism
\[
\beta\colon \P^1\otimes \P^\infty \to \P^\infty
\]
in $\Sh_\sNis(\SmlSm/S,\Spc_*)[\square^{-1},\SmAdm^{-1}]$.
From this, we also get
\[
\beta\colon \Sigma^{2,1}\Sigma_{\P^1}^\infty (\P^\infty)_+\to \Sigma_{\P^1}^\infty (\P^\infty)_+
\]
in $\logSH(S)$.
\end{const}

The following result, which is a log motivic version of the Snaith theorem \cite{zbMATH03642323}, relies on a motivic version of the Snaith theorem established in \cite{AI}:

\begin{thm}
\label{boundarify.34}
For $S\in \Sch$,
there exists a natural isomorphism
\[
\Sigma_{\P^1}^\infty (\P^\infty)_+ [\beta^{-1}]
\cong
\Log \KGL
\]
in $\logSH(S)$.
\end{thm}
\begin{proof}
By Proposition \ref{boundarify.18},
it suffices to show the claim when $S=\Spec(\Z)$.
In particular, we may assume $S\in \Rg$.
Consider the algebraic K-theory spectrum $\KGL$ in $\MS_\Nis(S)$.
By \cite[Theorem 5.3.3]{AI},
we have an isomorphism
\[
\Sigma^\infty (\P^\infty)_+ [\beta^{-1}]
\cong
\KGL
\]
in $\MS_\Nis(S)$ with $\beta$ in \cite[Lemma 5.2.6]{AI}.
Apply $L_\lmot\omega^\sharp\colon \MS_\Nis(S)\to \logSH(S)$ to this isomorphism and use Theorem \ref{boundarify.6} and Corollary \ref{boundarify.16} to conclude.
Here, we can see that $L_\lmot \omega^\sharp$ sends $\beta$ in \cite[Lemma 5.2.6]{AI} to $\beta$ in Construction \ref{boundarify.36} using the open immersion $\G_m\to \Gmlog$ and hence the induced morphism $\rB \G_m \to \rB \Gmlog$.
\end{proof}

\section{Logarithmic cyclotomic trace}
\label{trace}

See \cite[Theorem 8.6.5]{BPO2} for the log cyclotomic trace under resolution of singularities.
We now define the log cyclotomic trace without resolution of singularities.

\begin{thm}
\label{Trace.2}
Let $X\in \RglRg$ (resp.\ $\SmlSm/S$ with $S\in \Sch$).
Then there exists a natural morphism of $\E_\infty$-rings
\[
\logTr\colon
\Log \Kth(X)
\to
\logTC(X)
\]
that coincides with the usual cyclotomic trace when $X$ has the trivial log structure.
We call $\logTr$ the \emph{log cyclotomic trace}.
\end{thm}
\begin{proof}
For every $Y\in \lSch$,
consider the composite morphism of $\E_\infty$-rings
\[
\Kth(\ul{Y})
\to
\TC(\ul{Y})
\to
\logTC(Y).
\]
Use this to obtain the morphism of $\E_\infty$-rings
\begin{equation}
\label{trace.2.1}
\Log \Kth(X)
=
\colim_{n\in \bDelta^\op}
\colim_{Y\in (\SBl_X^n)^\op}
\Kth(\ul{Y})
\to
\colim_{n\in \bDelta^\op}
\colim_{Y\in (\SBl_X^n)^\op }
\logTC(Y)=\widetilde{\Log}\logTC(X).
\end{equation}
The right-hand side is isomorphic to $\logTC(X)$ by Proposition \ref{boundarify.13}.
Hence we get $\logTr$.

Note that by \cite[Corollary 3.2.3.2]{HA},
the colimits in \eqref{trace.2.1} computed as $\E_\infty$-rings agree with the colimits in \eqref{trace.2.1} computed as spectra after forgetting the $\E_\infty$-structures.

Assume that $X$ has the trivial log structure.
Then we have the commutative square
\[
\begin{tikzcd}
\Kth(X)\ar[d,"\cong"']\ar[r,"\Tr"]&
\TC(X)\ar[d,"\cong"]
\\
\Log \Kth(X)\ar[r,"\logTr"]&
\logTC(X),
\end{tikzcd}
\]
so $\logTr$ can be identified with $\Tr$.
\end{proof}

\begin{rmk}
\label{trace.3}
Let $X\in \RglRg$.
Then by Theorem \ref{boundarify.29}, the log cyclotomic trace can be written as
\[
\logTr\colon \Kth(X-\partial X)\to \TC(X).
\]
We also have a commutative triangle
\[
\begin{tikzcd}
&
\logTC(X)\ar[d]
\\
\Kth(X-\partial X)\ar[ru,"\logTr"]\ar[r,"\Tr"]&
\logTC(X-\partial X).
\end{tikzcd}
\]
Hence the cyclotomic trace factors through the log cyclotomic trace,
i.e., the log cyclotomic trace is a refinement of the cyclotomic trace.
\end{rmk}

\begin{rmk}
Let $X\in \Rg$.
If $D$ is a strict normal crossing divisor on $X$,
then let $(X,D)$ denote the fs log scheme with underlying scheme $X$ and the Deligne-Faltings log structure associated $D$.
By \cite[Example III.1.11.9]{Ogu},
we have $(X,D)\in \RglRg$.

If $Z\to X$ is a closed immersion in $\Rg$ such that $D+Z$ is a strict normal crossing divisor on $X$,
then we use the convenient notation
\[
((X,D),Z)
:=
(X,D+Z).
\]
\end{rmk}

\begin{cor}
\label{trace.6}
Let $i\colon Z\to X$ be a closed immersion in $\Rg$.
Then there exists a morphism of fiber sequences
\[
\begin{tikzcd}
\Kth(Z)\ar[d,"\Tr"']\ar[r]&
\Kth(X)\ar[d,"\Tr"]\ar[r,"j^*"]&
\Kth(X-Z)\ar[d,"\logTr"]
\\
\TC(Z)\ar[r]&
\TC(X)\ar[r,"p^*"]&
\logTC(\Bl_ZX,E),
\end{tikzcd}
\]
where $E$ is the exceptional divisor,
$j\colon X-Z\to X$ is the open immersion,
and $p\colon (\Bl_Z X,E)\to X$ is the projection.
\end{cor}
\begin{proof}
Consider the $\square$-deformation space
\[
\rD_Z X
:=
\Bl_Z(X\times \square)-\Bl_Z(X\times \{0\}).
\]
Then we have the induced commutative diagram
\[
\Gys(Z\to X)
:=
\begin{tikzcd}[column sep=small]
(\Bl_Z X,E)\ar[d]\ar[r]&
(\Bl_{Z\times \square}(\rD_Z X),E^D)\ar[r,leftarrow]\ar[d]&
(\Bl_Z (\rN_Z X),\E^N)\ar[d]
\\
X\ar[r]&
\rD_Z X\ar[r,leftarrow]&
\rN_Z X,
\end{tikzcd}
\]
where $E^D$ and $E^N$ are the exceptional divisors,
see \cite[Definition 7.5.1]{BPO} and \cite[Remark 2.8]{regGysin}.

By \cite[Theorem 13.6.3]{CD12} and Theorem \ref{boundarify.29},
the squares in $\Log \Kth(\Gys(Z\to X))$ are cartesian.
Moreover, by \cite[Corollary 4.7]{regGysin},
the squares in $\logTC(\Gys(Z\to X))$ are cartesian.
Together with Proposition \ref{boundary.19} and Theorem \ref{Trace.2},
we have a morphism of fiber sequences
\[
\begin{tikzcd}
\fib(\Kth(\P(\rN_Z X\oplus \cO))\to \Kth(\P(\rN_Z X)))\ar[d,"\Tr"']\ar[r]&
\Kth(X)\ar[d,"\Tr"]\ar[r,"j^*"]&
\Kth(X-Z)\ar[d,"\logTr"]
\\
\fib(\TC(\P(\rN_Z X\oplus \cO))\to \TC(\P(\rN_Z X)))\ar[r]&
\TC(X)\ar[r,"p^*"]&
\logTC(\Bl_ZX,E).
\end{tikzcd}
\]
To conclude,
observe that the projective bundle formulas for $\Kth$ and $\TC$ are compatible via the cyclotomic trace by \cite[Theorem 1.5]{BM12}.
\end{proof}

See \cite[\S 2.6]{BLPO} for the animated version of log rings.

\begin{lem}
\label{trace.15}
Let $(R,P)\to (A,M),(B,N)$ be maps of animated log rings. 
Then the induced morphism of $\E_\infty$-rings in cyclotomic spectra
\begin{equation}
\label{trace.15.1}
\THH(A,M)\otimes_{\THH(R,P)}\THH(B,N)
\to
\THH(A\otimes_R B,M\oplus_P N)
\end{equation}
is an isomorphism.
\end{lem}
\begin{proof}
The forgetful functor from the category of cyclotomic spectra to spectra is conservative by \cite[Corollary II.1.7]{zbMATH07009201}.
Hence it suffices to show that \eqref{trace.15.1} is an isomorphism as $\E_\infty$-rings.
We have an isomorphism of $\E_\infty$-rings
\[
\THH(R,P)
\cong
\THH(R)\otimes_{\THH(\Sph[P])}\Sph[P\oplus \rB P^\gp]
\]
by the animated versions of \cite[Lemma 3.17, Remark 8.12]{zbMATH05610455}, see also \cite[Proposition 2.21]{BLPO}.
Hence it suffices to show that the functors $R\mapsto \THH(R)$, $P\mapsto \THH(\Sph[P])$, and $P\mapsto \Sph [P\oplus \rB P^\gp]$ preserve colimits.
Since $\THH$ is defined by a colimit,
$R\mapsto \THH(R)$ preserves colimit.
To conclude,
observe that $\Sph[-]$, the group completion, and the bar construction preserve colimits.
\end{proof}

\begin{cor}
\label{trace.7}
Let $X\in \RglRg$, and let $Z\to \ul{X}$ be a closed immersion in $\Rg$ such that $\partial X+Z$ is a strict normal crossing divisor on $\ul{X}$.
Then there exists a commutative diagram
\[
\begin{tikzcd}
\Kth(Z-\partial Z)\ar[d,"\logTr"']\ar[r]&
\Kth(X-\partial X)\ar[d,"\logTr"]\ar[r]&
\Kth(X-(\partial X\cup Z))\ar[d,"\logTr"]
\\
\logTC(Z)\ar[r]&
\logTC(X)\ar[r]&
\logTC(X,Z)
\end{tikzcd}
\]
whose rows are cofiber sequences.
\end{cor}
\begin{proof}
As in Corollary \ref{trace.6},
consider the induced commutative diagram
\[
\Gys(Z\to X)
:=
\begin{tikzcd}
(X,Z)\ar[d]\ar[r]&
(\rD_Z X,Z\times \square)\ar[r,leftarrow]\ar[d]&
(\rN_Z X,Z)\ar[d]
\\
X\ar[r]&
\rD_Z X\ar[r,leftarrow]&
\rN_Z X.
\end{tikzcd}
\]
It suffices to show that the squares in $\E(\Gys(Z\to X))$ are cartesian for $\E=\Log \Kth,\logTC$.
The claim holds for $\Log \Kth$ by \cite[Theorem 13.6.3]{CD12} and Theorem \ref{boundarify.29}.
To show the claim for $\logTC$,
it suffices to show that the squares in $\E(\Gys(Z\to X))$ for $\E=\logTHH$ are cocartesian as cyclotomic spectra.
Since the forgetful functor from the $\infty$-category of cyclotomic spectra to the $\infty$-category of spectra is conservative and exact by \cite[Corollary II.1.7]{zbMATH07009201},
it suffices to show that the squares in $\E(\Gys(Z\to X))$ for $\E=\logTHH$ are cocartesian as spectra.

We can work Zariski locally on $Y:=(X,Z)$ for this claim.
Hence we may assume that there exists a strict morphism $Y\to Y_0:=\A_\N^{n+1}$ given by $\Spec(f)$ for some map of log rings $f\colon (\Z[\N^{n+1}],\N^{n+1})\to (A,M)$.
We set $X_0:=\A_\N^n \times \A^1$ and $Z_0:=\A_\N^n \times \{0\}$ so that $Y_0\cong (X_0,Z_0)$.
As in \cite[Construction 4.2]{regGysin},
$\Gys(Z\to X)$ is a derived pullback of $\Gys(Z_0\to X_0)$.
Together with Lemma \ref{trace.15},
we see that $\E(\Gys(Z\to X))$ is a pushout of $\E(\Gys(Z_0\to X_0))$ as $\E_\infty$-rings.
Hence we reduce to the case of $Z_0\to X_0$.
In this case, we have $Y\in \SmlSm/\Z$,
so \cite[Theorem 7.5.4]{BPO} finishes the proof.
\end{proof}

Now, we collect several results relating $\Kth$ and $\TC$ of log schemes.

\begin{thm}
\label{trace.16}
Let $K$ be a local field of residual characteristic $p>0$.
Then there is an isomorphism of $\E_\infty$-rings
\[
\Kth(K;\Z_p)
\cong
\TC((\cO_K,\langle \pi \rangle);\Z_p),
\]
where $\cO_K$ denotes the ring of integers of $K$, and $\pi$ is the uniformizer of $K$.
\end{thm}
\begin{proof}
Combine \cite[Theorem D]{zbMATH00938856} and Corollary \ref{trace.6}.
\end{proof}

The following removes the assumption of resolution of singularities in \cite[Proposition 7.10]{BPO3}.
Let $L_\seta^\wedge$ denote the strict \'etale hypersheafification functor.

\begin{thm}
\label{trace.9}
Let $k$ be a perfect field of characteristic $p>0$.
For $X\in \SmlSm/k$,
there is an isomorphism of $\E_\infty$-rings
\[
(L_\seta^\wedge \Log \Kth(X))_p^\wedge
\cong
\logTC(X;\Z_p).
\]
\end{thm}
\begin{proof}
Since the presheaf $\logTC$ on $\lSch$ is a strict \'etale hypersheaf by \cite[Proposition 8.3.13]{BPO2},
we have the induced natural morphism of $\E_\infty$-rings
\[
\logTr\colon
L_\seta^\wedge \Log \Kth(X)\to \logTC(X).
\]
Take the $p$-completions on both sides to obtain
\[
\logTr\colon (L_\seta^\wedge \Log \Kth(X))_p^\wedge \to \logTC(X;\Z_p).
\]

If $X$ has the trivial log structure,
then the claim is due to Geisser-Hesselholt \cite[Theorem 4.2.6]{zbMATH01421292}.

For general $X$,
suppose that $Z\in \Rg$ is a closed subscheme of $\ul{X}$ such that $\partial X+Z$ is a strict normal crossing divisor on $\ul{X}$.
Corollary \ref{trace.7} yields a morphism of fiber sequences
\[
\begin{tikzcd}
(L_\seta^\wedge \Log \Kth(Z))_p^\wedge\ar[d,"\logTr"']\ar[r]&
(L_\seta^\wedge \Log \Kth(X))_p^\wedge\ar[d,"\logTr"]\ar[r]&
(L_\seta^\wedge \Log \Kth(X,Z))_p^\wedge\ar[d,"\logTr"]
\\
\logTC(Z;\Z_p)\ar[r]&
\logTC(X;\Z_p)\ar[r]&
\logTC((X,Z);\Z_p).
\end{tikzcd}
\]
We finish the proof by induction on the number of irreducible components of $\partial X$.
\end{proof}

\begin{thm}
\label{trace.10}
Let $A$ be a henselian discrete valuation ring of mixed characteristic $(0,p)$.
For proper $X\in \SmlSm/A$,
there is an isomorphism of $\E_\infty$-rings
\[
(L_\seta^\wedge \Log \Kth(X))_p^\wedge
\cong
\logTC(X;\Z_p).
\]
\end{thm}
\begin{proof}
If $X$ has the trivial log structure,
then this is due to Geisser-Hesselholt \cite[Theorem A]{GeissHesTrans}.
Argue as in Theorem \ref{trace.9} to conclude.
\end{proof}

\begin{thm}
\label{trace.8}
Let $R$ be a strictly henselian regular local ring with residue characteristic $p>0$,
and let $x_1,\ldots,x_n$ be a regular sequence of $R$ such that the divisor $(x_1)+\cdots +(x_n)$ is strict normal crossing.
Then there is a natural isomorphism of $\E_\infty$-rings
\[
\Kth(R[1/x_1,\ldots,1/x_n];\Z_p)
\cong
\logTC(R,(x_1)+\cdots+(x_n);\Z_p).
\]
\end{thm}
\begin{proof}
The case $n=0$ without regularity assumption is due to Clausen-Mathew-Morrow \cite[Theorem A]{GeissHesTrans}.
Argue as in Theorem \ref{trace.9} to conclude.
\end{proof}

\begin{const}
Let $S$ be the spectrum of a quasi-syntomic ring in the sense of \cite[Definition 4.10]{BMS19}.
Consider the very effective slice filtration $\tilde{\fil}_n\colon \logSH(S)\to \logSH(S)$ with graded pieces $\tilde{\slice}_n$ for integers $n$ in \cite[Construction 2.4]{BPO3} and the Bhatt-Morrow-Scholze filtration $\Fil_n^\BMS\blogTC$ in \cite[Definition 5.10]{BPO3}.
By \cite[Lemma 2.12, Theorem 5.17]{BPO3},
the log cyclotomic trace induces a natural morphism
\[
\tilde{\fil}_n\Log\KGL \to \Fil_n^\BMS \blogTC
\]
in $\logSH(S)$ for every integer $n$.
Its graded pieces are
\[
\tilde{\slice}_n \Log \KGL \to \MZ_p^\syn(n)[2n],
\]
see \cite[Definition 5.12, Proposition 5.13]{BPO3} for the right-hand side.
In \cite[Theorem 1.4(1)]{BPO3}, the construction of such a morphism assumed that $S$ is the spectrum of a perfect field admitting resolution of singularities,
but now this condition is no longer needed.
Unfortunately, we still need resolution of singularities to identify $\tilde{\slice}_0\Log\KGL$ with the motivic cohomology spectrum,
i.e., we have not improved \cite[Theorem 1.1]{loghomotopy} and hence \cite[Theorem 3.4]{BPO3}.

It would be also an interesting question to compare $\tilde{\slice}_0\Log\KGL$ with the Elmanto-Morrow motivic cohomology \cite{Elmanto_Morrow}. 
\end{const}

\appendix

\section{Pointed monoids}

A monoid in this paper means a commutative monoid. Let $0$ denote the identity of a monoid.
As rings are building blocks for schemes,
pointed monoids are building blocks for monoid schemes.
In this section,
we review the theory of pointed monoids following \cite[\S 1]{CHWW}.
See also \cite[\S 1]{Ogu} for the theory of monoids.

\begin{df}
\label{mon.1}
A \emph{pointed monoid} is a monoid $A$ with a fixed element $\infty$\footnote{With the multiplicative notation, $0$ denotes the base point, and $1$ denotes the identity.},
which we call the \emph{base point of $A$},
such that $x+\infty=\infty$ for every $x\in A$.
A \emph{map of pointed monoids} is a map of monoids preserving the base points.
\end{df}

\begin{df}
\label{mon.2}
For a monoid $M$,
let $\F_1[M]$ denote the pointed monoid $M\cup \{\infty\}$ with the base point $\infty$ such that its monoid operation extends the monoid operation of $M$.
We set $\F_1:=\F_1[0]$, where $0$ means the zero monoid.
\end{df}

\begin{df}
\label{mon.3}
A monoid $M$ is \emph{integral} if $a+b=a+c$ with $a,b,c\in M$ implies $b=c$.
In this case,
$M$ is a submonoid of the group completion $M^\gp$.
A pointed monoid $A$ is \emph{cancellative} if $a+b=a+c$ with $a\in A-\{\infty\}$ and $b,c\in A$ implies $b=c$.

If $A$ is cancellative,
then $A-\{\infty\}$ is an integral monoid,
and we have $A\cong \F_1[A-\{\infty\}]$.
On the other hand,
if $M$ is an integral monoid,
then $\F_1[M]$ is cancellative.
\end{df}

\begin{df}
A pointed monoid $A$ is \emph{torsion free} if $na=nb$ with $a,b\in A$ and integer $n>0$ implies $a=b$.

If $A$ is cancellative and $(A-\{\infty\})^\gp$ is a torsion free abelian group,
then $A$ is torsion free.
\end{df}

\begin{df}
\label{mon.11}
A monoid $M$ is \emph{saturated} if it is integral and satisfies the following condition: For $x\in M^\gp$, if there exists an integer $n>0$ such that $nx\in M$, then we have $x\in M$.
A pointed monoid $A$ is \emph{normal} if it is cancellative and $A-\{\infty\}$ is saturated.

For a monoid $M$,
we have the saturation $M^\sat$ \cite[Proposition I.1.3.5(1)]{Ogu}.
For a cancellative monoid $A$,
we have the normalization
\begin{equation}
\label{mon.11.1}
A^\nor
:=
\F_1[(A-\{\infty\})^\sat].
\end{equation}
\end{df}

\begin{rmk}
Let $\Mon$ be the category of monoids,
and let $\Mon_*$ be the category of pointed monoids.
Then we have the adjoint functors
\[
\F_1[-]\colon \Mon \rightleftarrows \Mon_*:U,
\]
where $U$ is the forgetful functor.
\end{rmk}

\begin{rmk}
Let $\Ring$ be the category of commutative rings.
Then we have the adjoint functors
\[
\Z[-]\colon \Mon \rightleftarrows \Ring:U,
\]
where $U$ is the forgetful functor,
and $\Z[M]$ denotes the monoid ring for a monoid $M$.
Note that an element of $\Z[M]$ is written as a formal sum $\Sigma_{t\in M} a_tx^t$ with $a_t\in \Z$.
\end{rmk}

\begin{prop}
\label{mon.13}
Consider the functor
\[
U\colon \Ring\to \Mon_*
\]
forgetting the additive structure and setting $0$ in the ring as the base point $\infty$.
This admits a left adjoint sending a pointed monoid $A$ to the quotient ring $\Z[A]/(x^\infty)$.
\end{prop}
\begin{proof}
Let $R$ be a ring.
We have the natural map
\begin{equation}
\label{mon.13.1}
\Hom_{\Mon_*}(A,U(R))
\to
\Hom_{\Ring}(\Z[A]/(x^\infty),R)
\end{equation}
sending a map $\theta\colon A\to U(R)$ to the map $\Z[A]/(x^\infty)\to R$ given by $x^t\mapsto \theta(t)$.
We have the natural map
\begin{equation}
\label{mon.13.2}
\Hom_{\Ring}(\Z[A]/(x^\infty),R)
\to
\Hom_{\Mon_*}(A,U(R))
\end{equation}
sending a map $f\colon \Z[A]/(x^\infty)\to R$ to the map $A\to U(R)$ given by $t\mapsto f(x^t)$.
Check that \eqref{mon.13.1} and \eqref{mon.13.2} are inverses to each other.
\end{proof}

\begin{const}
\label{mon.4}
We can unify $\Mon_*$ and $\Ring$ into a single category $\Mon_*\cup \Ring$,
whose objects are the pointed monoids and rings,
and whose hom sets are given as follows:
\[
\Hom_{\Mon_*\cup \Ring}(A,B)
:=
\left\{
\begin{array}{ll}
\Hom_{\Mon_*}(A,B) & \text{if }A,B\in \Mon_*,
\\
\Hom_{\Ring}(A,B) & \text{if }A,B\in \Ring,
\\
\Hom_{\Mon_*}(A,U(B)) & \text{if }A\in \Mon_*\text{ and }B\in \Ring_*,
\\
\emptyset & \text{if }A\in \Ring\text{ and }B\in \Mon_*.\end{array}
\right.
\]
Let $\otimes$ denote the coproduct in $\Mon_*\cup \Ring$.
Observe that $\F_1$ is an initial object of $\Mon_*\cup \Ring$,
so we have $\otimes = \otimes_{\F_1}$.
\end{const}

\begin{exm}
\label{mon.12}
By \cite[Explanations after Remark 1.7.1]{CHWW},
we have the following descriptions of coproducts of pointed monoids.

For pointed monoids $A$ and $B$,
$A\otimes B$ is the smash product,
i.e., the quotient set of $A\times B$ such that $(a,\infty)$ and $(\infty,b)$ are identified with $\infty$ for all $a\in A$ and $b\in B$.
For maps of pointed monoids $A\to B,C$, $B\otimes_A C$ is the quotient of $B\otimes C$ by the congruence relation generated by $(b+a,c)=(b,a+c)$ for $a\in A$, $b\in B$, and $c\in C$.
\end{exm}

\begin{exm}
\label{mon.5}
For monoids $M$ and $N$,
we have a natural isomorphism
\[
\F_1[M]\otimes_{\F_1} \F_1[N]
\cong
\F_1[M\oplus N],
\]
where $M\oplus N$ denotes the coproduct in the category of monoids.
\end{exm}

\begin{exm}
\label{mon.6}
For a pointed monoid $A$ and a ring $R$,
consider the monoid ring $R[A]$ whose elements are written as a formal sum $\Sigma_{t\in A} a_t x^t$ with $a_t\in R$.
Using Proposition \ref{mon.13},
one can show
\[
A\otimes_{\F_1}R\cong R[A]/(x^\infty).
\]
For a monoid $M$,
we also have
\[
\F_1[M]\otimes_{\F_1} R
\cong
R[M].
\]
\end{exm}

\begin{rmk}
\label{mon.7}
In the category $\Mon_*\cup \Ring$,
we also have
\[
A\otimes_{\F_1}B\cong A\otimes_\Z B.
\]
for $A,B\in \Ring$.
However, this is against the philosophy that $\Z\otimes_{\F_1}\Z$ should be neither a ring nor pointed monoid, see e.g.\ \cite[End of \S 1]{zbMATH06894815}.
Hence we avoid using the notation $A\otimes_{\F_1}B$.
The reason why we introduce $\Mon_*\cup \Ring$ is to justify the notation like $R\otimes_A B$ with $A,B\in \Mon_*$ and $R\in \Ring$.
\end{rmk}

\begin{df}
\label{mon.8}
An \emph{ideal $I$} of a pointed monoid $A$ is a subset satisfying the following two conditions:
\begin{enumerate}
\item[(i)] $\infty\in I$.
\item[(ii)] $r\in I$ and $a\in A$ implies $r+a\in I$.
\end{enumerate}

If $A=\F_1[M]$ for some monoid $M$,
then there is a one-to-one correspondence between ideals of $M$ in the sense of \cite[Definition I.1.4.1]{Ogu} and ideals of $A$ given by $(I\subset M)\mapsto(I\cup \{\infty\}\subset A)$.

An ideal $I$ of $A$ is \emph{prime} if $I\neq A$ and $A-I$ is closed under $+$.
\end{df}

\begin{df}
For a pointed monoid $A$,
let $A^*$ denote the set of units of $A$,

A map of pointed monoids $A\to B$ is \emph{local} if it is local as a map of monoids in the sense of \cite[Definition I.4.1.1(1)]{Ogu}, i.e., $\theta^{-1}(B^*)=A^*$.
\end{df}

\begin{df}
For a pointed monoid $A$ and its nonempty subset $S$ closed under $+$,
the \emph{localization $A_S$} is the quotient set of $A\times S$ under the following equivalence relation: $(a,s)\sim (a',s')$ if $a+s'+s''=a'+s+s''$ for some $s''\in S$.
Observe that $A_S$ is a pointed monoid whose monoid operation is the induced one.

For a prime ideal $\mathfrak{p}$ of $A$,
we set $A_\mathfrak{p}:=A_{A-\mathfrak{p}}$.
\end{df}

\section{Monoid schemes}
\label{msch}

In this section,
we review the theory of monoid schemes following \cite{CHWW}.

\begin{df}
\label{msch.2}
A \emph{pointed monoidal space}\footnote{This is called \emph{monoid space} in \cite[\S 2]{CHWW}.
The reason for our choice of the terminology is that we also use the terminology \emph{monoidal spaces}, see Definition \ref{msch.1}.} $X$ is a topological space $X$ equipped with a sheaf of pointed monoids $\cO_X$ on $X$.
A \emph{morphism of pointed monoidal spaces} $f\colon X\to S$ is a morphism of topological spaces $f\colon X\to S$ equipped with a morphism of sheaves of pointed monoids $f^\sharp\colon  \cO_S\to f_*\cO_X$ such that the induced map of the stalks $f^\sharp_x\colon \cO_{S,f(x)}\to \cO_{X,x}$ is local.
\end{df}

\begin{df}
\label{msch.3}
For a pointed monoid $A$, the \emph{affine monoid scheme} $\Spec(A)$ is defined as follows.
The underlying topological space of $\Spec(A)$ is the set of primes ideals of $A$ endowed with the Zariski topology, which is generated by the base consisting $D(f)$ for all $f\in A$ with
\[
D(f):=\{\mathfrak{p}\in \Spec(A):f\notin \mathfrak{p}\}.
\]
The assignment $D(f)\mapsto A_f$ defines a presheaf on the base,
and let $\cO_{\Spec(A)}$ be the associated sheaf on $\Spec(A)$.
\end{df}

\begin{df}
\label{msch.4}
A \emph{monoid scheme} is a pointed monoidal space admitting an open covering $\{U_i\}_{i\in I}$ such that each $U_i$ is an affine monoid scheme.
\end{df}

\begin{rmk}
The category of monoid schemes admits fiber products by \cite[Proposition 3.1]{CHWW}. Its proof also shows that fiber products of affine monoid schemes is affine.
\end{rmk}

\begin{prop}
Let $A$ be a pointed monoid,
and let $X$ be a pointed monoidal space.
Then there is a natural isomorphism
\[
\Hom(X,\Spec(A))
\cong
\Hom(A,\Gamma(X,\cO_X)).
\]
\end{prop}
\begin{proof}
Argue as in \cite[Proposition II.1.2.2]{Ogu}.
\end{proof}

\begin{df}
For a pointed monoid $A$,
we define $\Spec(A)\otimes_{\F_1}\Z:=\Spec(A\otimes_{\F_1}\Z)$.
For a monoid scheme $X$,
glue this local construction to define $X\otimes_{\F_1} \Z$.
\end{df}

\begin{const}
\label{msch.6}
We can unify the categories of schemes and monoid schemes into a single category, whose objects are the schemes and monoid schemes,
and whose hom sets are given as follows:
\begin{enumerate}
\item[(i)]
If $S$ and $X$ are both schemes or both monoid schemes,
then $\Hom(X,S)$ is the original one.
\item[(ii)]
If $S$ is a scheme and $X$ is a monoid scheme,
then $\Hom(X,S):=\emptyset$.
\item[(iii)]
If $S$ is a monoid scheme and $X$ is a scheme,
then
\[
\Hom(X,S):=\Hom(X,S\otimes_{\F_1}\Z).
\]
\end{enumerate}

In this category, the product $\times$ agrees with the fiber product $\times_{\Spec(\F_1)}$ since $\Spec(\F_1)$ is a final object.

For a monoid scheme $X$ and a ring $R$,
the \emph{$R$-realization of $X$} is
$X_R:=X\times_{\Spec(\F_1)}\Spec(R)$.
The local description of $X_R$ is in Example \ref{mon.6}.
\end{const}

\begin{df}
A morphism of monoid schemes $j\colon U\to X$ is an \emph{open immersion} if $j$ is a homeomorphism onto an open subset of $X$ and the induced morphism $j^{-1}\cO_X\to \cO_U$ is an isomorphism.
Observe that the class of open immersions is closed under compositions and pullbacks.
When $U$ is regarded as an open subset of $X$,
we say that $U$ is an \emph{open subscheme of $X$}.
\end{df}

\begin{df}
A morphism of monoid schemes $Z\to X$ is a \emph{closed immersion} if for every open affine subscheme $U$ of $X$,
$Z\times_X U$ is affine, and the induced map $\Gamma(U,\cO_U)\to \Gamma(Z\times_X U,\cO_{Z\times_X U})$ is surjective.

Observe that in this case, $Z\to X$ is a homeomorphism onto its image as topological spaces,
which was included as a condition of closed immersions in \cite[Definition 2.5]{CHWW}.
Indeed, we can work locally on $X$ for the claim,
and then use \cite[Lemma 2.7]{CHWW}.

When $Z$ is regarded as a subset of $X$,
we say that $Z$ is a \emph{closed subscheme of $X$}.
\end{df}

\begin{prop}
\label{msch.20}
The class of closed immersions is closed under compositions and pullbacks.
\end{prop}
\begin{proof}
Let $W\to Z\to X$ be closed immersions of monoid schemes.
To show that the composite $W\to X$ is a closed immersion,
we can work Zariski locally on $X$.
Hence we may assume that $X$ is affine.
To conclude,
observe that the composite of surjective maps is surjective.

Similarly,
to show that the class of closed immersions is closed under pullbacks,
it suffices to show that for every surjective map of pointed monoids $A\to B$ and a map of pointed monoids $\theta\colon A\to A'$,
the pushout $A'\to A'\otimes_A B$ is surjective.
For every $(a',b)\in A'\otimes_A B$ such that $a',b\neq \infty$,
there exists $a\in A$ whose image in $B$ is $b$.
Then we have $(a',b)=(a'+\theta(a),0)$ by Example \ref{mon.12},
which shows the desired claim.
\end{proof}

\begin{df}
\label{logSH.9}
The \emph{Zariski distinguished square in $\Sm/\F_1$} is the set of cartesian squares
\begin{equation}
\label{logSH.9.1}
\begin{tikzcd}
W\ar[d]\ar[r]&
V\ar[d]
\\
U\ar[r]&
X
\end{tikzcd}
\end{equation}
in $\Sm/\F_1$ such that $U\to X$ and $V\to X$ are jointly surjective open immersions.
The \emph{Zariski cd-structure on $\Sm/\F_1$} is the set of Zariski distinguished squares in $\Sm/\F_1$.
The associated topology is the \emph{Zariski topology}.
We refer to \cite[Definition 2.1]{Vcdtop} for the notion of cd-structures.
\end{df}

\begin{rmk}
\label{logSH.11}
Consider a Zariski distinguished square in $\Sm/\F_1$ of the form \eqref{logSH.9.1}.
Then the induced square
\[
\begin{tikzcd}
W\ar[d]\ar[r]&
V\ar[d]
\\
W\times_U W\ar[r]&
V\times_X V
\end{tikzcd}
\]
is a Zariski distinguished square,
where the vertical morphisms are the diagonal morphisms.
Furthermore,
any pullback of Zariski distinguished squares in $\Sm/\F_1$ is a Zariski distinguished square.
In particular,
the Zariski cd-structure on $\Sm/\F_1$ is complete and regular in the sense of \cite[Definitions 2.3, 2.10]{Vcdtop} by \cite[Lemmas 2.5, 2.11]{Vcdtop},
see \cite[Theorem 12.7]{CHWW}.
Furthermore,
it is bounded in the sense of \cite[Definition 2.22]{Vcdtop} by \cite[Theorem 12.8]{CHWW}.
\end{rmk}

\begin{rmk}
\label{msch.9}
Let $A$ be a pointed monoid.
Then every affine open subscheme of $\Spec(A)$ is of the form $\Spec(A_\mathfrak{p})$ for some prime ideal of $A$ by \cite[Lemma 2.4]{CHWW}.
Furthermore,
$A_\mathfrak{p}=A_s$ for some element $s\in A$ by \cite[Lemma 1.3]{CHWW}.
\end{rmk}

\begin{df}
\label{msch.7}
A monoid scheme $X$ is \emph{cancellative} (resp.\ \emph{normal}, \emph{torsion free}, \emph{locally of finite type}) if the stalk $\cO_{X,x}$ is cancellative (resp.\ normal, torsion free, finitely generated) for every $x\in X$.
This is equivalent to saying that there exists an open neighborhood $U$ of $x$ such that $\Gamma(U,\cO_U)$ is cancellative (resp.\ normal, resp.\ torsion free, resp.\ finitely generated) for every $x\in X$ by Remark \ref{msch.9}.

A monoid scheme $X$ is \emph{of finite type} if $X$ is quasi-compact and locally of finite type.
In this case,
observe that the underlying topological space of $X$ has only a finite number of points.

For a cancellative monoid scheme $X$,
we have the \emph{normalization} $X_\mathrm{nor}$ of $X$ whose local description is given as \eqref{mon.11.1}.
\end{df}

\begin{df}
An integral monoid $M$ is \emph{valuative} if $x\in M^\gp$ implies either $x\in M$ or $-x\in M$. As observed in \cite[p.\ 14]{Ogu},
every valuative monoid is saturated.
\end{df}

\begin{df}
\label{msch.11}
Let $f\colon X\to S$ be a morphism of monoid schemes.
\begin{enumerate}
\item[(1)]
$f$ is \emph{separated} if the diagonal morphism $X\to X\times_S X$ is a closed immersion.
\item[(2)]
$f$ satisfies the \emph{valuative criterion of separateness} (resp.\ \emph{properness}) if for every commutative square of monoid schemes
\[
\begin{tikzcd}
\Spec(\F_1[V^\gp])\ar[d]\ar[r]&
X\ar[d,"f"]
\\
\Spec(\F_1[V])\ar[ru,dashed]\ar[r]&
S
\end{tikzcd}
\]
with a valuative monoid $V$,
there exists at most one morphism (resp.\ a unique morphism) of monoid schemes $\Spec(\F_1[V])\to X$ such that the resulting diagram still commutes.
\end{enumerate}

When $X$ and $S$ are of finite type,
$f$ is \emph{proper} if it satisfies the valuative criterion of properness.
\end{df}

\begin{rmk}
\label{msch.19}
Let $k$ be a field.
For a morphism of monoid schemes $f\colon X\to S$,
$f$ satisfies the valuative criterion of separateness (resp.\ properness) if and only if the $k$-realization $f_k\colon X_k\to S_k$ satisfies the valuative criterion of separateness (resp.\ properness), see \cite[Theorem 8.9]{CHWW}.
\end{rmk}

\begin{exm}
\label{logmonoid.15}
A closed immersion of monoid schemes satisfies the valuative criterion of properness by \cite[Corollary 8.6]{CHWW}.
For every cancellative monoid scheme of finite type $X$,
the canonical morphism $X^\nor\to X$ satisfies the valuative criterion of properness by \cite[Propositions 6.3, 8.5]{CHWW}.
\end{exm}

\begin{prop}
Let $f\colon X\to S$ be a morphism of monoid schemes of finite type.
Then $f$ satisfies the valuative criterion of separateness if and only if $f$ is separated.
\end{prop}
\begin{proof}
Use the valuative criterion of separateness for schemes, \cite[Proposition 5.13]{CHWW}, and Remark \ref{msch.19}.
\end{proof}

\begin{df}
\label{msch.21}
A monoid scheme $X$ is \emph{toric} if it is connected, separated, normal, torsion free, and of finite type.
\end{df}

\begin{rmk}
\label{msch.8}
For every fan $\Sigma$,
one can naturally associate a monoid scheme by \cite[Theorem 4.4]{CHWW},
which we often denote by the same notation $\Sigma$ when no confusion seems likely to arise.
A monoid scheme is toric if and only if it is associated with a fan by \cite[Theorem 4.4]{CHWW}.
\end{rmk}

\begin{df}
\label{msch.22}
A toric monoid scheme is \emph{smooth} if it is associated with a smooth fan. 
Let $\Sm/\F_1$ denote the category of smooth toric monoid schemes.
\end{df}

\begin{df}
\label{msch.16}
Let $M$ be a toric monoid, i.e., $M$ is fine saturated and $M^\gp$ is torsion free.
We use the notation $\Spec(\F_1[M])$ for the fan in the dual lattice $(M^\gp)^\vee$ whose single maximal cone is the dual monoid $M^\vee$ because its associated monoid scheme can be identified with $\Spec(\F_1[M])$ in the sense of Definition \ref{msch.3}.
\end{df}

\begin{df}
\label{msch.17}
Let $\A^1$ be the fan $\Spec(\F_1[\N])$.
Let $\G_m$ be the fan $\Spec(\F_1[\Z])$.
For $n\in \N$,
let $\P^n$ be the fan in $\Z^n$ whose maximal cones are $\Cone(e_1,\ldots,e_n)$ and
\[
\Cone(e_1,\ldots,e_{i-1},e_{i+1},\ldots,e_n,-e_1-\cdots-e_n)
\]
for integers $1\leq i\leq n$,
where $e_1,\ldots,e_n$ are the standard coordinates in $\Z^n$.

Following the convention in Remark \ref{msch.8},
we use the same notations $\A^1$, $\G_m$, and $\P^n$ for the associated toric monoid schemes.
\end{df}

\begin{rmk}
For a ring $R$ and integer $n$,
the schemes $\A_{R}^1$, $\G_{m,R}$, and $\P_{R}^n$ obtained by Construction \ref{msch.6} and Definition \ref{msch.17} are the usual ones.
\end{rmk}

\begin{df}
\label{msch.18}
Let $\Sigma$ be a fan in a lattice $N$,
and let $\sigma$ be a cone of $\Sigma$.
Consider the projection $\ol{(-)}\colon N\to N/N_\sigma$,
where $N_\sigma$ denotes the sublattice of $N$ generated by $\sigma$.
Let $V(\sigma)$ (denoted $\mathrm{Star}(\sigma)$ in \cite[(3.2.8)]{CLStoric}) be the fan in $N/N_\sigma$ consisting of $\ol{\tau}$ for all cones $\tau$ of $\Sigma$ containing $\sigma$.
\end{df}

\begin{const}
\label{msch.10}
Let $\Sigma$ be a fan with a cone $\tau$.
In \cite[p.\ 53]{Fulton:1436535},
a natural morphism of schemes $V(\tau)_\C\to \Sigma_{\C}$ was constructed.
Let us promote this to a natural morphism of monoid scheme $V(\tau)\to \Sigma$ as follows.

If $\sigma$ is a cone of $\Sigma$ containing $\tau$,
then we have the map
\[
\F_1[\sigma^\vee]\to \F_1[\sigma^\vee\cap \tau^\bot]
\]
sending $x$ to $x$ if $x\in \sigma^\vee\cap \tau^\bot$ and to $\infty$ otherwise,
where $\sigma^\vee$ denotes the dual cone of $\sigma$,
and $\tau^\bot$ denotes set of $y$ in the dual lattice of $\Sigma$ such that $\langle x,y\rangle =0$ for all $x\in \tau$.
If $\sigma'$ is a cone of $\Sigma$ containing $\sigma$,
then the induced square
\[
\begin{tikzcd}
\F_1[\sigma'^\vee]\ar[d]\ar[r]&
\F_1[\sigma'^\vee\cap \tau^\bot]\ar[d]
\\
\F_1[\sigma^\vee]\ar[r]&
\F_1[\sigma^\vee\cap \tau^\bot]
\end{tikzcd}
\]
commutes.
Hence we can glue these maps to construct a natural morphism $V(\tau)\to \Sigma$.

Observe that the morphism of monoid schemes $V(\tau)\to \Sigma$ is \emph{not} a morphism in the category of fans and hence not a morphism in the category of monoschemes in the sense of \cite[Definition II.1.2.3]{Ogu} even though $V(\tau)$ and $\Sigma$ are fans.
Working with monoid schemes has this advantage.
\end{const}

\begin{df}
Let $X$ be a monoid scheme.
A \emph{rational point of $X$} is a morphism of monoid schemes $\Spec(\F_1)\to X$.
\end{df}

\begin{exm}
\label{msch.15}
There are exactly two maps of pointed monoids $\F_1[\N] \to \F_1$ given by $n\mapsto 0$ and $n\mapsto \infty$ for $n\in \N^+$.
Hence $\A^1$ has exactly two rational points $0$ and $1$.

There is exactly one map of pointed monoids $\F_1[\Z] \to \F_1$ given by $n\mapsto 0$ for $n\in \Z$.
Hence $\G_m$ has exactly one rational point $1$.
Furthermore,
the obvious open immersion $\G_m\to \A^1$ identifies $1$ in $\G_m$ with $1$ in $\A^1$.

Hence $\P^1$ has exactly three rational points $0$, $1$, and $\infty$.
\end{exm}

\begin{df}
\label{msch.12}
A morphism of monoid schemes $f\colon X\to S$ is \emph{birational} if there exists a dense open subscheme $U$ of $S$ such that $f^{-1}(U)$ is dense in $X$ and the induced morphism $f^{-1}(U)\to U$ is an isomorphism.
\end{df}

\begin{df}
A morphism of fans $f\colon \Delta\to \Sigma$ is a \emph{partial subdivision} if $f$ induces an isomorphism between the lattices of $\Delta$ and $\Sigma$.
A partial subdivision $f$ is a \emph{subdivision} if $f$ sends the support $\lvert \Delta \rvert$ onto the support $\lvert \Sigma \rvert$.
\end{df}

\begin{exm}
\label{msch.13}
A subdivision of fans $\Delta\to \Sigma$ is the same as a proper birational morphism of toric monoid schemes under the identification of fans into toric monoid schemes,
see \cite[Corollary 10.2]{CHWW}.
Similarly, a partial subdivision of fans corresponds to a birational morphism of toric monoid schemes by \cite[Theorem 4.4(2)]{CHWW}.
\end{exm}

\begin{df}
\label{image.1}
Let $f\colon X\to S$ be a morphism of monoid schemes.
The \emph{scheme theoretic image of $f$} is the closed subscheme $Z$ of $S$ factoring $f$ that is initial among such closed subschemes.
If $f$ is an open immersion,
then $Z$ is also called the \emph{closure of $X$ in $S$}.

For the existence of the scheme theoretic image,
we refer to \cite[Proposition 9.3]{CHWW},
which also provides the local description as follows:
Suppose $S=\Spec(A)$ for some pointed monoid $A$.
Then the scheme theoretic image of $f$ is identified with the closed subscheme $\Spec(C)$ of $S$,
where $C$ is the image of the induced map $A\to \Gamma(X,\cO_X)$.
\end{df}

\begin{prop}
\label{image.5}
Let $X$ be a cancellative monoid scheme,
and let $U$ be its dense open subscheme.
Then the closure of $U$ in $X$ is $X$.
\end{prop}
\begin{proof}
We can work Zariski locally on $X$,
so we may assume that $X=\Spec(A)$ for some cancellative pointed monoid $A$.
We can also assume $U=\Spec(A_s)$ for some $s\in A-\{\infty\}$.
To conclude,
observe that the image of $A\to A_s$ is $A$ since $A$ is cancellative.
\end{proof}

\begin{lem}
\label{image.3}
Let $f\colon X\to S$ be a morphism of monoid schemes of finite type,
and let $Z$ be the scheme theoretic image of $f$.
If $X$ is toric and $S$ is separated,
Then $Z$ is cancellative, and its normalization $Z^\nor$ is toric.
\end{lem}
\begin{proof}
Let $Z'$ be a connected component of $Z$ such that $f(X)\cap Z'\neq \emptyset$.
Then $f^{-1}(Z')$ is closed and open in $X$.
Since $X$ is connected,
we have $f^{-1}(Z')=X$.
Hence $f$ factors through $Z'$,
so we have $Z=Z'$, i.e., $Z$ is connected.
Since $S$ is separated and $Z$ is a closed subscheme of $S$,
$Z$ is separated.

It remains to show that $Z$ is cancellative, torsion free, and of finite type since this implies that $Z^\nor$ is torsion free and of finite type and $Z^\nor$ is already normal.
This question is Zariski local on $S$,
so we may assume $S=\Spec(A)$ for some finitely generated pointed monoid $A$.
Then we have $Z=\Spec(C)$,
where $C$ is the image of the induced map $A\to \Gamma(X,\cO_X)$.
We can regard $X$ as a fan in a lattice $N$.
Since $\Gamma(X,\cO_X)$ is a submonoid of $\Gamma(X,\F_1[N])\cong \F_1[N]$,
$\Gamma(X,\cO_X)$ is cancellative and torsion free.
Hence $C$ is cancellative and torsion free.
To conclude,
observe that $C$ is finitely generated since $A$ is finitely generated.
\end{proof}

\begin{lem}
\label{image.4}
Let $j\colon U\to X$ be an open immersion of monoid schemes,
let $i\colon W\to U$ be a closed immersion of monoid schemes,
and let $Z$ be the scheme theoretic image of $ji$.
Then the induced morphism $W\to Z$ is an open immersion.
\end{lem}
\begin{proof}
The question is Zariski local on $U$ and $X$,
so we may assume that $X=\Spec(A)$, $U=\Spec(A_s)$, and $W=\Spec(C)$ for some pointed monoids $A$ and $C$ and element $s\in A$.
Then we have $Z=\Spec(B)$,
where $B$ is the image of the induced map $A\to C$.
The induced map $B\to C$ is injective.
Let $t$ be the image of $s$ in $B$.
Since $s$ is a unit in $A_s$,
$t$ is a unit in $C$.
It follows that the induced map $B_t\to C$ is injective too.
Since the map $A_s\to C$ is surjective,
the map $B_t\to C$ is surjective too.
Hence the map $B\to C$ is a localization,
i.e.,
$W\to Z$ is an open immersion.
\end{proof}

\begin{prop}
\label{image.6}
Let $f,g\colon X\rightrightarrows S$ be morphisms of monoid schemes.
Assume that $S$ is separated and $X$ is cancellative.
If $j\colon U\to X$ is a dense open immersion of monoid schemes such that $fj=gj$,
then $f=g$.
\end{prop}
\begin{proof}
Form the cartesian square
\[
\begin{tikzcd}
Z\ar[r,"i"]\ar[d]&
X\ar[d,"{(f,g)}"]
\\
S\ar[r,"\Delta"]&
S\times S,
\end{tikzcd}
\]
where $\Delta$ is the diagonal morphism.
Since $S$ is separated,
$\Delta$ is a closed immersion.
Hence $i$ is a closed immersion too by Proposition \ref{msch.20}.
Since $fj=gj$, $j$ factors through $i$.
Proposition \ref{image.5} implies that $i$ is an isomorphism.
Hence we have $f=g$.
\end{proof}

\begin{lem}
\label{image.8}
Let $\Delta\to \Sigma$ be a partial subdivision of $n$-dimensional fans with an integer $n\geq 1$,
and let $a$ be a ray of $\Sigma$.
Consider the open immersion $U:=\G_m^{n-1}\to V(a)$ from the maximal torus of $V(a)$.
If there exists a ray $b\in \Delta$ mapping to $a$,
then the scheme theoretic image $Z$ of $U\times_\Sigma \Delta$ in $\Delta$ is isomorphic to $V(b)$.
If not,
then we have $U\times_\Sigma \Delta=\emptyset$.
\end{lem}
\begin{proof}
We can work locally on $\Sigma$ and $\Delta$.
Hence we may assume that $\Delta\to \Sigma$ is given by $\Spec(\F_1[Q])\to \Spec(\F_1[P])$ for some map of toric monoids $P\to Q$.

Assume that there exists no ray $b\in \Delta$ mapping to $a$.
The composite $U\to \Sigma$ factors through the subfan $\Sigma'\cong \A^1\times \G_m^{n-1}$ with a single largest cone $a$.
The fiber product $\Sigma'\times_{\Sigma}\Delta$ is the torus $\G_m^n$.
We have $U\times_\Sigma \Delta=\emptyset$ from $(\{0\}\times \G_m^{n-1})\times_{\A^1\times \G_m^{n-1}}\G_m^n=\emptyset$.

Assume that there exists a ray $b\in \Delta$ mapping to $a$.
Then we have a commutative square
\[
\begin{tikzcd}
V(b)\ar[d]\ar[r]&
\Delta\ar[d]
\\
V(a)\ar[r]&
\Sigma.
\end{tikzcd}
\]
Observe that $V(b)\to V(a)$ is a partial subdivision and $U\times_\Sigma \Delta\cong U$.
Consider the face $G:=Q\cap b^\bot$ of $Q$ so that we have $V(b)=\Spec(\F_1[G])$.
The morphism $V(b)\to \Delta$ is induced by the map $\F_1[Q]\to \F_1[G]$ sending $x$ to $x$ if $x\in Q$ and to $\infty$ otherwise.
Let $C$ be the image of the composite map
\[
\F_1[Q]\to \F_1[G]\to \F_1[G^\gp].
\]
Then $\Spec(C)$ is isomorphic to $Z$.
To conclude, observe that we have $C\cong \F_1[G]$.
\end{proof}

\section{Log monoid schemes}
\label{logmonoid}

In this section,
we introduce the notion of log monoid schemes imitating \cite{Ogu}.
We also explain basic properties of log monoid schemes, which are direct analogs of the corresponding properties of log schemes.

\begin{df}
\label{msch.1}
Recall from \cite[Definition II.1.1.1]{Ogu} that a \emph{monoidal space} is a topological space $X$ equipped with a sheaf of monoids $\cM_X$ on $X$.
A \emph{morphism of monoidal spaces}  $f\colon X\to S$ is a map of topological spaces $f\colon X\to S$ equipped with a morphism of sheaves of pointed monoids $f^\flat\colon  \cM_S\to f_*\cM_X$ such that the induced map of the stalks $f^\flat_x \colon \cM_{S,f(x)}\to \cM_{X,x}$ is local.
\end{df}

\begin{df}
Let $X$ be a monoid scheme.
A \emph{prelog structure on $X$} is a morphism of sheaves of monoids $\alpha\colon \cP\to \cO_X$.
It is a \emph{log structure on $X$} if $\alpha^{-1}(\cO_X^*)\to \cO_X^*$ is an isomorphism.
By \cite[Proposition II.1.1.5]{Ogu},
the inclusion functor from the category of prelog structures on $X$ to the category of log structures on $X$ admits a left adjoint
\[
(\cP\to \cO_X)
\mapsto
(\cP^{\log}\to \cO_X)
\]
for a prelog structure $\cP$ on $X$.
The \emph{log structure on $X$ associated with $\cP$} is $\cP^{\log}$.

The \emph{trivial log structure on $X$} is the inclusion $\cO_X^*\to \cO_X$.
\end{df}

\begin{df}
A \emph{log monoid scheme $X$} is a monoid scheme $\ul{X}$ equipped with a log structure $\cM_X\to \cO_X:=\cO_{\ul{X}}$.
Note that we can regard $X$ as a monoidal space (resp.\ pointed monoidal space) if we forget the structure of $\cO_X$ (resp.\ $\cM_X$).

A \emph{morphism of log monoid schemes $f\colon X\to S$} is a morphism of monoid schemes equipped with a morphism of sheaves of monoids $f^\flat \colon \cM_S\to f_* \cM_X$ such that the induced square
\[
\begin{tikzcd}
\cM_S\ar[d]\ar[r,"f^\flat"]&
f_*\cM_X\ar[d]
\\
\cO_S\ar[r,"f^\sharp"]&
f_*\cO_X
\end{tikzcd}
\]
commutes.
Note that this condition automatically implies that $f^\flat_x\colon \cM_{S,f(x)}\to \cM_{X,x}$ is local for every point $x\in X$.
Hence $f$ induces a morphism of monoidal spaces.

A monoid scheme is regarded as a log monoid scheme with the trivial log structure.
\end{df}

\begin{df}
We refer to \cite[Definition II.2.1.1]{Ogu} for the notion of charts for monoidal spaces,
which we can use for log monoid schemes.
A log monoid scheme $X$ is \emph{coherent} (resp.\ fine, resp.\ fs) if locally on $X$, there exists a chart $P$ on $X$ such that $P$ is finitely generated (resp.\ fine, resp.\ fs).
\end{df}

\begin{prop}
The category of log monoid schemes admits fiber products.
\end{prop}
\begin{proof}
Let $X,Y\to S$ be morphisms of log monoid schemes.
As in \cite[Proof of Proposition III.2.1.2]{Ogu},
the fiber product has the following description:
The underlying monoid scheme $\ul{X}\times_{\ul{S}}\ul{Y}$.
The log structure is the coproduct of the log structures pulled back from the log structures on $X$, $Y$, and $S$.
\end{proof}

\begin{prop}
The inclusion functor from the category of fine (resp.\ fs) log monoid schemes to coherent log monoid schemes admits a right adjoint $X\mapsto X^\mathrm{int}$ (resp.\ $X\mapsto X^\sat$).
Furthermore,
the category of fine (resp.\ fs) log monoid schemes admits fiber products,
which is $(-)^\mathrm{int}$ (resp.\ $(-)^\sat$) of the fiber product in the category of log monoid schemes.
\end{prop}
\begin{proof}
Argue as in \cite[Proposition II.2.1.5, Corollary II.2.1.6]{Ogu}.
\end{proof}

\begin{rmk}
Let $X,Y\to S$ be morphisms of fs monoid schemes.
Then the fiber products $X\times_S^\mathrm{log} Y$, $X\times_S^\mathrm{int} Y$, and $X\times_S Y$ in the categories of log monoid schemes, fine log monoid schemes, and fs log monoid schemes respectively can be non-isomorphic and hence have to be distinguished.
\end{rmk}

\begin{df}
A morphism of log monoid schemes $f\colon X\to S$ is \emph{strict} if the induced morphism of sheaves of monoids $(f^{-1}\cM_S)^{\log}\to \cM_X$ is an isomorphism.
\end{df}

\begin{prop}
A morphism of log monoid schemes $f\colon X\to S$ is strict if and only if the induced morphism $X\to \ul{X}\times_{\ul{S}}S$ is an isomorphism.
\end{prop}
\begin{proof}
Argue as in \cite[Proposition III.1.2.5]{Ogu}.
\end{proof}

\begin{df}
A \emph{log monoid $(A,M)$} is a pair of a pointed monoid $A$ and a monoid $M$ equipped with a map of monoids $M\to A$.
A \emph{map of log monoids} is an obvious commutative square.

Let $\Spec(A,M)$ be the log monoid scheme whose underlying scheme is $\Spec(A)$ and log structure is associated with the constant prelog structure $M$.
\end{df}

\begin{rmk}
Let $X$ be a log monoid scheme.
Observe that $X$ is coherent (resp.\ fine, resp.\ fs) if and only if $X$ is locally isomorphic to $\Spec(A,M)$ for some log monoid $(A,M)$ such that $\ol{M}:=M/M^*$ is finitely generated (resp.\ fine, resp.\ fs).
\end{rmk}

\begin{df}
For a monoid $P$,
we set $\A_P:=\Spec(\F_1[P],P)$.
\end{df}

\begin{exm}
Let $(A,M)$ be a log monoid,
and let $A\to B$ be a map of pointed monoids.
Then the induced morphism of log monoid schemes $\Spec(B,M)\to \Spec(A,M)$ is strict.
\end{exm}

\begin{prop}
\label{logmonoid.6}
Let $f\colon X\to S$ be a morphism of log monoid schemes.
Assume that $X$ has the trivial log structure.
Then for every point $x\in X$,
the log structure on $S$ at $f(x)$ is trivial.
\end{prop}
\begin{proof}
Recall that the map $f^\flat_x\colon \cM_{S,f(x)}\to \cM_{X,x}$ is local,
which means
\[
(f^\flat_x)^{-1}(\cM_{X,x}^*)\cong \cM_{S,f(x)}^*.
\]
Since $X$ has the trivial log structure,
we have $\cM_{X,x}\cong \cM_{X,x}^*$.
Hence we have $\cM_{S,f(x)}\cong \cM_{S,f(x)}^*$, which means that the log structure on $S$ at $f(x)$ is trivial.
\end{proof}

\begin{df}
\label{logmonoid.8}
Let $\lSm/\F_1$ be the full subcategory of the category of log monoid schemes $X$ satisfying the following condition:
The underlying monoid scheme $\ul{X}$ is toric, and Zariski locally on $\ul{X}$, $X$ is isomorphic to $\A_P\times \A^m$ for some toric monoid $P$ and integer $m\geq 0$.

Let $\SmlSm/\F_1$ be the full subcategory of $\lSm/\F_1$ spanned by those $X$ such that $\ul{X}$ is a smooth toric monoid scheme.
Observe that Zariski locally on $\ul{X}$, $X$ is isomorphic to $\A_\N^r \times \A^s \times \G_m^t$ for some integers $r,s,t\geq 0$.
\end{df}

\begin{df}
\label{logmonoid.19}
A \emph{Zariski distinguished square in $\lSm/\F_1$} is a cartesian square
\begin{equation}
\begin{tikzcd}
W\ar[d]\ar[r]&
V\ar[d]
\\
U\ar[r]&
X
\end{tikzcd}
\end{equation}
in $\lSm/\F_1$ such that $U\to X$ and $V\to X$ are jointly surjective open immersions.
The \emph{Zariski cd-structure on $\lSm/\F_1$} is the set of Zariski distinguished squares in $\lSm/\F_1$.
The associated topology is the \emph{Zariski topology}.
\end{df}

\begin{rmk}
\label{logmonoid.17}
The Zariski cd-structure on $\lSm/\F_1$ satisfies the claims analogous to Remark \ref{logSH.11}.
In particular,
the Zariski cd-structure on $\lSm/\F_1$ is bounded, complete, and regular.
\end{rmk}

\begin{const}
\label{logmonoid.7}
Let $S$ be a smooth toric monoid scheme associated with a smooth fan $\Sigma$,
and let $\alpha_1,\ldots,\alpha_n$ be different rays of $\Sigma$.
We define the log monoid scheme $X:=(S,Z)\in \SmlSm/\F_1$ with the notation $Z:=V(\alpha_1)+\cdots +V(\alpha_n)$ as follows.

Let $U:=X-\partial X$ be the open subscheme of $S$ associated with the subfan of $\Sigma$ consisting of the cones $\sigma$ such that $\sigma$ does not contain $\alpha_i$ for every $i$.
Assume $S=\A^r\times \G_m^t$ and $\alpha_i=e_i$ for $1\leq i\leq d$ with integers $d\leq r$ and $t\geq 0$,
where $e_1,\ldots,e_{r+t}$ are the standard coordinates in $\Z^{r+t}$.
Then we have $U\cong \A_\N^d \times \A^{r-d}\times \G_m^t$.
Hence with $P:=\N^r\times \Z^t$ and $F:=\N^{r-d}\times 0 \times \Z^t$, we have $X\cong \Spec(\F_1[P])$ and $U\cong \Spec(\F_1[P_F])$.
In this case, we define $(S,Z):=\Spec(F\to \F_1[P])$.
Glue this local construction to define $(S,Z)$ for general $X$.

Observe that for $x\in X$, $X$ has the trivial log structure at $x$ if and only if $x\in X-\partial X$.
\end{const}

\begin{df}
In the fan $\P^n$ with an integer $n\geq 1$
consider the ray $\alpha:=(-1,\ldots,-1)$, and we regard $V(\alpha)$ as the hyperplane $\P^{n-1}$ of $\P^n$ at $\infty$.
Then we have $(\P^n,\P^{n-1})\in \SmlSm/\F_1$.
We set $\square:=(\P^1,\P^0)$.

In the fan $\P^1$,
consider the rays $e_1$ and $-e_1$,
and we regard $V(e_1)$ as the point $0$ and $V(-e_1)$ as the point $\infty$.
Then we have $\Gmlog:=(\P^1,0+\infty)\in \SmlSm/\F_1$.
\end{df}

\begin{df}
\label{logmonoid.26}
For an $n$-dimensional fan $\Sigma$,
let $\T_\Sigma$ be the log monoid scheme in $\lSm/\F_1$ defined as follows.
If $\Sigma=\Spec(\F_1[P])$ for some toric monoid $P$,
then $\T_\Sigma:=\A_P$.
For general $P$, glue this local construction to define $\T_\Sigma$.
Observe that $\Sigma$ is a smooth fan if and only if $\T_\Sigma\in \SmlSm/\F_1$.
\end{df}

\begin{prop}
\label{logmonoid.5}
Let $S,X\in \lSm/\F_1$.
Then we have
\[
\Hom(X,S)
\cong
\{\ul{f}\colon \ul{X}\to \ul{S}:\ul{f}(X-\partial X)\subset S-\partial S\}.
\]
\end{prop}
\begin{proof}
Let $f\colon X\to S$ be a morphism in $\lSm/\F_1$.
Then its underlying morphism of schemes is $\ul{f}\colon \ul{X}\to \ul{S}$,
and this sends $X-\partial X$ into $S-\partial S$ by Proposition \ref{logmonoid.6}.

Conversely,
let $\ul{f}\colon \ul{X}\to \ul{S}$ be a morphism of monoid schemes such that $\ul{f}(X-\partial X)\subset S-\partial S$.
Assume that $\ul{f}$ is associated with a map of monoids $\F_1[P]\to \F_1[Q]$ for some toric monoids $P$ and $Q$ and $X-\partial X$ and $S-\partial S$ are associated with $\F_1[P_F]$ and $\F_1[Q_G]$ for some faces $F$ and $G$ of $P$ and $Q$.
Then the condition $f(X-\partial X)\subset S-\partial S$ implies that the composite map $\F_1[P]\to \F_1[Q]\to \F_1[Q_G]$ factors through the localization $\F_1[P]\to \F_1[P_F]$.
Hence $\F_1[P]\to \F_1[Q]$ sends $F$ into $G$,
so we have a map of log monoids $(\F_1[P],F)\to (\F_1[Q],G)$.
For general $\ul{f}$,
glue this local construction to obtain a morphism $X\to S$.

To conclude,
observe that the above two constructions are inverses to each other.
\end{proof}

\begin{exm}
\label{logmonoid.10}
Let $X$ be a monoid scheme,
and let $P$ be a monoid.
As in \cite[Proposition III.1.2.4]{Ogu},
there is a natural isomorphism
\[
\Hom(X,\A_P)
\cong
\Hom(P,\Gamma(X,\cM_X)).
\]
\end{exm}

\begin{df}
\label{logmonoid.18}
Let $f\colon X\to S$ be a morphism in $\lSm/\F_1$ such that $\ul{f}\colon \ul{X}\to \ul{S}$ is proper birational.
We say that $f$ is a \emph{dividing cover} (resp.\ \emph{partial dividing cover}) if Zariski locally on $S$, it is of the form $\T_u\times \id \colon \T_\Sigma\times \A^m \to \A_P\times \A^m$ for some toric monoid $P$,
subdivision (resp.\ partial subdivision) of fans $u\colon \Sigma\to \Spec(\F_1[P])$, and integer $m\geq 0$.
\end{df}

\begin{exm}
\label{logmonoid.9}
For a subdivision (resp.\ partial subdivision) of fans $\Delta\to \Sigma$,
the induced morphism $\T_\Delta\to \T_\Sigma$ is a dividing (resp.\ partial dividing) cover.
\end{exm}

\begin{const}
\label{logmonoid.16}
Let $S\in \lSm/\F_1$.
Consider the fans $\Sigma$ and $\Sigma'$ whose associated monoid schemes are $\ul{S}$ and $S-\partial S$.
Let $\Sigma''$ be the subfan of $\Sigma$ consisting of those cones $\sigma$ such that all the rays of $\sigma$ are not in $\Sigma'$.
We set $S^\sharp:=\T_{\Sigma''}$,
which we regard as an open subscheme of $S$.

In the local case $S=\A_P\times \A^m$ with a toric monoid $P$ and integer $m\geq 0$,
$\Sigma$ is $\Spec(\F_1[P])\times \A^m$,
$\Sigma'$ is $\Spec(\F_1[P^\gp])\times \A^m$,
and $\Sigma''$ is $\Spec(\F_1[P])\times \G_m^m$.
Hence the categories of dividing covers of $S$ and $S^\sharp$ are equivalent.

For general $S$,
we can glue the local constructions to have the same conclusion.
\end{const}

\begin{prop}
\label{logmonoid.11}
Let $f\colon X\to S$ be a dividing cover in $\lSm/\F_1$.
If $S=\T_\Sigma\times \A^m$ for some fan $\Sigma$ and integer $m\geq 0$,
then $f$ can be identified with $\T_u\times \id\colon \T_\Delta\times \A^m\to \T_\Sigma \times \A^m$ for some subdivision $u\colon \Delta\to \Sigma$.
\end{prop}
\begin{proof}
If $\Sigma=\Spec(\F_1[P])$ for some toric monoid $P$,
then the claim is due to the definition of dividing covers since every Zariski covering of $\Spec(\F_1[P])$ contains $\Spec(\F_1[P])$ itself.
For general $\Sigma$,
glue the local constructions.
\end{proof}

\begin{prop}
\label{logmonoid.12}
The class of dividing covers in $\lSm/\F_1$ is closed under compositions.
\end{prop}
\begin{proof}
Let $f\colon X\to S$ and $g\colon Y\to X$ be dividing covers in $\lSm/\F_1$.
To show that $fg$ is a dividing cover,
we can work Zariski locally on $S$.
Hence we may assume that $f=\T_u\times \id\colon \T_\Sigma \times \A^m \to \A_P\times \A^m$ for some subdivision of fans $u\colon \Sigma\to \Spec(\F_1[P])$ with a toric monoid $P$ and integer $m\geq 0$.
By Proposition \ref{logmonoid.11},
$g$ is identified with $\T_v\times \id\colon \T_\Delta\times \A^m \to \T_\Sigma \times \A^m$ for some subdivision of fans $v\colon \Delta\to \Sigma$.
The composite morphism $gf$ is identified with $\T_{uv}\times \id$, so it is a dividing cover.
\end{proof}

\begin{prop}
\label{logmonoid.13}
Let $f\colon X\to S$ be a dividing cover in $\lSm/\F_1$,
and let $g\colon S'\to S$ be a morphism in $\lSm/\F_1$.
Then the pullback $X'\to S'$ of $f$ in the category of fs log monoid schemes is a dividing cover in $\lSm/\F_1$.
Furthermore,
$\ul{X'}$ is the normalization of the closure of $X'-\partial X'$ in $\ul{X}\times_{\ul{S}}\ul{S'}$.
\end{prop}
\begin{proof}
We can work Zariski locally on $S$,
so we may assume that $f$ is of the form $T_u\times \id\colon \T_\Sigma \times \A^m\to \A_P \times \A^m$ for some subdivision of fans $u\colon \Sigma\to \Spec(\F_1[P])$ and integer $m\geq 0$.
Since $T_u\times \id$ is a pullback of $T_u$,
we reduce to the case of $m=0$.

We can work Zariski locally on $S'$ too, so we may assume $S'=\A_{P'}\times \A^{m'}$ for some toric monoid $P'$ and integer $m'\geq 0$.
We have $\Gamma(S',\cM_{S'})\cong \Gamma(\A_{P'},\cM_{\A_{P'}})\cong P'$,
so the morphism $S'\to S$ factors through the projection $S'\to \A_{P'}$ by Example \ref{logmonoid.10}.
Hence we also reduce to the case of $m'=0$.
In this case,
$g=\A_\theta$ for some map of monoids $\theta\colon P\to P'$.

Let $\{\Spec(\F_1[Q_i])\}_{i\in I}$ be the cones of $\Sigma$,
and let $Q_i'$ be the coproduct $Q\oplus_P^\mathrm{mon} P'$ in the category of monoids,
and consider its saturation $(Q_i')^\sat$.
Then the cones $\{\Spec(\F_1[(Q_i')^\sat])\}_{i\in I}$ form a fan $\Sigma'$,
and the morphism $X'\to X$ can be identified with $\T_{\Sigma'}\to \T_{\Sigma}$ with a subdivision $\Sigma'\to \Sigma$.
Hence $X'\to X$ is a dividing cover.

For the last claim,
consider $X_i:=\Spec(\F_1[Q_i])$ and $X_i'=\Spec(\F_1[(Q_i')^\sat])$ for $i\in I$.
Then we have $\ul{X_i}\times_{\ul{S}}\ul{S'}\cong \Spec(\F_1[Q_i'])$ and $X_i'-\partial X_i'\cong \Spec(\F_1[Q^\gp\oplus_{P^\gp}P'^\gp])$.
Also, the image of the induced map $\F_1[Q_i']\to \F_1[Q^\gp\oplus_{P^\gp}P'^\gp]$ is precisely $\F_1[(Q_i')^\mathrm{int}]$.
Hence the integral closure of $X'-\partial X'$ in $\ul{X}\times_{\ul{S}}\ul{S'}$ is $\Spec(\F_1[(Q_i')^\mathrm{int}])$ using the local description of the scheme theoretic image in Definition \ref{image.1}.
To conclude,
observe that the normalization of $\Spec(\F_1[(Q_i')^\mathrm{int}])$ is $\Spec(\F_1[(Q_i')^\sat])$.
\end{proof}

\begin{prop}
\label{logmonoid.14}
A partial dividing cover $f\colon X\to S$ in $\lSm/\F_1$ is a monomorphism of fs log monoid schemes.
\end{prop}
\begin{proof}
We can work Zariski locally on $S$,
so we may assume that $f$ is given by $\T_u\times \id \colon \T_\Sigma\times \A^m\to \A_P\times \A^m$ for some toric monoid $P$, partial subdivision of fans $u\colon \Sigma\to \Spec(\F_1[P])$, and integer $m\geq 0$.
We further reduce to the case of $m=0$ since any pullback of a monomorphism is a monomorphism.

We need to show that the diagonal morphism $X\to X\times_S X$ is an isomorphism.
Let $\sigma$ and $\sigma'$ be cones of $\Sigma$,
and let $U_\sigma$ be the corresponding open subscheme of $\ul{X}$.
As in the proof of Proposition \ref{logmonoid.13},
$\ul{X\times_S X}$ is the normalization of the closure of $X-\partial X$ in $\ul{X}\times_{\ul{S}}\ul{X}$.
Using this,
we can show
\[
U_{\sigma\cap \sigma'}
\cong
U_\sigma \times_{S} U_{\sigma'}.
\]
Glue these local constructions to obtain $X\cong X\times_S X$.
\end{proof}

\begin{df}
For a log monoid $(A,M)$,
we define $\Spec(A,M)\otimes_{\F_1}\Z:=\Spec(A\otimes_{\F_1}\Z,M)$.
For a coherent log monoid scheme $X$,
glue the local constructions to define $X\otimes_{\F_1} \Z$.
\end{df}

\begin{const}
\label{logmonoid.20}
As in Construction \ref{msch.6},
we can unify the categories of coherent log schemes and coherent log monoid schemes into a single category whose hom sets are given as follows:
\begin{enumerate}
\item[(i)]
If $S$ and $X$ are both coherent log schemes or coherent log monoid schemes,
then $\Hom(X,S)$ is the original one.
\item[(ii)]
If $S$ is a coherent log scheme and $X$ is a coherent log monoid scheme,
then
\[
\Hom(X,S):=\emptyset.
\]
\item[(iii)]
If $S$ is a coherent log monoid scheme and $X$ is a coherent log scheme,
then
\[
\Hom(X,S):=\Hom(X,S\otimes_{\F_1}\Z).
\]
\end{enumerate}
\end{const}

\begin{df}
\label{logmonoid.21}
Let $X$ be an fs log scheme,
and let $X_0$ be an fs log monoid scheme.
A morphism $X\to X_0$ is \emph{strict} if the induce morphism $X\to \ul{X}\times_{\ul{X_0}}X_0$ is an isomorphism.
\end{df}

\begin{df}
\label{logmonoid.22}
Let $\cSm/\F_1$ be the category of pairs $X:=[\ul{X},X-\partial X]$ equipped with an open immersion of toric monoid schemes $X-\partial X\to \ul{X}$ such that $X-\partial X\in \Sm/\F_1$.

We also have the Zariski topology on $\cSm/\F_1$,
which consists of the families $\{U_i\to X\}_{i\in I}$ with finite $I$ such that $\{\ul{U_i}\to \ul{X}\}_{i\in I}$ is the Zariski covering of monoid schemes.

A morphism $X\to S$ in $\cSm/\F_1$ is an \emph{admissible blow-up} if $\ul{X}\to \ul{S}$ is proper and the induced morphism $X-\partial X\to S-\partial S$ is an isomorphism.
\end{df}

\begin{rmk}
\label{logmonoid.23}
We have the functor $\lSm/\F_1\to \cSm/\F_1$ sending $X$ to $[\ul{X},X-\partial X]$,
which is fully faithful by Proposition \ref{logmonoid.5}.
\end{rmk}

\begin{prop}
\label{logmonoid.24}
Let $\Sigma$ be a smooth fan,
and let $\Delta$ be its subfan.
Regard them as smooth toric monoid schemes.
Then we have $[\Sigma,\Delta]\in \SmlSm/\F_1$ if and only if for every cone $\sigma\in \Sigma$ not in $\Delta$,
there exists a ray of $\sigma$ not in $\Delta$.
\end{prop}
\begin{proof}
Assume that for every cone $\sigma\in \Sigma$ not in $\Delta$,
there exists a ray of $\sigma$ not in $\Delta$.
Let $\alpha_1,\ldots,\alpha_d$ be the set of rays of $\Sigma$ not in $\Delta$.
Then $\Delta$ is precisely the set of cones $\delta$ of $\Sigma$ such that all the rays of $\delta$ are different from $\alpha_1,\ldots,\alpha_d$.
By construction \ref{logmonoid.7},
we have $[\Sigma,\Delta]\in \SmlSm/\F_1$.

Conversely,
assume $[\Sigma,\Delta]\in \SmlSm/\F_1$.
We can work Zariski locally on $\Sigma$,
so we may assume $[\Sigma,\Delta]=\A_\N^r\times \A^s \times \G_m^t$ for some integers $r,s,t\geq 0$.
The desired condition holds for this case.
\end{proof}

\begin{prop}
\label{logmonoid.25}
Let $f\colon X\to S$ be a birational morphism in $\cSm/\F_1$ such that the induced morphism $f^{-1}(S-\partial S)\to (S-\partial S)$ is an open immersion.
If $\ul{X}\in \Sm/\F_1$ and $S\in \SmlSm/\F_1$,
then we have $X\in \SmlSm/\F_1$.
\end{prop}
\begin{proof}
We have $S=[\Sigma,\Delta]$ and $X=[\Sigma',\Delta']$ for a partial subdivision of smooth fans $\Sigma'\to \Sigma$ and subfans $\Delta$ and $\Delta'$ of $\Sigma$ and $\Sigma'$.
The assumption on $f$ implies that $\Delta'$ is a subfan of $\Delta$ and $\lvert \Delta'\rvert = \lvert \Sigma' \rvert \cap \lvert \Delta \rvert$.
Let $\sigma'$ be a cone of $\Sigma'$ not in $\Delta'$ such that every ray of $\sigma$ is in $\Delta'$.
By Proposition \ref{logmonoid.24},
we only need to show that this leads to a contradiction.

We have $\sigma'=\Cone(f_1,\ldots,f_d)$ for some rays $f_1,\ldots,f_d\in \Delta'$.
Let $\sigma$ be the smallest cone of $\Sigma$ containing $\sigma'$.
Then we have $\dim \sigma =\dim \sigma' = d$ and $f_1,\ldots,f_d\in \Sigma$,
so we have $\sigma'=\sigma$.
By Proposition \ref{logmonoid.24},
we have $\sigma'\in \Delta$ since $S\in \SmlSm/\F_1$ and hence $\sigma'\in \Delta'$ since $\lvert \Delta'\rvert = \lvert \Sigma' \rvert \cap \lvert \Delta \rvert$,
which is a contradiction.
\end{proof}

\section{Toric resolution of singularities}

In this section,
we collect several results on toric resolution of singularities that are needed in this paper.

\begin{prop}
\label{logmonoid.1}
Let $\Sigma$ be a fan.
Then there exists a subdivision of fans $\Sigma'\to \Sigma$ satisfying the following conditions:
\begin{enumerate}
\item[\textup{(i)}] $\Sigma'$ is smooth.
\item[\textup{(ii)}] If $\sigma$ is a smooth cone of $\Sigma$,
then $\sigma$ is a cone of $\Sigma'$,
i.e., $\sigma$ is not divided in $\Sigma'$.
\item[\textup{(iii)}] $\Sigma'$ is obtained by a finite sequence of star subdivisions.
\end{enumerate}
\end{prop}
\begin{proof}
See \cite[Theorem 11.1.9]{CLStoric}.
\end{proof}

\begin{prop}
\label{resolution.1}
Let $X$ be a toric monoid scheme.
Then there exists a proper birational morphism of monoid schemes $Y\to X$ such that $Y\in \Sm/\F_1$.
\end{prop}
\begin{proof}
This is an immediate consequence of Proposition \ref{logmonoid.1} due to the correspondence between toric monoid schemes and fans.
\end{proof}

\begin{df}
Let $\sigma$ and $\tau$ be cones in a lattice $N$.
We say that \emph{$\tau$ crosses $\sigma$} if $\sigma\cap \tau$ is not a face of $\sigma$.
Let $\Sigma$ be a fan in $N$.
We say that \emph{$\tau$ crosses $\Sigma$} if there exists $\sigma\in \Sigma$ such that $\tau$ crosses $\sigma$.
\end{df}

\begin{lem}
\label{resolution.5}
Let $\sigma$ and $\tau$ be $n$-dimensional cones in an $n$-dimensional lattice $N$ with an integer $n\geq 0$.
If $\tau$ crosses $\sigma$,
then there exists an $(n-1)$-dimensional face of $\tau$ crossing $\sigma$.
\end{lem}
\begin{proof}
Since $\sigma\cap \tau$ is not a face of $\sigma$,
there exists a face $\eta$ of $\sigma\cap \tau$ such that $\dim \eta\leq n-1$ and $\eta$ is not a face of $\sigma$.
Let $\eta'$ be an $(n-1)$-dimensional face of $\tau$ containing $\eta$.
Then $\eta'$ crosses $\sigma$.
\end{proof}

\begin{lem}
\label{logmonoid.3}
Let $\Sigma$ be a smooth fan in $\Z^n$ with the support $\N^n$,
where $n\in \N$.
Let $e_1,\ldots,e_n$ be the standard coordinates in $\Z^n$.
If $\eta:=\Cone(e_1,\ldots,e_r)$ with an integer $0\leq r\leq n$ and $\tau$ is an $(n-1)$-dimensional cone contained in $\N^n$ such that $\eta\cap \tau$ is a face of $\eta$,
then there is a subdivision $\Sigma'$ of $\Sigma$ obtained by a sequence of star subdivisions relative to $2$-dimensional cones such that $\eta\in \Sigma'$ and $\tau$ does not cross $\Sigma'$.
\end{lem}
\begin{proof}
Without loss of generality,
we may assume $\eta\cap \tau=\Cone(e_1,\ldots,e_s)$ with an integer $0\leq s\leq r$.
We proceed by induction on $n$.
The claim is obvious if $n=0$.
Assume $n\geq 1$.

We first treat the case where $\tau$ is the intersection of a hyperplane and $\N^n$,
i.e.,
\[
\tau=\{(x_1,\ldots,x_n)\in \N^n:a_1x_1+\cdots+a_nx_n=0\}
\]
for some $(a_1,\ldots,a_n)\in \Z^n-0$.
Observe that we have $a_1=\cdots=a_s=0$ and $a_{s+1},\ldots,a_r\neq 0$.
Since $\eta\cap \tau$ is a face of $\eta$,
$a_{s+1},\ldots,a_r$ have the same sign.
We may assume that they are positive.
Consider the function $\varphi_\tau\colon \Z^n\to \Z$ given by
\[
\varphi_\tau(x_1,\ldots,x_n):=a_1x_1+\cdots+a_nx_n.
\]
According to \cite[Proof of Lemma 2.3]{MR803344},
there exists a sequence of star subdivisions $\Sigma':=\Sigma_m\to \cdots \to \Sigma_0:=\Sigma$ such that $\Sigma'$ does not cross $\tau$ and $\Sigma_{i+1}$ is the star subdivision of $\Sigma_i$ relative to a $2$-dimensional cone $\delta_i$ for each $i$ satisfying the following condition:
One vertex $v$ of $\delta_i$ satisfies $\varphi_\tau(v)>0$ and the other vertex $v'$ of $\delta_i$ satisfies $\varphi_\tau(v')<0$.
Since $\varphi_\tau(e_1),\ldots,\varphi_\tau(e_r)\geq 0$,
the center of $\delta_i$ is not in $\eta$.
Hence we have $\eta\in \Sigma_i$ for each $i$ by induction,
so $\Sigma'$ is a desired one.

For general $\tau$,
consider
\[
\delta
:=
\{(x_1,\ldots,x_n)\in \N^n:-x_{s+1}-\cdots-x_r+mx_{r+1}+\cdots+mx_n=0\}
\]
with an integer $m\geq 0$.
If $f$ is a ray of $\tau$ different from $e_1,\ldots,e_s$,
then $f\notin \Cone(e_1,\ldots,e_r)$ since $\eta\cap \tau=\Cone(e_1,\ldots,e_s)$.
Hence there exists an integer $m>0$ such that $\varphi_\delta(f)>0$ for such a ray $f$.
By the special case above,
there exists a subdivision $\Delta$ of $\Sigma$ containing $\eta$ and obtained by a sequence of star subdivisions relative to $2$-dimensional cones such that if $\delta$ does not cross $\Delta$.
Then every cone of $\Delta$ is contained either in the region $\varphi_\delta\geq 0$ or $\varphi_\delta\leq 0$.
Let $\Delta'$ be the subfan of $\Delta$ consisting of the cones in the region $\varphi_\delta\geq 0$.
Then we have $\tau\subset \lvert \Delta'\rvert$.
Hence we can replace $\Sigma$ by $\Delta'$,
so we may assume $\eta\cap \tau=\eta$.

With this assumption,
let $\ol{\tau}$ be the intersection of $\N^n$ and the hyperplane generated by $\tau$.
We have $\ol{\tau}\cap \eta=\eta$ since $\tau\cap \eta=\eta$.
By the special case above,
there exists a subdivision $\Sigma'\to \Sigma$ obtained by a finite sequence of star subdivisions relative to $2$-dimensional cones such that $\eta\in \Sigma'$ and $\Sigma'$ does not cross $\ol{\tau}$.
Consider the subfan $\Gamma$ of $\Sigma'$ consisting of the cones contained in $\ol{\tau}$.
Using the induction hypothesis to $\Gamma$ and all the $(n-2)$-dimensional faces of $\tau$,
and after star subdividing $\Gamma$ and hence $\Sigma'$ too,
we may assume that these faces do not cross $\Gamma$.
Now, let $\sigma$ be a cone of $\Sigma'$.
Since $\ol{\tau}$ does not cross $\Sigma$,
$\sigma\cap \ol{\tau}$ is a cone of $\sigma$.
Together with Lemma \ref{resolution.5},
we see that $\tau$ does not cross $\sigma\cap \ol{\tau}$ and hence $\sigma$ too.
\end{proof}

The following result is a slight adjustment of 
\cite[Theorem 2.4]{MR803344} (see also \cite[pp.\ 39-40]{TOda}).

\begin{prop}
\label{logmonoid.2}
Let $\Sigma$ and $\Delta$ be fans with the same support $\lvert \Sigma\rvert = \lvert \Delta\rvert$ and a common cone $\eta\in \Sigma,\Delta$.
Assume that $\Sigma$ is smooth.
Then there exist subdivisions of fans $\Sigma'\to \Sigma,\Delta$ such that $\eta\in \Sigma'$ and $\Sigma'\to \Sigma$ is obtained by a finite sequence of star subdivisions relative to $2$-dimensional cones.
\end{prop}
\begin{proof}
Let $\sigma_1,\ldots,\sigma_m$ be the maximal cones of $\Sigma$.
Assume that the claim holds for any fan whose support agrees with $\sigma_i$ for some $i$.
Then by induction on $m$,
we obtain a desired $\Sigma'$.
Hence we reduce to the case where $\Sigma$ is a fan in $\Z^n$ with support $\N^n$.

By Lemma \ref{logmonoid.3}, there exists a subdivision $\Sigma'\to \Sigma$ obtained by a finite sequence of star subdivisions relative to $2$-dimensional cones such that $\eta\in \Sigma'$ and every $(n-1)$-dimensional cone of $\Delta$ does not cross $\Sigma'$.
Lemma \ref{resolution.5} implies that every $n$-dimensional cone of $\Delta$ does not cross $\Sigma'$.
It follows that $\Sigma'$ is a subdivision of $\Delta$.
\end{proof}

\begin{prop}
\label{resolution.3}
Let $X\in \cSm/\F_1$.
Then there exists an admissible blow-up $Y\to X$ with $Y\in \SmlSm/\F_1$.
\end{prop}
\begin{proof}
By Proposition \ref{logmonoid.1},
there exists an admissible blow-up $X'\to X$ in $\cSm/\F_1$ such that $\ul{X'}$ is smooth.
Replacing $X$ by $X'$,
we may assume that $\ul{X}$ is smooth.

Consider the smooth fans $\Sigma$ and $\Delta$ corresponding to $\ul{X}$ and $X-\partial X$.
Let $\alpha(\Sigma,\Delta)$ be the set of cones $\sigma$ of $\Sigma$ not in $\Delta$ such that every ray of $\sigma$ is in $\Delta$.
If $\alpha(\Sigma,\Delta)=\emptyset$, then we have $X\in \SmlSm/\F_1$ by Proposition \ref{logmonoid.24}.
This is the reason why we consider $\alpha(\Sigma,\Delta)$.

Now, let $\sigma$ be a cone of $\alpha(\Sigma,\Delta)$ having a maximal dimension among the cones of $\alpha(\Sigma,\Delta)$, say $d$.
Consider the star subdivision $\Sigma^*(\sigma)$.
Then the center of $\sigma$ is not in $\Delta$ and hence every $d$-dimensional cone containing that center is not in  $\alpha(\Sigma^*(\sigma),\Delta)$.
It follows that $\alpha(\Sigma^*(\sigma),\Delta)$ has one less $d$-dimensional cones than $\alpha(\Sigma^*(\sigma),\Delta)$.
Also, $\alpha(\Sigma^*(\sigma),\Delta)$ does not contain a cone whose dimension is $>d$.
By repeating this process,
we see that there exists a composition of star subdivisions $\Sigma'\to \Sigma$ such that $\alpha(\Sigma',\Delta)=\emptyset$,
Then Construction \ref{logmonoid.7} yields a desired $Y\to X$.
\end{proof}

\begin{prop}
\label{logmonoid.4}
Let $X\in \Sm/\F_1$.
Then there exists proper $Y\in \SmlSm/\F_1$ containing $X$ as an open subscheme.
\end{prop}
\begin{proof}
The monoid scheme $X$ is associated with a smooth fan $\Sigma$.
There exists a complete fan $\Delta$ containing $\Sigma$ as a subfan by \cite[p.\ 18]{TOda}.
Then the pair $(\Delta,\Sigma)$ yields $Y\in \cSm/\F_1$ such that $\ul{Y}$ is proper.
Proposition \ref{resolution.3} finishes the proof.
\end{proof}

\begin{prop}
\label{resolution.2}
Let $X\in \lSm/\F_1$.
Then there exists a dividing cover $Y\to X$ with $Y\in \SmlSm/\F_1$.
\end{prop}
\begin{proof}
Consider $X^\sharp\in \lSm/\F_1$ obtained by Construction \ref{logmonoid.16},
which also shows that the categories of dividing covers of $X$ and $X^\sharp$ are equivalent.
Let $\Sigma$ be the fan such that $X^\sharp\cong \T_\Sigma$.
By Proposition \ref{logmonoid.1},
there exists a subdivision $\Delta\to \Sigma$ such that $\Delta$ is smooth.
The induced morphism $\T_\Delta\to \T_\Sigma$ is a dividing cover,
and this corresponds to a dividing cover $Y\to X$ in $\lSm/\F_1$.
From the smoothness of $\Delta$,
we have $\ul{Y}\in \Sm/\F_1$ and hence $Y\in \SmlSm/\F_1$.
\end{proof}

\bibliography{bib}
\bibliographystyle{siam}
\end{document}